\pdfoutput=1 

\documentclass[reqno]{amsart}

\usepackage{amssymb,amsmath,amsthm,graphicx,xcolor,mathtools,mathrsfs,tabularx,bbm,tikz,url}
\usepackage{mathabx}\changenotsign
\usepackage{dsfont}
\usepackage[shortlabels]{enumitem}
\usepackage{lmodern}
\usepackage[babel]{microtype}
\usepackage[british]{babel}
\usepackage[utf8]{inputenc}
\allowdisplaybreaks
\usepackage{thm-restate}

\usepackage[backref, hypertexnames=false]{hyperref} 
\hypersetup{
	colorlinks,
	linkcolor={red!60!black},
	citecolor={green!60!black},
	urlcolor={blue!60!black}
}

\def\ssign{\textsection\nobreak\hspace{1pt plus 0.3pt}}
\makeatletter
\let\origsection=\section 
\def\mysection{\@mystartsection{section}{1}\z@{.7\linespacing\@plus\linespacing}{.5\linespacing}{\normalfont\scshape\centering\ssign}}
\def\section{\@ifstar{\origsection*}{\mysection}}
\def\appendix{\par\c@section\z@ \c@subsection\z@
	\let\sectionname\appendixname
	\let\section=\origsection
	\def\thesection{\@Alph\c@section}}
\def\@mystartsection#1#2#3#4#5#6{\if@noskipsec \leavevmode \fi
	\par \@tempskipa #4\relax
	\@afterindenttrue
	\ifdim \@tempskipa <\z@ \@tempskipa -\@tempskipa \@afterindentfalse\fi
	\if@nobreak \everypar{}\else
	\addpenalty\@secpenalty\addvspace\@tempskipa\fi
	\@dblarg{\@mysect{#1}{#2}{#3}{#4}{#5}{#6}}}
\def\@mysect#1#2#3#4#5#6[#7]#8{\edef\@toclevel{\ifnum#2=\@m 0\else\number#2\fi}\ifnum #2>\c@secnumdepth \let\@secnumber\@empty
	\else \@xp\let\@xp\@secnumber\csname the#1\endcsname\fi
	\@tempskipa #5\relax
	\ifnum #2>\c@secnumdepth
	\let\@svsec\@empty
	\else
	\refstepcounter{#1}\edef\@secnumpunct{\ifdim\@tempskipa>\z@ \@ifnotempty{#8}{\@nx\enspace}\else
		\@ifempty{#8}{.}{\@nx\enspace}\fi
	}\@ifempty{#8}{\ifnum #2=\tw@ \def\@secnumfont{\bfseries}\fi}{}\protected@edef\@svsec{\ifnum#2<\@m
		\@ifundefined{#1name}{}{\ignorespaces\csname #1name\endcsname\space
		}\fi
		\@seccntformat{#1}}\fi
	\ifdim \@tempskipa>\z@ \begingroup #6\relax
	\@hangfrom{\hskip #3\relax\@svsec}{\interlinepenalty\@M #8\par}\endgroup
	\ifnum#2>\@m \else \@tocwrite{#1}{#8}\fi
	\else
	\def\@svsechd{#6\hskip #3\@svsec
		\@ifnotempty{#8}{\ignorespaces#8\unskip
			\@addpunct.}\ifnum#2>\@m \else \@tocwrite{#1}{#8}\fi
	}\fi
	\global\@nobreaktrue
	\@xsect{#5}}
\makeatother

\usepackage{pgfplots}
\pgfplotsset{compat=1.15}
\usepackage{mathrsfs}
\usetikzlibrary{arrows}

\usetikzlibrary{arrows.meta,    
	positioning,
	quotes}         

\usepackage{setspace}

\usepackage{fullpage}

\def\rmlabel{\upshape({\itshape \roman*\,})}

\let\setminus=\smallsetminus
\let\emptyset=\varnothing
\let\vn=\varnothing

\makeatletter
\def\moverlay{\mathpalette\mov@rlay}
\def\mov@rlay#1#2{\leavevmode\vtop{   \baselineskip\z@skip \lineskiplimit-\maxdimen
		\ialign{\hfil$\m@th#1##$\hfil\cr#2\crcr}}}
\newcommand{\charfusion}[3][\mathord]{
	#1{\ifx#1\mathop\vphantom{#2}\fi
		\mathpalette\mov@rlay{#2\cr#3}
	}
	\ifx#1\mathop\expandafter\displaylimits\fi}
\makeatother

\usepackage{hyperref}
\hypersetup{colorlinks, linkcolor={red!50!black}, citecolor={green!50!black}, urlcolor={blue!50!black}}

\usepackage{caption}
\usepackage{subcaption}

\usepackage[mathlines]{lineno}
\usepackage{etoolbox} 

\newcommand*\linenomathpatch[1]{%
	\expandafter\pretocmd\csname #1\endcsname {\linenomath}{}{}%
	\expandafter\pretocmd\csname #1*\endcsname{\linenomath}{}{}%
	\expandafter\apptocmd\csname end#1\endcsname {\endlinenomath}{}{}%
	\expandafter\apptocmd\csname end#1*\endcsname{\endlinenomath}{}{}%
}
\newcommand*\linenomathpatchAMS[1]{%
	\expandafter\pretocmd\csname #1\endcsname {\linenomathAMS}{}{}%
	\expandafter\pretocmd\csname #1*\endcsname{\linenomathAMS}{}{}%
	\expandafter\apptocmd\csname end#1\endcsname {\endlinenomath}{}{}%
	\expandafter\apptocmd\csname end#1*\endcsname{\endlinenomath}{}{}%
}

\expandafter\ifx\linenomath\linenomathWithnumbers
\let\linenomathAMS\linenomathWithnumbers
\patchcmd\linenomathAMS{\advance\postdisplaypenalty\linenopenalty}{}{}{}
\else
\let\linenomathAMS\linenomathNonumbers
\fi

\linenomathpatchAMS{gather}
\linenomathpatchAMS{multline}
\linenomathpatchAMS{align}
\linenomathpatchAMS{alignat}
\linenomathpatchAMS{flalign}
\linenomathpatch{equation}

\usepackage{cleveref}

\theoremstyle{plain}
\newtheorem{theorem}{Theorem}[section]
\crefname{theorem}{Theorem}{Theorems}

\newtheorem{proposition}[theorem]{Proposition}
\crefname{proposition}{Proposition}{Propositions}

\newtheorem{corollary}[theorem]{Corollary}
\crefname{corollary}{Corollary}{Corollaries}

\newtheorem{lemma}[theorem]{Lemma}
\crefname{lemma}{Lemma}{Lemmata}

\newtheorem{conjecture}[theorem]{Conjecture}
\crefname{conjecture}{Conjecture}{Conjectures}

\crefname{problem}{Problem}{Problem}

\newtheorem{claim}[theorem]{Claim}
\crefname{claim}{Claim}{Claims}

\newtheorem{observation}[theorem]{Observation}
\crefname{observation}{Observation}{Observations}

\crefname{setup}{Setup}{Setups}

\newtheorem{fact}[theorem]{Fact}
\crefname{fact}{Fact}{Facts}

\crefname{algorithm}{Algorithm}{Algorithms}

\crefname{remark}{Remark}{Remarks}

\crefname{example}{Example}{Examples}

\theoremstyle{definition}
\newtheorem{definition}[theorem]{Definition}
\crefname{definition}{Definition}{Definitions}

\newtheorem{construction}[theorem]{Construction}
\crefname{construction}{Construction}{Constructions}

\crefname{question}{Question}{Questions}

\numberwithin{equation}{section}

\crefformat{section}{\S#2#1#3}
\crefformat{subsection}{\S#2#1#3}
\crefformat{subsubsection}{\S#2#1#3}
\crefrangeformat{section}{\S\S#3#1#4 to~#5#2#6}
\crefmultiformat{section}{\S\S#2#1#3}{ and~#2#1#3}{, #2#1#3}{ and~#2#1#3}

\crefname{appendix}{Appendix}{Appendix}

\crefname{figure}{Figure}{Figures}

\newcommand{\rf}[1]{\cref{#1} (\nameref*{#1})}

\theoremstyle{definition}
\newtheorem*{connectivity}{Connectivity}
\newtheorem*{space-prop}{Space}
\newtheorem*{aperiodicity}{Aperiodicity}

\newenvironment{proofclaim}[1][Proof of the claim]{\begin{proof}[#1]\renewcommand*{\qedsymbol}{\(\blacksquare\)}}{\end{proof}}

\def\COMMENT#1{}

\usepackage{comment}

\let\polishlcross=\l
\def\l{\ifmmode\ell\else\polishlcross\fi}
\newcommand{\vecb}{\mathbf} 
\newcommand{\es}{\emptyset}
\newcommand{\eps}{\varepsilon}
\renewcommand{\rho}{\varrho}
\newcommand{\sm}{\setminus}
\renewcommand{\subset}{\subseteq}

\newcommand{\NATS}{\mathbb{N}}

\newcommand{\INTS}{\mathbb{Z}}
\newcommand{\REALS}{\mathbb{R}}

\newcommand{\Exp}{\mathbb{E}}
\newcommand{\ori}[1]{\smash{\overrightarrow{#1}}}
\let\vn\relax
\newcommand{\vn}{\mathbf{1}}
\newcommand{\bvec}[1]{\mathbf{#1}}

\let\th\relax
\DeclareMathOperator{\th}{\delta}

\newcommand{\cA}{\mathcal{A}}

\newcommand{\cC}{\mathcal{C}}

\newcommand{\cF}{\mathcal{F}}
\newcommand{\cG}{\mathcal{G}}

\newcommand{\cL}{\mathcal{L}}
\newcommand{\cM}{\mathcal{M}}

\newcommand{\cP}{\mathcal{P}}
\newcommand{\cQ}{\mathcal{Q}}
\newcommand{\cR}{\mathcal{R}}

\newcommand{\cU}{\mathcal{U}}
\newcommand{\cV}{\mathcal{V}}
\newcommand{\cW}{\mathcal{W}}
\newcommand{\cX}{\mathcal{X}}

\DeclareMathOperator{\cover}{{cov}}
\DeclareMathOperator{\res}{{res}}

\DeclareMathOperator{\poly}{poly}
\DeclareMathOperator{\polylog}{polylog}

\DeclareMathOperator{\con}{\mathsf{Con}}
\DeclareMathOperator{\spa}{\mathsf{Spa}}
\DeclareMathOperator{\ape}{\mathsf{Ape}}
\DeclareMathOperator{\ham}{\mathsf{HC}}

\DeclareMathOperator{\hf}{\mathsf{HF}}
\DeclareMathOperator{\natfw}{\mathsf{Nat}}

\DeclareMathOperator{\udiv}{\mathsf{Div}}
\DeclareMathOperator{\ucon}{\overline{\mathsf{Con}}}
\DeclareMathOperator{\uspa}{\overline{\mathsf{Spa}}}
\DeclareMathOperator{\uape}{\overline{\mathsf{Ape}}}

\DeclareMathOperator{\dcon}{{\con}}
\DeclareMathOperator{\dspa}{{\spa}}
\DeclareMathOperator{\dape}{{\ape}}

\DeclareMathOperator{\hamcon}{\mathsf{HamCon}}
\DeclareMathOperator{\mat}{\mathsf{Mat}}
\DeclareMathOperator{\Del}{\mathsf{Del}}
\DeclareMathOperator{\adh}{\mathsf{adh}}
\DeclareMathOperator{\tc}{\mathsf{tc}}
\let\P\relax
\DeclareMathOperator{\P}{\mathsf{P}}

\newcommand{\PG}[3]{{P^{(#3)}}(#1,#2)}

\DeclareMathOperator{\Deg}{\mathsf{MinDeg}}
\newcommand{\DegF}[2]{\Deg(#1,#2)}
\newcommand{\DegSeq}{{\mathsf{DegSeq}}}

\newcommand{\rdeg}{\overline{\deg}}

\newcommand{\tightly}{}
\newcommand{\tight}{}

\title{A hypergraph bandwidth theorem}

\author[R.~Lang]{Richard Lang}

\address[R.~Lang]{
	Departament de Matemàtiques,
	Universitat Politècnica de Catalunya,
	Barcelona, Spain and
	Centre de Recerca Matemàtica, Barcelona, Spain
}
\email{richard.lang@upc.edu}

\author[N.~Sanhueza-Matamala]{Nicolás Sanhueza-Matamala}

\address[N.~Sanhueza-Matamala]{ 
	Departamento de Ingeniería Matemática and CI$^2$MA,
	Facultad de Ciencias Físicas y Matemáticas,
	Universidad de Concepción,
	Concepción, Chile
}
\email{nsanhuezam@udec.cl}

\begin{document}

\begin{abstract}
	A cornerstone of extremal graph theory due to Erdős and Stone states that the edge density which guarantees a fixed graph $F$ as subgraph also asymptotically guarantees a blow-up of $F$ as subgraph.
	It is natural to ask whether this phenomenon generalises to vertex-spanning structures such as Hamilton cycles.
	This was confirmed by Böttcher, Schacht and Taraz for graphs in the form of the Bandwidth Theorem.
	Our main result extends the phenomenon to hypergraphs.

	A graph on $n$ vertices that robustly contains a Hamilton cycle must satisfy certain conditions on space, connectivity and aperiodicity.
	Conversely, we show that if these properties are robustly satisfied, then all blow-ups of cycles on $n$ vertices with clusters of size at most $\poly (\log \log n)$ are guaranteed as subgraphs.
	This generalises to powers of cycles and to the hypergraph setting.

	As an application, we recover a series of classic results and recent breakthroughs on Hamiltonicity under degree conditions, which are then immediately upgraded to blown up versions.
	The proofs are based on a new setup for embedding large substructures into dense hypergraphs, which is of independent interest and does not rely on the Regularity Lemma or the Absorption Method.
\end{abstract}

\subjclass[2020]{05C35 (primary), 05C45, 05C65, 05C70 (secondary)}
\keywords{Hamilton cycles, hypergraphs, minimum degree}

\maketitle


\vspace{-0.5cm}

\section{Introduction}\label{sec:introduction}

An old problem of Turán is to determine the optimal edge density $\tau(F)$ that guarantees a copy of a $k$-uniform hypergraph $F$ in a host hypergraph $G$.
In the graph setting, Erdős and Stone~\cite{ES46} proved that under a mildly stronger assumption, one can even guarantee a \emph{blow-up} of $F$.
Formally, such a blow-up is formed by replacing the vertices of $F$ with vertex-disjoint clusters and each edge with a complete $k$-partite $k$-graph whose parts are the corresponding clusters.
Erdős~\cite{Erd64} extended this phenomenon to hypergraphs, showing that the presence of many copies of $F$ already forces a blow-up of $F$ with clusters of size $\polylog n$.
In combination with supersaturation~\cite{ES83}, this extends the Erdős--Stone Theorem to hypergraphs, which is remarkable insofar as the value of $\tau(F)$ is generally unknown.
Random constructions show that the cluster sizes in these blow-ups cannot exceed $\polylog n$.
Nevertheless, an ongoing branch of research is dedicated to optimising the bounds in this area~\cite{BE73,CS81,Nik09}.

So far, the order of the substructures has been tiny in comparison to the host structure.
It is then natural to ask whether the Erdős--Stone theorem generalises to much larger, even vertex-spanning substructures.
To extend Turán-type problems in this direction, one typically replaces bounds on density with some condition on the degrees.
For instance, the classic theorem of Dirac~\cite{Dir52} provides optimal minimum degree conditions for the existence of a Hamilton cycle in an (ordinary) graph.
This was later generalised to powers of cycles by Komlós, Sárközy and Szemerédi~\cite{KSS98}, which can be viewed as a combination of Dirac's and Turán's theorems.
Given this, Bollobás and Komlós~\cite{KS96} asked whether one can find a vertex-spanning blow-up of a power of a cycle under marginally stronger degree conditions.
This vertex-spanning extension of the Erdős--Stone theorem was confirmed by Böttcher, Schacht and Taraz~\cite{BST09} by proving the Bandwidth Theorem (\cref{thm:bandwidth-graphs-linear}).

Our main result presents a solution to this question for hypergraphs.
To state Dirac-type results for hypergraphs, we introduce some terminology.
For $1 \leq d <k$, the \emph{minimum $d$-degree $\delta_d(G)$} of a  $k$-uniform hypergraph $G$ is the maximum $m$ such that every set of $d$ vertices is in at least $m$ edges.
The \emph{$(t-k+1)$st power of a Hamilton cycle} $C \subset G$ has a cyclical vertex ordering such that every set of $t$ cyclically consecutive vertices forms a clique $K_t^{(k)}$.
Given this, the \emph{threshold} $\delta_d^{\ham}(k,t)$ is the infimum $\delta \in [0,1]$ such that for every $\eps > 0$, there is $n_0$ such that every $k$-graph $G$ on $n \geq n_0$ vertices with $\delta_d(G)\geq (\delta + \eps ) \binom{n-d}{k-d}$ contains the $(t-k+1)$st power of a Hamilton cycle.
For instance, Dirac's theorem implies that $\delta_1(2, 2) = 1/2$ and the above-mentioned result of Komlós, Sárközy and Szemerédi implies that $\delta_1(2,t) = 1-1/t$.
In general, we can think of $\delta_d^{\ham}(k,t)$ as the analogue of $\tau(K_t^{(k)})$ for Hamilton cycles.
As with the latter, little is known about its value except for a few cases (surveyed in \cref{sec:background+applications}).

In the spirit of the Erdős--Stone problem, our goal is to state a hypergraph bandwidth theorem in terms of $\delta_d^{\ham}(k,t)$.
However, in addition to the existence of Hamilton cycles, one also needs a certain consistency guarantee of their location.
In \cref{sec:frameworks}, we define a threshold $\delta_d^{\hf}(k,t)$ that captures this condition.
Moreover, we shall see that $\delta_d^{\ham}(k,t) = \delta_d^{\hf}(k,t)$ in all cases where $\delta_d^{\ham}(k,t)$ is known (\cref{thm:thresholds}).
Given this, we can formulate a simplified version of our main result.
Unlike the original \emph{Bandwidth Theorem} (\cref{thm:bandwidth-graphs-linear}), our statement is formulated in terms of blow-ups to avoid the `bandwidth' parameter, which is not defined for hypergraphs (see the discussion after \cref{thm:bandwidth-graphs-linear}).

\begin{theorem}[Hypergraph Bandwidth Theorem]\label{thm:framework-bandwidth-simple}
	For every $1 \leq d < k \leq t$ and $\eps >0$, there exist $c$ and~$n_0$ with the following properties.
	Let $G$ be a $k$-uniform hypergraph on $n\geq n_0$ vertices with $$\delta_d(G) \geq \big(\delta_d^{\hf}(k,t) + \eps \big) \tbinom{n-k}{k-d}.$$
	Let $H$ be an $n$-vertex blow-up of the $(t-k+1)$st power of a $k$-uniform cycle with clusters of (not necessarily uniform) size at most $(\log \log n)^c$.
	Then $H \subset G$.
\end{theorem}

In the next two sections, we give a more detailed introduction to the overall topic and our contributions.
We conclude this overview with a summary of our main outcomes:

\begin{itemize}
	\item The threshold $\delta_d^{\hf}(k,t)$ is conceptually simpler than $\delta_d^{\ham}(k,t)$ and thus easier to determine.
	      Plain powers of Hamilton cycles can be recovered by taking clusters of size one in \cref{thm:framework-bandwidth-simple}.
	      So in particular $\delta_d^{\ham}(k,t) \leq \delta_d^{\hf}(k,t)$.
	\item \cref{thm:framework-bandwidth-simple} is a consequence of a more general result (\cref{thm:framework-bandwidth}) where the requirement on the minimum degree is replaced by a more general condition based on subgraph counts.
	      This corresponds to Erdős' theorem~\cite{Erd64}, where a blow-up of $F$ is forced via the presence of many copies of $F$ instead of the Turán density~$\tau(F)$.
	\item We also prove a Hamilton-connectedness result (\cref{thm:framework-connectedness-robust}), which allows us to find Hamilton paths with some prescribed end-tuples. Together with our former work with Joos~\cite{JLS23}, this leads to counting and random-robust Dirac-type theorems.
	\item Our proofs do not rely on the Regularity Lemma, the Blow-up Lemma or the Absorption Method, which are the prevalent techniques in the field.
	      Instead, we develop a new setup (\cref{pro:blow-up-cover}) which uses `blow-up covers' to tackle these types of problems.
\end{itemize}

For a gentle introduction to the proof ideas, we also refer to our related work~\cite{LS24b}, where a similar approach is used to find blow-ups with optimal cluster sizes in the more basic setting of Hamilton cycles in graphs with large minimum degree.

\subsection*{Organisation of the paper}

The remainder of the paper is organised as follows.
In the next section, we provide more details on the history of Dirac-type problems and present outcomes derived from our work.
In \cref{sec:frameworks}, we introduce \cref{thm:framework-bandwidth}, which generalises \cref{thm:framework-bandwidth-simple} and presents the main result of this paper.
The exposition until \cref{sec:frameworks} can be regarded as an extended introduction.

The proofs of the applications can be found in \cref{sec:applications-proofs}.
In \cref{sec:directed-setup}, we reformulate our main result in the directed setting (\cref{thm:framework-bandwidth-directed}) and derive \cref{thm:framework-bandwidth-simple} using a transition between directed and undirected graphs (\cref{pro:framework-undirected-to-directed}).
The proof of \cref{thm:framework-bandwidth-directed} can be found in \cref{sec:proof-main-result} subject to results on covering with blow-ups (\cref{pro:blow-up-cover}) and allocating subgraphs (\cref{pro:allocation-Hamilton-path,pro:allocation-bandwidth-frontend}).
We continue by giving a proof of \rf{pro:blow-up-cover} in \cref{sec:blow-up-covers}.
This is followed by an interlude in \cref{sec:robust-ham-connect}, where we prove the aforementioned result on Hamilton connectedness (\cref{thm:framework-connectedness-robust}).
Returning to the main line of the argument, the allocation details are spelled out in \cref{pro:allocation-Hamilton-path,pro:allocation-bandwidth-frontend} in \cref{sec:allocation-cycle} and \cref{sec:allocation-path-blow-up}, respectively.
We finish the proof of \cref{thm:framework-bandwidth} by showing \cref{pro:framework-undirected-to-directed} in \cref{sec:boosting+orientations}.
The paper is concluded in \cref{sec:conclusion} with a discussion of open problems and related questions.

\section{Background and applications}\label{sec:background+applications}

The study of Hamilton cycles and related structures under minimum degree conditions presents a rich branch of research.
In the following, we provide a short historical overview, followed by several applications of \cref{thm:framework-bandwidth-simple}.

\subsection*{Graphs}

To put \cref{thm:framework-bandwidth-simple} into perspective, we begin in the graph setting.
The study of Hamiltonicity under degree conditions dates back to Dirac~\cite{Dir52}, who proved that a graph $G$ on $n \geq 3$ vertices with minimum degree $\delta_1(G) \geq n/2$ contains a Hamilton cycle.
Over the past decades, this was generalised by Hajnal and Szemerédi~\cite{HS70} to clique factors and by Komlós, Sárközy and Szemerédi~\cite{KSS98} to powers of Hamilton cycles, who showed, in the above notation, that $\delta_1^{\ham}(2,t) = 1-1/t$.
Proving a conjecture of Bollobás and Komlós~\cite{KS96}, this was extended by Böttcher, Schacht and Taraz~\cite{BST09} to sparse subgraphs of blow-ups of cycles:

\begin{theorem}[Bandwidth Theorem]\label{thm:bandwidth-graphs-linear}
	For every $2 \leq k \leq t$ and $\eps >0$, there exist $c$ and $n_0$ with the following properties.
	Let $G$ be a graph on $n\geq n_0$ vertices with $$\delta_1(G) \geq \big(\delta_1^{\hf}(2,t) + \eps \big) n.$$
	Let $H$ be a subgraph of an $n$-vertex blow-up of the $(t-1)$st power of a cycle with clusters of size at most $cn$ and maximum degree $\Delta(H) \leq \Delta$.
	Then $H \subset G$.
\end{theorem}

The conjecture of Bollobás and Komlós~\cite{KS96} and its solution were originally stated in terms of the bandwidth parameter, hence the name of the theorem.
It is not hard to see that \cref{thm:bandwidth-graphs-linear} implies this formulation \cite[Lemma 10.1]{LS23}.
(The basic reason is that a graph $H$ of chromatic number $\chi(H) \leq t$ and bandwidth at most $b$ can be greedily embedded into the $(t-1)$st power of a path with clusters of size at most $tb$.)

\cref{thm:framework-bandwidth-simple,thm:bandwidth-graphs-linear} coincide for blow-ups of bounded cluster size.
However, in general they are not quite comparable.
While our result comes with a more restrictive cluster size, the maximum degree is not bounded.
This allows embeddings of graphs with $n \polylog \log n$ edges, which was not possible before.

The proof of \cref{thm:bandwidth-graphs-linear} is based on a sophisticated combination of Szemerédi's Regularity Lemma~\cite{Sze76} and the Blow-up Lemma~\cite{KSS97}.
These techniques extend beyond the setting of minimum degree conditions, and with some additional ideas one can generalise \cref{thm:bandwidth-graphs-linear} to many other host graph families that have been studied in terms of Hamiltonicity~\cite{LS23}.

\subsection*{Hypergraphs}

{A \emph{$k$-uniform cycle} $C \subset G$ comes with a cyclical ordering of its vertices, such that the edges of $C$ are formed by the sets of $k$ consecutive vertices.}
(In the literature, these structures are also known as \emph{tight cycles}.)
The study of Dirac-type theorems for Hamilton cycles in hypergraphs was initiated by Katona and Kierstead~\cite{KK99} and further popularised in the seminal work of Rödl, Ruciński and Szemerédi~\cite{RRS08a}.
Over the past two decades, significant efforts have been directed towards extending their work.
It is generally believed that the following constructions due to Han and Zhao~\cite{HZ16} capture the extremal behaviour of the problem.

\begin{construction}\label{const:tight-general}
	For $1 \leq d \leq k-1$, let $\ell=k-d$.
	Choose $0 \leq j \leq k$  such that $(j-1)/k < \lceil \ell/2 \rceil /(\ell+1) < (j+1)/k$.
	Let $H$ be a $k$-uniform hypergraph on $n$ vertices and a subset of vertices  $X$ with $|X|= \lceil \ell/2 \rceil n /(\ell+1)$, such that $H$ contains precisely the edges $S$ for which $|S \cap X| \neq j$.
\end{construction}

For convenience, we abbreviate $\delta_d^{\ham}(k)=\delta_d^{\ham}(k,k)$.
It is not hard to see that the hypergraphs of \cref{const:tight-general} do not contain Hamilton cycles, and thus give lower bounds for the threshold $\delta_d^{\ham}(k)$~\cite{HZ16}.
Interestingly, the lower bounds obtained from these configurations are uniform in terms of $\ell = k-d$.
For instance, we have $\delta_d^{\ham}(k) \geq 1/2$, $5/9$, $5/8$ and $408/625$ whenever $k-d=1$, $2$, $3$ and $4$, respectively.
This behaviour is quite different for related structures, such as perfect matchings~\cite{Zha16}.
The central conjecture in the area is that the minimum degree conditions derived from these configurations are optimal.

\begin{conjecture}\label{con:Han-Zhao-are-best-possible}
	The minimum $d$-degree threshold for $k$-uniform Hamilton cycles coincides with the lower bounds given by Construction~\ref{const:tight-general}.
\end{conjecture}

The conjecture was resolved for \emph{codegrees} (when $d = k-1$) by Rödl, Ruciński, and
Szemerédi~\cite{RRS08a}, by introducing the \emph{Absorption
	Method} in its modern form. Since then, the focus has shifted
to degree types below $k-1$. After advances for nearly spanning cycles by Cooley and Mycroft~\cite{CM17},
it was shown by Reiher, Rödl, Ruciński, Schacht and
Szemerédi~\cite{RRR19} that $\delta_1^{\ham}(3)=5/9$, which resolves the case
of $d=k-2$ when $k=3$. Subsequently, this was generalised to
$k=4$~\cite{PRR+20} and finally, Polcyn, Reiher, R\"odl, and
Schülke~\cite{PRRS21} and, independently, Lang and
Sanhueza-Matamala~\cite{LS22} established $\delta_{k-2}^{\ham}(k)=5/9$ for all
$k \geq 3$.
Recently, Lang, Schacht and Volec~\cite{LSV24} proved that $\delta_{k-3}^{\ham}(k)=5/8$ for all $k \geq 4$.
For later reference, we summarise these efforts as follows:

\begin{theorem}\label{thm:hamilton-cycle-k-d-small}
	We have $	\delta_{k-1}^{\ham}(k) = 1/2$,  $\delta_{k-2}^{\ham}(k) = 5/9$ and $\delta_{k-3}^{\ham}(k) = 5/8.$
\end{theorem}

For general $k$ and $d$, the following cruder bounds are known~\cite{HZ16, LS22}.

\begin{theorem}\label{thm:hamilton-cycle-k-d-large}
	For $1\leq d < k$, we have
	$$1 - \sqrt{2/(k-d)\pi} \leq \delta_{d}^{\ham}(k) \leq 1 - 1/2(k-d).$$
\end{theorem}

Powers of Hamilton cycles for hypergraphs were first studied by Bedenknecht and Reiher~\cite{BR20} for $3$-uniform graphs.
More recently, Pavez-Signé, Sanhueza-Matamala and Stein~\cite{PSS23} generalised their work to higher uniformities.
Thus far, research has focused on the case of {codegrees}.
For $k \leq t$, set $f_k(t) = 1 - 1/\left(\tbinom{t-1}{k-1}+\tbinom{t-2}{k-2}\right)$.

\begin{theorem}[{\cite{BR20,PSS23}}]\label{thm:power-hamilton-cycle-codegree}
	We have $\delta_{k-1}^{\ham}(k,t) \leq f_k(t)$ for every $2 \leq k \leq t$.
\end{theorem}

We do not have a conjecture on the general behaviour of the threshold $\delta_{d}^{\ham}(k,t)$.
However, it is plausible that $\delta_{2}^{\ham}(3,4) = 3/4$ in this special case; see  \cref{sec:conclusion} for more details.

\medskip
As discussed before, bandwidth theorems are relatively well-understood in the graph setting.
For hypergraphs, however, no equivalent (or partial) results have been obtained, to the best of our knowledge.

\subsection*{Applications}

Our (simplified) main result, \cref{thm:framework-bandwidth-simple}, has a series of consequences for Dirac-type problems in hypergraphs.
We believe that vertex-spanning blow-ups of powers of cycles appear under the same (asymptotic) minimum degree conditions as powers of Hamilton cycles.
Given \cref{thm:framework-bandwidth-simple}, this can be formalised as follows.

\begin{conjecture}\label{con:thresholds}
	For all $1 \leq d < k \leq t$, we have $\delta_d^{\hf}(k,t) = \delta_d^{\ham}(k,t)$.
\end{conjecture}

We remark that one can combine a stronger version of our main result (\cref{thm:framework-connectedness-robust}) together with our work with Joos~\cite{JLS23} to show that \cref{con:thresholds} implies a conjecture of Kelly, Müyesser and Pokrovskiy~\cite[Conjecture 8.1]{KMP23}, which states that $\delta_d^{\ham}(k,t)$ is the optimal threshold in the random robust setting.
In support of \cref{con:thresholds}, we confirm all cases where $\delta_d^{\ham}(k,t)$ is known (\cref{thm:hamilton-cycle-k-d-small}).

\begin{theorem}\label{thm:thresholds}
	We have $\delta_{k-1}^{\hf}(k) = 1/2$, $\delta_{k-2}^{\hf}(k) = 5/9$ and $\delta_{k-3}^{\hf}(k) = 5/8$.
\end{theorem}

It is worth noting that \cref{thm:thresholds} does not rely on the technical setup of our past work~\cite{LS22}.
Instead, we integrate the structural ideas from the proofs of \cref{thm:hamilton-cycle-k-d-small} into our framework to give much cleaner proofs.
This is illustrated for the exemplary case $\delta_{1}^{\hf}(3) = 5/9$, for which we provide the entire details (see \cref{sec:vertex-deg-3-graph}).

Finally, we also extend  \cref{thm:hamilton-cycle-k-d-large,thm:power-hamilton-cycle-codegree}
to blow-ups as follows.

\begin{theorem}\label{thm:hamilton-map-k-d-large}
	For $1\leq d < k$, we have
	$$1 - \sqrt{2/(k-d)\pi} \leq \delta_{d}^{\hf}(k) \leq 1 - 1/2(k-d).$$
\end{theorem}

\begin{theorem}\label{thm:power-hamilton-map-codegree}
	We have $\delta_{k-1}^{\hf}(k,t) \leq f_k(t)$ for every $2 \leq k \leq t$.
\end{theorem}

In summary, we recover all known results on Hamiltonicity and extend them to the bandwidth setting (blow-ups of cycles).

\section{A framework for Hamiltonicity}\label{sec:frameworks}

So far, the requirements on the host (hyper)graphs have been expressed in terms of minimum degree conditions.
In the graph setting, many other assumptions have been considered, ranging from degree sequences to quasirandomness.
The corresponding results are often proved with similar embedding techniques, but differ in their structural analysis.
One might therefore wonder whether there is a common structural base that ‘sits between’ all of these assumptions and (variations of) Hamiltonicity. For (ordinary) Hamilton cycles, this was done by Kühn, Osthus and Treglown~\cite{KOT2010} using the notion of \emph{robust expanders}.
Ebsen, Maesaka, Reiher, Schacht and Schülke~\cite{EMR+20} raised the question of generalising the concept of robust expanders to handle powers of cycles and other cyclical structures in the graph setting.
We solved this problem by introducing the notion of Hamilton frameworks and proved a generalisation of the Bandwidth Theorem~\cite{LS23}.
In the following, we present a hypergraph version of these results.
For an alternative approach in the graph setting, see also the work of Maesaka~\cite{Mae23}.

\subsection{Necessary conditions}\label{sec:necessary-conditions}

We begin with a discussion of three structural features that constitute Hamiltonicity: \emph{connectivity}, \emph{space} and \emph{aperiodicity}.
Let $G$ be a $k$-uniform hypergraph, or \emph{$k$-graph} for short.

\begin{connectivity}
	We denote by $L(G)$ the \emph{line graph} of $G$, which is the $2$-graph on vertex set $E(G)$ with an edge $ef$ whenever $|e \cap f|=k-1$.
	A subgraph of $G$ is \emph{\tightly connected} if it has no isolated vertices and its edges induce a connected subgraph in $L(G)$.
	Moreover, we refer to edge-maximal \tightly connected subgraphs as \emph{\tight components}.
\end{connectivity}

\begin{space-prop}
A \emph{fractional matching} is a function $\omega\colon E(G) \to [0,1]$ such that $\sum_{e\colon v \in e} \omega (e) \leq 1$ for every vertex $v \in V(G)$.
The \emph{size} of a fractional matching $\omega$ is $\sum_{e \in E(G)} \omega (e)$.
We say that $\omega$ is \emph{perfect} if its size is $n/k$.
\end{space-prop}

\begin{aperiodicity}
	A \emph{homomorphism} from a $k$-graph $C$ to $G$ is a function $\phi \colon V(C) \to V(G)$ that maps edges to edges.
	We say that $W \subset G$ is a \emph{closed \tight walk} if $W$ is the image of a homomorphism of a $k$-uniform cycle $C$.
	The \emph{order} of $W$ is the order of $C$.
	We call $G$ \emph{aperiodic} if it contains a closed walk whose order is congruent to $1$ modulo $k$.
\end{aperiodicity}

It is easy to see that any $k$-uniform cycle is connected, contains a perfect fractional matching and (if its order is coprime to $k$) is aperiodic.
In other words, these three properties are satisfied by any family $\cG$ of $k$-graphs whose members are Hamiltonian after deleting up to $k-1$ vertices.
In addition to this, constructions show that such a family $\cG$ must also satisfy a fourth property, which ensures that these features are \emph{consistent} between reasonably similar members of $\cG$.\footnote{In \cite[Appendix B.1]{LS23}, we construct a non-Hamiltonian $3$-graph $G$ such that for every vertex $x \in V(G)$, there are $G'_x \subseteq G - x$ satisfying \ref{itm:hf-connected}--\ref{itm:hf-odd} (even in a robust way). The main obstacle to Hamiltonicity is that the union $G'_x \cup G'_y$ is not connected for certain choices of $x, y$; and this is precisely what is avoided by \ref{itm:hf-intersecting}.}
This motivates the following definition.

\begin{definition}[Hamilton framework]\label{def:hamilton-framework}
	A family $\cG$ of $s$-vertex $k$-graphs has a \emph{Hamilton framework} $F$ if for every $G \in \cG$ there is an {$s$-vertex} subgraph $F(G) \subset G$ such that
	\begin{enumerate}[(F1)]
		\item \label{itm:hf-connected} $F(G)$ is a \tight component, \hfill(connectivity)
		\item \label{itm:hf-matching} $F(G)$ has a perfect fractional matching, \hfill(space)
		\item \label{itm:hf-odd} $F(G)$ contains a closed walk of order $1 \bmod k$, and \hfill(aperiodicity)
		\item \label{itm:hf-intersecting} {$F(G) \cup F(G')$ is \tightly connected whenever $G,G' \in \cG$ are obtained by deleting a distinct vertex from the same $(s+1)$-vertex $k$-graph. \hfill(consistency)}
	\end{enumerate}
\end{definition}

Thus, a Hamilton framework is a family of graphs which satisfy the four aforementioned necessary properties for Hamiltonicity.
It is not true that every member of a $k$-graph family which has a Hamilton framework is Hamiltonian, but, as we will see, a natural strengthening of these properties does indeed imply Hamiltonicity.

Before we continue exploring this idea, let us connect Hamilton frameworks to degree conditions to complete the statements of \cref{thm:framework-bandwidth-simple,con:thresholds}.
For a $k$-graph $G$ and $t\geq k$, the \emph{clique-graph} $K_t(G)$ has vertex set $V(G)$ and a $t$-uniform edge $X$ whenever $G[X]$ induces a clique.
Note that a $(t-k+1)$st power of a Hamilton cycle in $G$ corresponds to a Hamilton cycle in $K_t(G)$.

\begin{definition}\label{def:ham-fw-threshold}
	Let $\delta_d^{\hf}(k,t)$ be the infimum $\delta \in [0,1]$ such that for every $\eps$, there is $n_0$ such that the family of $t$-graphs $K_t(G)$ on $n \geq n_0$ vertices, where $G$ is a $k$-graph with $\delta_d(G) \geq (\delta + \eps ) \binom{n-d}{k-d}$, admits a Hamilton framework.
\end{definition}

\subsection{Robustness}\label{sec:robustness}

To strengthen the notion of Hamilton frameworks, we require that a given $k$-graph  satisfies the required properties not only globally, but also in a robust local sense.
This is formalised using the notion of `property graphs', which was introduced in the context of perfect tilings~\cite{Lan23}.

\begin{definition}[Property graph]\label{def:property-graph}
	For a $k$-graph $G$ and a family of $k$-graphs $\P$, the \emph{$s$-uniform property graph}, denoted by $\PG{G}{\P}{s}$, is the $s$-graph on vertex set $V(G)$ with an edge $S \subset V(G)$ whenever the induced subgraph $G[S]$ \emph{satisfies}~$\P$, that is $G[S] \in \P$.
\end{definition}

Thus, an $s$-uniform property graph of $G$ tracks the local places that inherit the property $\P$.
We remark that for all practical purposes, we can assume that $s$ is much larger than $k$, and the order of $G$ is much larger than $s$.
Given this, we express the robustness of property $\P$ in a $k$-graph $G$ by saying that the property $s$-graph for $\P$ has sufficiently large minimum degree.

\begin{definition}[Robustness]\label{def:robustness}
	For a family of $k$-graphs $\P$, $r=2k$ and $\delta=1-1/s^2$, a $k$-graph $G$ \emph{$s$-robustly satisfies} $\P$ if the minimum $r$-degree of the property $s$-graph $\PG{G}{  \P   }{s}$ is at least $\delta \tbinom{n-r}{s-r}$.
\end{definition}

We remark that the `hard-coded' values of the parameters $r$ and $\delta$ in the previous definition can be further optimised in the context of our main results.
However, in practice this does not matter since $r=o(s)$ and $\delta = 1-\exp^{-\Omega(s)}$ in the intended applications of our main result.
For a more detailed discussion, see \cref{sec:booster-lemma}.

\subsection{Sufficient conditions}

We are now ready to formulate our main result, which states that every graph whose clique-graph robustly admits a Hamilton framework contains all suitable blow-ups of powers of a Hamilton cycle.

\begin{theorem}[General Hypergraph Bandwidth Theorem]\label{thm:framework-bandwidth}
	For  all $k$, $t$ and $s$, there are $c$ and $n_0$ such that the following holds.
	Let $\P$ be a family of $s$-vertex $t$-graphs that admits a Hamilton framework.
	Let $G$ be a $k$-graph on $n \geq n_0$ vertices such that $K_t(G)$ $s$-robustly satisfies $\P$.
	Let $H$ be an $n$-vertex blow-up of the $(t-k+1)$st power of a $k$-uniform cycle with clusters of size at most $(\log \log n)^c$.
	Then $H \subset G$.
\end{theorem}

The intended setting to apply \cref{thm:framework-bandwidth} are host graphs from hereditary graph families.
By this, we mean a family of graphs that is approximately closed under taking typical induced subgraphs of constant order.
This is formalised using an \emph{Inheritance Lemma} such as \cref{lem:inheritance-minimum-degree}.
Given inheritance, the robustness of the assumptions of  \cref{thm:framework-bandwidth} becomes trivial, and the remaining task consists in verifying the Hamilton framework properties.
Since degree conditions are hereditary, we shall thus see that \cref{thm:framework-bandwidth} implies \cref{thm:framework-bandwidth-simple}.
A more detailed explanation of the argument in various applications can be found in \cref{sec:applications-proofs}.

\subsection{Methodology}

There are two prevalent approaches for the embedding of spanning structures in graphs and hypergraphs.
The first approach combines Szemerédi's Regularity Lemma with the Blow-up Lemma of Komlós, Sárközy and Szemerédi~\cite{KSS98}.
This framework was key in the proof of the Bandwidth Theorem (\cref{thm:bandwidth-graphs-linear}).
The second approach is based on the Absorption Method that is often also used in conjunction with a (hypergraph) Regularity Lemma.
This setup was introduced by Rödl, Ruciński, and Szemerédi~\cite{RRS08a} to find \tight Hamilton cycles under codegree conditions (the case $k-d =1$ of \cref{thm:hamilton-cycle-k-d-small}) and has formed the basis of all contributions in this direction thereafter.

Both methods come with advantages and drawbacks.
The combination of the Regularity and Blow-up Lemma offers a systematic way to tackle embedding problems for dense graphs (more on this below).
However, this framework also comes with many technical challenges, which usually lead to long and convoluted proofs~\cite{LS23}.
This phenomenon is most dramatic for (proper) hypergraphs, where even stating the definitions requires effort and attention to detail.
In this setting, we are moreover still lacking some of the basic instruments such as a Blow-up Lemma with image restrictions.
The Absorption Method, on the other hand, avoids many of these technicalities and has therefore become the approach of choice in the hypergraph setting.
It is particularly well-suited for the embedding of spanning structures that exhibit strong symmetries.
Absorption for less symmetrical structures becomes more difficult and it is not well-understood whether it works at all.
For instance, there is so far no proof of the Bandwidth Theorem using the Absorption Method.
Finally, we note that neither of these methods works systematically well with graphs of unbounded maximum degree.

To overcome these limitations, we propose a new framework for embedding large structures into dense hypergraphs.
Our approach is best explained in the context of a Blow-up Lemma, which can be summarised as follows.
Broadly speaking, the Regularity Lemma allows us to approximate a given host-graph $G$ with a quasirandom blow-up of a \emph{reduced graph} $R$ of constant order.
So the vertices of $R$ are replaced by (disjoint, linear-sized) vertex clusters and the edges of $R$ are replaced by quasirandom partite graphs.
Importantly, the Regularity Lemma guarantees that $R$ approximately inherits structural assumptions on $G$ such as degree conditions.
The Blow-up Lemma then tells us that one can effectively assume that these partite graphs are complete.
Ignoring a number of technicalities, this reduces the problem of embedding a given guest graph $H$ into $G$ to the problem of embedding $H$ into a complete blow-up of $R$.
This last step is known as the \emph{allocation} of $H$, and there is a well-tested set of techniques available to facilitate this step. Therefore, the reduction to the allocation problem often presents a considerable step towards the successful embedding of $H$ into $G$.

In contrast to this, our method proceeds as follows.
Instead of approximating the entire structure of~$G$ with a quasirandom blow-up, we only cover the vertex set of~$G$ with a family of complete blow-ups $R_1(\cV_1),\dots, R_\ell(\cV_\ell)$ that are interlocked in a cycle-like fashion.
As before, the \emph{reduced graphs} $R_1,\dots,R_\ell$ inherit structural properties of $G$ such as degree conditions.
Now suppose we are given a blow-up of a cycle~$H$.
We subdivide $H$ into blow-ups of paths $H_1,\dots,H_\ell$, where each $H_i$ has approximately $|\bigcup \cV_i|$ vertices.
This effectively reduces the problem of embedding $H$ into $G$ to embedding each $H_i$ into the blow-up $R_i(\cV_i)$, which is analogous to the above-described allocation problem.

The connectivity, space and divisibility properties of \cref{def:hamilton-framework} allow us to carry out the allocation into each $R_i(\cV_i)$ by combining arguments from our past work~\cite{Lan23,LS23} with several new ideas.
The consistency property of \cref{def:hamilton-framework} then guarantees that the allocations for $R_i(\cV_i)$ and $R_{i+1}(\cV_{i+1})$ are properly synchronised.

\subsection{Notation}

We express some of the constant hierarchies in standard $\gg$-notation.
To be precise, we write $y \gg x$ to mean
that for any $y \in (0, 1]$ there exists an $x_0 \in (0,1)$
such that for all $x_0 \geq x$ the subsequent statements
hold.  Hierarchies with more constants are defined in a
similar way and are to be read from left to right following the order that the constants are chosen.
Moreover, we tend to ignore rounding errors if the context allows~it.

\section{Proofs of the applications}\label{sec:applications-proofs}

In this section, we prove the results announced in~\cref{sec:background+applications}.
Thanks to \cref{thm:framework-bandwidth}, we only need to find a robust Hamilton framework.
Moreover, the robustness comes effectively for free thanks to the inheritance principle.
The remaining challenge thus consists in identifying a Hamilton framework under minimum degree conditions.
This is done by exploiting the structural insights of former work.

We begin by deriving \cref{thm:framework-bandwidth-simple} from \cref{thm:framework-bandwidth}.
The remainder of the section is then dedicated to the proofs of \cref{thm:thresholds,thm:hamilton-map-k-d-large,thm:power-hamilton-map-codegree}.
For the sake of illustration, we also give the full details for vertex-degree conditions in $3$-graphs:

\begin{theorem}\label{pro:5/9-framework}
	For every $\eps > 0$, there is $n_0$ such that for each $n \geq n_0$, the family of $n$-vertex $3$-graphs $H$ satisfying $\delta_1(H) \geq \left( 5/9 + \varepsilon \right) \binom{n-1}{2}$ admit a Hamilton framework.
\end{theorem}

Together with \cref{thm:framework-bandwidth-simple}, this recovers and extends a breakthrough result of Reiher, Rödl, Ruciński, Schacht and Szemerédi~\cite{RRR19}.

\subsection{Inheritance}

Our framework applies to host graphs that satisfy graph properties, which are hereditary in the sense of being approximately closed under taking typical induced subgraphs of constant order.
This is formalised by the following Inheritance Lemma.
To state it, we need the following definition.

\begin{definition}\label{def:degree-family}
	For $1 \leq d \leq k$ and $\delta \geq 0$, the family $\DegF{d}{\delta}$ contains, for all $k \leq n$, every $n$-vertex $k$-graph~$G$ with minimum $d$-degree $\delta_d(G) \geq \delta \binom{n-d}{k-d}$.
\end{definition}

\begin{lemma}[Inheritance Lemma]\label{lem:inheritance-minimum-degree}
	For $1/k,\,1/r,\,\eps \gg 1/s \gg 1/n$ and $\delta \geq 0$, let $G$ be an $n$-vertex $k$-graph with $\delta_d(G) \geq (\delta + \eps) \tbinom{n-d}{k-d}$.
	Then the property $s$-graph $P =\PG{G}{\DegF{d}{\delta+\eps/2}}{s}$ satisfies $\delta_{r}(P) \geq  (1-e^{-\sqrt{s}}    )  \tbinom{n-r}{s-r}$.\qed
\end{lemma}

The proof of \cref{lem:inheritance-minimum-degree} follows from standard probabilistic concentration bounds~\cite[Lemma 4.9]{Lan23}.
For the sake of completeness, the details can be found in \cref{sec:inheritance}.

Now we can easily derive \cref{thm:framework-bandwidth-simple}.

\begin{proof}[Proof of \cref{thm:framework-bandwidth-simple}]
	Set $\delta = \delta_d^{\hf}(k,t)$.
	Given $\eps > 0$, set $r=2r$, and choose $s,c$ with $1/r,\,\eps \gg 1/s \gg c \gg 1/n$.
	Let $G$ be a $k$-graph on $n$ vertices with $\delta_d(G) \geq (\delta + \eps) \binom{n}{2}$.
	By \cref{lem:inheritance-minimum-degree} the property $s$-graph $P =\PG{G}{\DegF{d}{\delta+\eps/2}}{s}$ satisfies $\delta_{r}(P) \geq  (1-1/s^2)  \tbinom{n-r}{s-r}$.
	Let $\P$ be the family of $s$-vertex $t$-graphs $K_t(R)$ with $R \in \DegF{d}{\delta+\eps/2}$.
	By definition of $\delta$, $\P$ admits a Hamilton framework.
	Thus we can finish by applying \cref{thm:framework-bandwidth}.
\end{proof}

\subsection{Vertex-degree in $3$-uniform graphs} \label{sec:vertex-deg-3-graph}

In this section, we show \cref{pro:5/9-framework}.
The purpose of the following exposition is to illustrate the ideas for finding Hamilton frameworks.
A more comprehensive treatment of the subject follows in \cref{sec:natural-frameworks}, where \cref{pro:5/9-framework} is obtained again as a corollary of \cref{thm:hamilton-cycle-k-d-small}.

We begin by gathering some tools to work with connectivity, matchings, and cycles, starting with the following characterisation of \tight connectivity in $k$-graphs.
(The easy proof is left to the reader.)

\begin{observation}\label{obs:tight-connectivity}
	A $k$-graph is \tightly connected if and only if every two edges are on a common closed $k$-uniform \tight walk.\qed
\end{observation}

We also need the following lemma due to Alon, Frankl, Huang, Rödl, Ruciński and Sudakov~\cite{AFH+12}, whose proof we include for the sake of completeness.
For a $d$-set $S \subset V(G)$, we denote by $L(S)$ the \emph{link graph} of $S$ in~$G$, which is the $(k-d)$-graph on vertex set $V(G) \sm S$ with an edge $X$ whenever $X \cup S \in G$.
If $S = \{v\}$ is a singleton, we write $L(v)$ instead of $L(S)$.

\begin{lemma}\label{lem:matching-link-graph}
	Let $G$ be a $3$-graph on $n$ vertices such that $L(v)$ contains a fractional matching of size $n/3$ for every $v \in V(G)$.
	Then $G$ has a perfect fractional matching.
\end{lemma}

A \emph{fractional cover} in a $k$-graph $G$ is a function $\bvec{c}\colon E(G) \to [0,1]$  such that ${\sum_{v \in e} \bvec c(v) \geq 1}$ for every $v \in V(G)$.
The \emph{size} of $\vecb c$ is $\sum_{v \in V} \vecb c (v)$.
By linear programming duality, the maximum size of a fractional matching in $G$ is equal to the minimum size of a fractional cover in $G$.
Hence, \cref{lem:matching-link-graph} is equivalent to the following result.

\begin{lemma}\label{lem:cover-link-graph}
	Let $G$ be a $3$-graph on $n$ vertices such that, for every $v \in V(G)$, every fractional cover $L(v)$ has size at least $n/3$.
	Then every fractional cover in $G$ has size at least $n/3$.
\end{lemma}

\begin{proof}
	For the sake of contradiction, suppose that there is a fractional cover $\bvec{c}: V \to  [0,1]$ of size less than $n/3$.
	Let $w_0$ be a vertex of minimum value under $\bvec{c}$, and let $c_0 = \bvec{c}(w_0)$.
	Note that $c_0 <1/3$.
	(Otherwise, $\bvec c$ would have size $n/3$.)
	We define
	\begin{equation*}
		\bvec{c}'(v) = \frac{\bvec{c}(v)-c_0}{1-3 c_0}.
	\end{equation*}
	Observe that $\sum_{v \in V} \bvec{c}'(v)$ is still less than $n/3$, since
	\begin{equation*}
		(1-3 c_0) \sum_{v \in {V}}  \vecb c'(v) =\sum_{v \in {V}} (\bvec{c}(v) - c_0) < \frac{n}{3} - c_0 n = (1- 3 c_0) \frac{n}{3}.
	\end{equation*}

	Since $\bvec{c}'(w_0)=0$ and $\bvec{c}$ is a fractional cover, we have, for any $e \in L(w_0)$,
	\begin{equation*}
		\sum_{v \in e} \bvec c'(v) =   \sum_{v \in e \cup \{w_0\}} \bvec c'(v) =\frac{\sum_{v \in e \cup \{w_0\}} \bvec c(v) -c_0}{1-3 c_0} \geq 1.
	\end{equation*}
	Hence, $\bvec c'$ is a fractional cover of $L(w_0)$.
	Crucially, we also have that $\bvec c'(w_0)=0$.
	So $L(w_0)$ has a cover of size less than $n/3$,
	a contradiction.
\end{proof}

The proof of \cref{pro:5/9-framework} relies on the following lemma, which was originally shown by Cooley and Mycroft~\cite[Lemma 3.2]{CM17}.
We remark that the condition \ref{itm:cooleymycroft-triangle} is not present in the original statement but follows by inspecting the proof and applying Mantel's theorem (see also~\cite[Lemma 3.7]{LS22}).
For completeness, we give the details in \cref{sec:cooley-mycroft-lemma}.

\begin{lemma}  \label{lem:cooley-mycroft}
	For every $\eps > 0$, there is $n_0$ such that the following holds.
	Suppose that $L$ is a $2$-graph on $n$ vertices with $e(L) \geq (5/9 + \eps) \binom{n}{2}$.
	Let $C \subset L$ be a component with a maximum number of edges.
	Then
	\begin{enumerate}[label=\textnormal{(\roman*)}]
		\item \label{itm:cooleymycroft-edge-density} {$e(C) > 4/9\binom{n}{2}$,}
		\item \label{itm:cooleymycroft-fractional-matching} $C$ has a matching of size $n/3$ and
		\item \label{itm:cooleymycroft-triangle} $C$ has a triangle.
	\end{enumerate}
	Moreover, for another $2$-graph $L'$ with $e(L') \geq 5/9 \binom{n}{2}$, $V(L') = V(L)$ and largest component $C' \subset L'$, it follows that
	\begin{enumerate}[label=\textnormal{(\roman*)}]\addtocounter{enumi}{4}
		\item \label{itm:cooleymycroft-edge-in-common} $C$ and $C'$ have at least one edge in common.
	\end{enumerate}
\end{lemma}

\begin{proof}[Proof of \cref{pro:5/9-framework}]
	For $\eps \gg 1/n$, let $\P$ be the family of $n$-vertex $3$-graphs in $\DegF{1}{5/9+\eps}$.
	Consider $G \in \cP$.
	We select a $2$-uniform component $C_v \subset L(v)$ with a maximum number of edges, for each vertex $v \in V(G)$.
	So in particular, the graphs $C_v$ satisfy the outcome of \cref{lem:cooley-mycroft}.
	Let $F_v = \{e \cup \{v\}\colon e \in C_v\} \subseteq G$ for all $v \in V(G)$,
	and let $C \subset G$ be the spanning subgraph whose edges are formed by the union of $F_v$ over all $v \in V(G)$.
	Finally, set $F(G) = C$.
	We claim that the family $F = \{ F(G) \colon G \in \P \}$ is a Hamilton framework for $\P$.
	To this end, we verify conditions \ref{itm:hf-connected}--\ref{itm:hf-intersecting}.

	We first show that $C$ is \tightly connected.
	Note that for a fixed $v \in V(G)$, the edges $F_v$ are in a \tight component of $G$, since $C_v$ is connected by choice.
	Moreover, for another $u \in V(G)$, we can find a common edge $e \in C_v \cap C_v$ by \cref{lem:cooley-mycroft}(v).
	It follows that $F_u$ and $F_v$ are in the same \tight component.
	Hence $C$ is \tightly connected, and  part~\ref{itm:hf-connected} holds.

	For the matching property, we can simply combine \cref{lem:matching-link-graph} with \cref{lem:cooley-mycroft}\ref{itm:cooleymycroft-fractional-matching} to find a perfect fractional matching in $C$.
	In other words, part~\ref{itm:hf-matching} holds.

	Now we show that $C$ contains a closed walk of order $1 \bmod 3$.
	For each $v \in V(G)$, we have that $|F_v| = |C_v| > \frac{4}{9} \binom{n}{2}$, by \ref{itm:cooleymycroft-edge-density}.
	Since $\sum_{v \in V(G)}|F_v| > (4n/9)\binom{n}{2} > \binom{n}{3} \geq |C|$, there are distinct $x, y \in V(G)$ such that $F_x \cap F_y$ have non-empty intersection.
	Hence, there exists $z \in V$ distinct from $x$ and $y$ such that
	$xyz$ is an edge in $C$, and such that $yz \in C_x$ and $xz \in C_y$.
	Since $C_x$ contains an odd cycle by \ref{itm:cooleymycroft-triangle}, there exists $k \geq 1$ and a walk $W_x = w_0 w_1 \dotsb w_{2k} w_{2k+1}$ in $C_x$ such that $w_0w_1 = yz$ and $w_{2k} w_{2k+1} = zy$.
	Similarly, there exists $r \geq 1$ and a walk $W_y = u_0 u_1 \dotsb u_{2r} u_{2r+1}$ in $C_y$ which begins with $xz$ and ends with $zx$.
	Note then that
	\[ P_x = yz x  w_2 w_3 x w_4 w_5 x \dotsb x zy \]
	and
	\[ P_y = xz y u_2 u_3 y u_4 u_5 y \dotsb y zx \]
	are both \tight walks in $C$, each one containing a number of vertices which is equal to $2$ modulo $3$.
	Also the cyclic concatenation $P_x P_y$ is a closed \tight walk whose number of vertices is $1 \bmod 3$, as required.
	Hence,  part~\ref{itm:hf-odd} holds.

	We finish by verifying the consistency property.
	Suppose that there is an $(n+1)$-vertex $3$-graph $J$ with $x,y \in V(J)$ such that $G=J-x$ and $G'=J-y$, where $G'$ is another $n$-vertex $k$-graph with $\delta_1(G') \geq (5/9 + \eps) \binom{n-1}{2}$.
	For $v \in V(J-x-y)$, consider components $C_v\subset L_G(S)$ and $C'_v \subset L_{G'}(S)$ each with a maximum number of edges.
	Then $C_v-x-y$ and $C'_v-x-y$ have each more than $(1/4)\binom{n-1}{2}$ edges, and thus intersect in a vertex.
	It follows that $F(G) \cup F(G')$ is \tightly connected, as desired.
\end{proof}

\subsection{Power of Hamilton cycles}\label{sec:applications-proofs-powers}

Now we come to powers of hypergraph Hamilton cycles and prove \cref{thm:power-hamilton-map-codegree}.

We require the following three facts, that can be derived from former work~\cite[Lemmata 3.6 and 4.3]{PSS23}.
Recall that $f_k(t) = 1 - 1/\left(\tbinom{t-1}{k-1}+\tbinom{t-2}{k-2}\right)$.
	{We also abbreviate minimum codegrees to $\delta(G) = \delta_{k-1}(G)$.}

\begin{lemma} \label{lemma:crucialbound}
	For $2 \leq k \leq t \leq n$, let $G$ be a $k$-graph on $n$ vertices with $\delta(G) > f_k(t)n$.
	Then for any $D_1, D_2 \in K_{t-1}(G)$ with $|D_1 \cap D_2| = t-2$,
	there is a vertex $v$ such that $D_1 \cup \{v\}, D_2 \cup \{v\}\in K_{t}(G)$.
\end{lemma}
\begin{proof}
	We count the number of $(k-1)$-sets completely contained in either $C_1$ or $C_2$.
	Each of $C_1, C_2$ has~$\binom{t-1}{k-1}$ such sets of which~$\binom{t-2}{k-1}$ are shared,
	so by inclusion-exclusion and by Pascal's identity, the desired number is $2\binom{t-1}{k-1} - \binom{t-2}{k-1} = \binom{t-1}{k-1}+\binom{t-2}{k-2}$.
	So, the number of possible choices for $v$ is at least
	\begin{equation*}
		n - \left(\tbinom{t-1}{k-1}+\tbinom{t-2}{k-2}\right)\left (n- \delta(G)\right)
		> 0. \qedhere
	\end{equation*}
\end{proof}

For a $k$-graph $G$, we denote by $\partial G$ the $(k-1)$-sets of $V(G)$ that are contained in an edge of $G$.

\begin{lemma}\label{lem:powers-connectivity}
	For $2 \leq k \leq t \leq n$, let $G$ be a $k$-graph on $n$ vertices with $\delta(G) > f_k(t)n$.
	Then $K_t(G)$ is \tightly connected.
\end{lemma}
\begin{proof}
	We go by induction on $t$.
	For the base case $t = k$, we have $f_k(t) = 1/2$.
	Suppose that, for the sake of contradiction, there are distinct \tight components $E,F \subset G$.
	Consider edges $e \in E$, $f \in F$ such that $d = |e \cap f|$ is maximal.
	Since $\delta(G) > n/2$, there is a vertex such that $\{v\} \cup e$ and $\{v\} \cup f$ are both in $G$.
	Note that these two edges contain $(k-1)$-sets $e',f' \in \partial G$ with $|e' \cap f'| > d$.
	By maximality, $e'$ and $f'$ are in the same \tight component.
	But this contradicts the fact that $E$ and $F$ are distinct.

	Now suppose that $t>k$.
	Consider two $(t-1)$-edges $e,f \in \partial(K_t(G)) \subset K_{t-1}(G)$.
	Note that $f_k(t) \geq f_k(t-1)$.
	By the induction assumption, there is a $(t-1)$-uniform closed \tight walk $W \subset K_{t-1}(G)$ that contains both $e$ and~$f$.
	We then apply \cref{lemma:crucialbound} along the consecutive edges of $W$ to find that $K_t(G)$ is \tightly connected.
\end{proof}

A \emph{neighbour} of a set $X$ of vertices in a hypergraph $G$ is any set $Y$ such that $X \cup Y$ is an edge of $G$.

\begin{lemma}\label{lem:powers-odd-cycle}
	For $2 \leq k \leq t \leq n$, let $G$ be a $k$-graph on $n$ vertices with $\delta(G) > f_k(t)n$.
	Then $K_t(G)$ has a closed \tight walk of order $-1 \bmod t$.
\end{lemma}
\begin{proof}
	Set $K = K_t(G)$.
	By \cref{lemma:crucialbound}, for any $D_1, D_2 \in K_{t-1}(G)$ with $|D_1 \cap D_2| = t-2$, there is a vertex $v$ such that $D_1 \cup \{v\}, D_2 \cup \{v\}\in K$.
	So for $t=2$, we obtain a triangle and are done.
	This allows us to assume that $t \geq 3$.
	Let $A=\{a_1,\dots,a_t\}$ be an edge in $K$.
	By the above, $A \sm \{a_1\}$ and $A \sm \{a_i\}$ share a neighbour $b_i$ for each $2 \leq i \leq   t$.
	We define $\cA_0=(a_1,\dots,a_t)$, $\cA_1=(b_2,a_1)(a_3,a_4,\dots,a_t)$ and, for $2 \leq i \leq t-2$,
	\begin{equation*}
		\cA_i=	(a_2,\dots,a_{i})(b_{i+1},a_1)(a_{i+2},\dots,a_{t}).
	\end{equation*}
	Finally, let $\cA_{t-1}=(a_2,\dots,a_{t-1},b_{t-1})$.
	Note that, $\cA_0,\dots,\cA_{t-2}$ are each $t$-tuples, while $\cA_{t-1}$ is an $(t-1)$-tuple.
	We claim that the concatenation $W =\cA_0\cA_1\dotsm\cA_{t-1}$ is a \tight walk.
	This follows because the $t-1$ vertices to the left of each $b_{i+1}$ consist of $A\sm \{a_1\}$ and the $t-1$ vertices to the right of each $b_{i+1}$ consist of $A\sm \{a_{i}\}$.
	We conclude by noting that $W$ has order $(t-1)t + (t-1) \equiv -1 \bmod t$, as desired.
\end{proof}

Finally, we require a result of  Lo and Markström~\cite[Theorem 1.6]{LM13b}.

\begin{theorem}\label{thm:lo-markstrom}
	For $k, t$ and $\eps >0$, there is $n_0$ such that the following holds.
	Let $G$ be a $k$-graph on $n \geq n_0$ vertices with $\delta_{k-1}(G) \geq (1 - \binom{t-1}{k-1}^{-1} + \eps)n$.
	Then $K_t(G)$ contains a perfect fractional matching.
\end{theorem}

\begin{proof}[Proof of \cref{thm:power-hamilton-map-codegree}]
	Given $t \geq k \geq 2$ and $\eps > 0$, let $n$ be sufficiently large.
	Let $\P$ be the family consisting of $K_t(G)$ for all $n$-vertex $k$-graphs $G$ where $\delta_{k-1}(G) \geq \left( f_k(t) + \eps \right) n$.
	Then $\P$ has a Hamilton framework.

	Given $k,t, \eps$, choose $n_0$ large enough.
	For each $K \in \P$, we set $F(K) = K$.
	We show that this is a Hamilton framework.
	Indeed, the conditions \ref{itm:hf-connected}--\ref{itm:hf-odd}  of \cref{def:hamilton-framework} follow immediately from \cref{lem:powers-connectivity,lem:powers-odd-cycle,thm:lo-markstrom}.
	For the intersecting property, suppose that there is an $(n+1)$-vertex $t$-graph $J$ with $x,y \in V(J)$ such that $K=J-x$ and $K'=J-y$, where $G' \in \P$ is another $n$-vertex $t$-graph.
	Recall that $F(K)=K$ and $F(K')=K'$.
	Since $F(K)-x-y$ and $F(K')-x-y$ have a (common) edge, it follows that $F(G) \cup F(G')$ is \tightly connected, as desired.
\end{proof}

\subsection{Natural frameworks}\label{sec:natural-frameworks}

One of the principal difficulties in finding \tight Hamilton cycles under $d$-minimum degree conditions is that beyond codegrees, the host graphs are no longer guaranteed to be \tightly connected.
In our recent work~\cite{LS22}, we proposed to tackle \cref{con:Han-Zhao-are-best-possible} via the study of $(k-d)$-uniform link graphs.
To find a $k$-uniform \tight component, the idea is to select a $(k-d)$-uniform \tight component of maximum density in each link graph.

To formalise this, consider a $k$-graph $G$.
Recall that for a $d$-set $S \subset V(G)$, we denote by $L(S)$ the {link graph} of $S$ in $G$, which is the $(k-d)$-graph on vertex set $V(G) \sm S$ with an edge $X$ whenever $X \cup S \in G$.
A \emph{$d$-vicinity} $C$ assigns to each $d$-set $S \subset V(G)$ a subgraph $C(S) \subset L(S)$.
The subgraph $C(G) \subset G$ \emph{generated} by ${C}$ is the $k$-graph on $V(G)$ whose edges consist of $S \cup X$ for all $X \in C(S)$ and $d$-sets $S \subset V(G)$.
A \emph{natural} $d$-vicinity $C$ of $G$ is formed by taking for each $d$-set $S \subset V(G)$ a \tight component $C(S) \subset L(S)$ with a maximum number of edges.
For a $k$-graph family $\P$, we say that $F$ is a \emph{natural framework} if for each $G \in \P$, the subgraph $F(G) \subset G$ is generated by a natural vicinity of $G$.

We define $\delta_d^{\natfw}(k)$ as the
minimum $d$-degree threshold for natural frameworks.
More precisely, $\delta_d^{\natfw}(k)$ is the infimum $\delta \in [0,1]$ such that for every $\eps > 0$ and $n$ sufficiently large, every natural framework of the $n$-vertex $k$-graphs in $\DegF{d}{\delta}$ is a Hamilton framework.

We obtain the following upper bound from \cref{thm:framework-bandwidth-simple}
\begin{corollary}\label{cor:natural-frameworks}
	We have $\delta_d^{\hf}(k) \leq \delta_d^{\natfw}(k)$ for all $1 \leq d < k$.
\end{corollary}

We believe that this is tight.

\begin{restatable}{conjecture}{connaturalframework}\label{con:natural-framework}
	We have $\delta_d^{\hf}(k) = \delta_d^{\natfw}(k)$ for all $1 \leq d < k$.
\end{restatable}

Owing to the results of Rödl, Ruciński and Szemerédi~\cite{RRS08a}, Cooley and Mycroft~\cite{CM17}, Polcyn, Reiher, Rödl and Schülke~\cite{PRRS21}, Lang, Schacht and Volec~\cite{LSV24}, as well as our own work~\cite{LS22}, the conjecture has thus far been confirmed in the cases $k-d \in \{1, 2, 3\}$.

\begin{lemma}\label{lem:natural-framework-threshold-1}
	We have $\delta_d^{\natfw}(k) = 1/2$ for $k-d = 1$.
\end{lemma}

\begin{lemma}\label{lem:natural-framework-threshold-2}
	We have $\delta_d^{\natfw}(k) = 5/9$ for $k-d = 2$.
\end{lemma}

\begin{lemma}\label{lem:natural-framework-threshold-3}
	We have $\delta_d^{\natfw}(k) = 5/8$ for $k-d = 3$.
\end{lemma}

Note that \cref{lem:natural-framework-threshold-1,lem:natural-framework-threshold-2,lem:natural-framework-threshold-3} confirm \cref{thm:hamilton-cycle-k-d-small}.
To conclude the proofs of the applications, we derive \cref{lem:natural-framework-threshold-1,lem:natural-framework-threshold-2,lem:natural-framework-threshold-3} in the next three subsections from the structural insights of former work.

\subsection{Singleton vicinities}
We begin by setting up a bit of terminology and some related tools.
Recall that a neighbour of a set $X$ of vertices in a hypergraph $G$ is any set $Y$ such that $X \cup Y$ forms an edge of $G$.
Note that in the following discussion we sometimes allow for $1$-uniform graphs.

\begin{definition}[Switcher]
	For $k\geq 1$, a \emph{switcher} in a $k$-graph $G$ is an edge $A$ of $G$ with a distinguished \emph{central vertex} $a \in A$ such that, for every $b \in A$, the $(\ell-1)$-sets $A \sm \{a\}$ and $A \sm \{b\}$ share a neighbour in $G$.
\end{definition}

\begin{definition}[Arc]
	For $k\geq 2$, let $G$ be a $k$-graph with a $d$-vicinity $C$. A tuple of vertices $(v_1,\dots,v_{k+1})$ is an \emph{arc} for $C$, if  $\{v_{d+1},\dots,v_k\} \in C(\{v_1,\dots,v_d\} )$ and $\{v_{d+2},\dots,v_{k+1}\} \in C(\{v_2,\dots,v_{d+1}\} )$.
\end{definition}

\begin{lemma}[Connectivity and divisibility, {\cite[Lemma 2.1]{LS22}}]\label{lem:connectivity}
	For $k\geq 2$, let $G$ be a $k$-graph with a $d$-vicinity~$C$.
	Suppose that, for all $d$-sets $S,S' \subset V(G)$:
	\begin{enumerate}[label=\textnormal{(V\arabic*)}]
		\item  $C(S)$ is \tightly connected,
		\item  $C(S)$ and $C(S')$ intersect and
		\item  $L$ has a switcher and the vicinity $C$ has an arc.
	\end{enumerate}
	Then the graph $H \subset G$ generated by $\cC$ satisfies:
	\begin{enumerate}[label=\textnormal{(F\arabic*)}] \addtocounter{enumi}{1}
		\item   $H$ is \tightly connected and
		\item   $H$ contains a \tight closed walk of order $1 \bmod k$.
	\end{enumerate}
\end{lemma}

\begin{lemma}[Space, {\cite[Proposition 1.1]{AFH+12}}]\label{lem:space}
	For $k\geq 1$, let $G$ be a $k$-graph on $n$ vertices such that every link $(k-d)$-graph has a fractional matching of size at least $n/k$.
	Then the graph $G$ has a perfect fractional matching.
\end{lemma}

Given this, \cref{lem:natural-framework-threshold-1} follows immediately.

\begin{proof}[Proof of \cref{lem:natural-framework-threshold-1}]
	Let $n \geq 10k$.
	Denote by $\P$ the $n$-vertex $k$-graphs of $\DegF{k-1}{1/2}$.
	Let $F$ be a natural framework of $\P$.
	Since $1$-uniform graphs are trivially \tightly connected, the choice of $F$ is trivial, and we have $F(G) = G$ for all $G \in \P$.

	It is then straightforward to check that \cref{lem:connectivity,lem:space} yield properties \ref{itm:hf-connected}--\ref{itm:hf-odd} of \cref{def:hamilton-framework}.
	(In fact, this can also be done directly without much effort.)
	For the intersecting property, suppose that there is an $(n+1)$-vertex $t$-graph $J$ with $x,y \in V(J)$ such that $G=J-x$ and $G'=J-y$, where $G,G' \in \P$.
	Recall that $F(G)=G$ and $F(G')=G'$.
	Note that $F(G)-x-y$ and $F(G')-x-y$ have a (common) edge, since $n\geq 10k$.
	It follows that $F(G) \cup F(G')$ is \tightly connected, as desired.
\end{proof}

\subsection{Pair vicinities}

Next we prove \cref{lem:natural-framework-threshold-2}.
The proof relies on \cref{lem:cooley-mycroft}.
We also need the following lemma for finding arcs.

\begin{lemma}[Arc, {\cite[Proposition 3.6]{LS22}}]\label{prop:arc}
	Let $1\leq d \leq k-1$, $t \in \NATS$ and $\delta,\eps >0$ with $\delta \gg \eps \gg 1/t$ and $\delta+\delta^{1-1/(k-d)} > 1+\eps$.
	Let $R$ be a $k$-graph on $t$ vertices with a subgraph $H$ that is generated by a $d$-vicinity $\cC$.
	Suppose that each $C_S \in \cC$ has edge density at least $\delta$.
	Then $\cC$ admits an arc.
\end{lemma}

\begin{proof}[Proof of \cref{lem:natural-framework-threshold-2}]
	Given $\eps > 0$, let $n$ be sufficiently large.
	Denote by $\P$ the $n$-vertex $k$-graphs of $\DegF{k-2}{5/9+\eps}$.
	Let $F$ be a natural framework of $\P$.
	Consider $G \in \P$, and let $C$ be the natural vicinity of $F(G)$.

	Let us first verify that $F(G)$ has properties \ref{itm:hf-connected}, \ref{itm:hf-matching} and \ref{itm:hf-odd} of \cref{def:hamilton-framework}.
	Note that for each $(k-2)$-set $S\subset V(G)$, the $2$-graph $C(S) \subset L_G(S)$ satisfies conditions~\ref{itm:cooleymycroft-edge-density}--\ref{itm:cooleymycroft-triangle} of \cref{lem:cooley-mycroft}.
	We remark that a triangle corresponds to a $2$-uniform switcher.
	Moreover, since  $4/9+(4/9)^{1-1/2}=1+1/9$, it follows by Proposition~\ref{prop:arc} that $C$ has an arc.
	So by \cref{lem:connectivity}, $F(G)$ is \tightly connected and contains a closed walk of order $1 \bmod k$.
	Finally, $F(G)$ has a perfect fractional matching by \cref{lem:space}.

	We finish by verifying the intersecting property.
	Suppose that there is an $(n+1)$-vertex $3$-graph $J$ with $x,y \in V(J)$ such that $G=J-x$ and $G'=J-y$, where $G'$ is another $n$-vertex $k$-graph in $\P$.
	For a $(k-2)$-set $S \subset V(J-x-y)$ consider components $C(S)\subset L_G(S)$ and $C'(S) \subset L_{G'}(S)$, where $C'$ is the natural vicinity of $F(G')$.
	Then $C(S)-x-y$ and $C'(S)-x-y$ have each more than $(1/4)\binom{n-1}{2}$ edges, and thus intersect in a vertex.
	It follows that $F(G) \cup F(G')$ is \tightly connected, as desired.
\end{proof}

\subsection{Triple vicinities}

Finally, we show \cref{lem:natural-framework-threshold-2}.
Let us begin with a few technical observations.

\begin{lemma}[{\cite[Proposition 3.2]{LS22}}] \label{prop:large-component}
	For an $\ell$-graph $L$, let $C$ be the \tight component that maximises $e_{\ell}(C) / e_{\ell-1}(C)$.
	Then $$\frac{e_\ell(C)}{e_{\ell-1}(C)}  \ge \frac{e_\ell(L)}{e_{\ell-1}(L)}.$$
	In particular, if $\delta, \nu, \nu'$ denote the edge densities of $L, C$ and $\partial(C)$ respectively, then $\nu /\nu'\ge \delta $.
\end{lemma}
\begin{proof}
	Let $C_1,\dots,C_s$ be the \tight components of $L$.
	Fix $1 \leq q \leq s$ which maximises $e_{\ell}(C_q) / e_{\ell-1}(C_q)$.
	By the definition of \tight components, we have $\sum_{i=1}^s e_\ell(C_i) = e_\ell(L)$ and $\sum_{i=1}^s e_{\ell-1}(C_i) = e_{\ell-1}(L)$.
	It follows that
	\begin{equation*}
		e_\ell(L) = \sum_{i=1}^s e_\ell(C_i) \leq \frac{e_\ell(C_q)}{e_{\ell-1}(C_q)} \sum_{i=1}^s e_{\ell-1}(C_i) = \frac{e_\ell(C_q)}{e_{\ell-1}(C_q)}  e_{\ell-1}(L).
	\end{equation*}
	We obtain the second part by bounding the edge density of $\partial_{\ell-1}(L)$ {by~$1$}.
\end{proof}

\begin{fact}\label{obs:density}
	For $n$ sufficiently large, let $L$ be a $3$-graph $n$ vertices with $e(G) = \lceil 5/8 \binom{n}{3} \rceil$.
	Let $C \subset G$ be a \tight component with $e(C) \geq 1/2 \binom{n}{3}$.
	Denote the edge densities of $L$, $C$ and $\partial (C)$ by $\delta$, $\nu$ and $\nu'$ respectively.
	Then $\nu/ \nu' \geq \delta$.
\end{fact}
\begin{proof}
	Let $C_1 \subset G$ be the \tight component that maximises $e_{\ell}(C_1) / e_{\ell-1}(C_1)$.
	Denote by $\nu_1, \nu'_1$ the edge densities of $C_1$ and $\partial(C_1)$ respectively.
	By \cref{prop:large-component}, we have $\nu_1 / \nu_1'\ge \delta $.
	If $C = C_1$, we are done.
	So suppose otherwise.
	But then $\nu_1 \leq 1/8$ due to the assumption that  $e(G) = \lceil 5/8 \binom{n}{3} \rceil$.
	By the Kruskal--Katona theorem, we have $\nu'_1 \geq \nu_1^{2/3} - o(1)$.
	So $$ \frac{\nu_1}{\nu'_1} \leq \nu_1^{1/3} + o(1) \leq \frac 1 2 + o(1) < \frac{5}{8} \leq \delta,$$ which is absurd.
\end{proof}

The proof of the following result uses a nice convexity argument of Reiher, Rödl, Ruciński, Schacht and Szemerédi~\cite[Lemma 6.2]{RRR19}.

\begin{lemma}[Switcher, {\cite[Proposition 3.5]{LS22}}]\label{prop:switcher}
	For $1/t \ll \mu$, let $L$ be an $\ell$-graph on $t$ vertices with edge density at least $\delta +\mu$, where $\delta = \ell/(\ell + 2 \sqrt{\ell-1})$.
	Consider a vertex spanning subgraph $C\subset L$ and denote the edge densities of $C$ and $\partial (C)$ by $\nu$ and $\nu'$ respectively.
	Suppose that $\nu/ \nu' \geq \delta +\mu $.
	Then $C$ has a switcher.
\end{lemma}

Finally, we require the following result due to Lang, Schacht and Volec~\cite{LSV24}.

\begin{theorem}{\cite[Lemmata 2.3 and 2.4]{LSV24}}\label{thm:5/8}
	For every $\eps$, there is $n_0$ such that the following holds.
	Suppose that $G$ is a $3$-graph on $n \geq n_0$ vertices with $e(G) \geq (5/8 + \eps) \binom{n}{3}$.
	Then there is a subgraph $C\subset G$ which is
	\begin{enumerate}[\upshape (1)]
		\item \label{itm:problem-erdos-gallai-connected} \tightly connected,
		\item \label{itm:problem-erdos-gallai-matching}  has a matching of size $n/4$ and
		\item \label{itm:problem-erdos-gallai-dense}  has more than $1/2 \binom{n}{\ell}$ edges.
	\end{enumerate}
\end{theorem}

\begin{proof}[Proof of \cref{lem:natural-framework-threshold-3}]
	Given $\eps > 0$, let $n$ be sufficiently large.
	Denote by $\P$ the $n$-vertex $k$-graphs of $\DegF{k-3}{5/8+\eps}$.
	Let $F$ be a natural framework of $\P$.
	Consider $G \in \P$, and let $C$ be the natural vicinity of $F(G)$.

	Let us first verify that $F(G)$ has properties \ref{itm:hf-connected}, \ref{itm:hf-matching} and \ref{itm:hf-odd} of \cref{def:hamilton-framework}.
	Note that for each $(k-3)$-set $S\subset V(G)$, the $3$-graph $C(S) \subset L_G(S)$ satisfies conditions~\ref{itm:problem-erdos-gallai-connected}--\ref{itm:problem-erdos-gallai-dense} of \cref{thm:5/8}.
	Note that $3/(3+2\sqrt{2}) < 0.52 < 5/8$.
	By \cref{obs:density,prop:switcher}, there is a switcher in $C(S)$.
	(To attain the edge density condition of \cref{obs:density}, we may delete some edges outside of $C(S)$.)
	Moreover, since  $1/2+(1/2)^{1-1/3} < 1.2$, it follows by Proposition~\ref{prop:arc} that $C$ has an arc.
	So by \cref{lem:connectivity}, $F(G)$ is \tightly connected and contains a closed walk of order $1 \bmod k$.
	Finally, $F(G)$ has a perfect fractional matching by \cref{lem:space}.

	We finish by verifying the intersecting property.
	Suppose that there is an $(n+1)$-vertex $4$-graph $J$ with $x,y \in V(J)$ such that $G=J-x$ and $G'=J-y$, where $G'$ is another $n$-vertex $k$-graph in $\P$.
	For a $(k-3)$-set $S \subset V(J-x-y)$ consider components $C(S)\subset L_G(S)$ and $C'(S) \subset L_{G'}(S)$, where $C'$ is the natural vicinity of $F(G')$.
	Then $C(S)-x-y$ and $C'(S)-x-y$ have each more than $(1/2)\binom{n-1}{2}$ edges, and thus intersect in a pair of their shadows.
	It follows that $F(G) \cup F(G')$ is \tightly connected, as desired.
\end{proof}

\section{A directed setup}\label{sec:directed-setup}

In this section, we formulate directed versions of our main results and then derive \cref{thm:framework-bandwidth} from them.
Since our guest structures are based on linear vertex orderings, it is natural to carry out the analysis itself in a directed setting.
We also replace the notion of $k$-graphs with $k$-bounded graphs, which are hypergraphs that allow edges of uniformity up (but not restricted) to $k$.
This is particularly useful to track the presence of specific tuples that witness connectivity in a robust way.
We note that the transition from the undirected to the directed setting, formalised in \cref{pro:framework-undirected-to-directed}, is not trivial and an important step towards the proof of \cref{thm:framework-bandwidth}.
Moreover, working with $k$-bounded directed graphs also turns out to be an appropriate interface for passing on the property of robust Hamiltonicity to other fields of study such as random graphs, which is discussed after \cref{thm:framework-connectedness-robust}.
The concepts we have seen so far generalise straightforwardly to this setting.
For the sake of completeness we spell out the details.

A \emph{$k$-bounded directed hypergraph} (or \emph{$[k]$-digraph} for short) $G$ consists of a set of vertices $V(G)$ and a set of edges $E(G) \subset \bigcup_{i \in [k]} V(G)^i$, where edges do not have repeated vertices.
We denote by $G^{(i)} \subset G$ the vertex-spanning subgraph that contains the edges of uniformity $i$.
A (directed) \emph{$k$-uniform path (cycle)} in $G$ comes with a linear (cyclical) ordering of its vertex set such that every subsequence of $k$ consecutive vertices forms a (directed) edge in $G$.
A \emph{homomorphism} from a $k$-digraph $C$ to $G$ is a function $\phi \colon V(C) \to V(G)$ that maps directed edges to directed edges.
We say that $W \subset G$ is a \emph{(closed) walk} if $W$ is the image of a homomorphism of a path (cycle) $P$.
The \emph{order} of $W$ is the order of $P$.

\subsection{Necessary conditions}\label{sec:directed-setupt-necessary-conditions}

In the following, we recover directed versions of the concepts introduced in \cref{sec:necessary-conditions}.
A $[k]$-digraph $G$ is \emph{\tightly connected} if every two edges in $G^{(k)}$ are on a common closed $k$-uniform \tight walk.
{A \emph{\tight component} is a maximal set of $k$-edges which form a \tightly connected subgraph.}

{A $[k]$-digraph is \emph{shift-closed} if for every $k$-edge $(x_1, \dotsc, x_k) \in G^{(k)}$, the cyclic shifts $(x_2, \dotsc ,x_{k}, x_1)$, $(x_3, \dotsc, x_{k}, x_1, x_2),$ and so on also belong to $G^{(k)}$.
Note that if a $[k]$-digraph is shift-closed then every $e \in G^{(k)}$ belongs to some \tight component, and for every $f \in G^{(k-1)}$, the $k$-edges which contain $f$ as a subsequence all belong to the same \tight component.}

{Given a shift-closed $[k]$-digraph} and an edge $e \in G^{(k-1)}$, we denote by $\tc(e)$ the \tight component comprising all $k$-edges which contain $e$.
The \emph{(directed) \tight adherence} $\adh(G) \subset G^{(k)}$ is obtained by taking the union of the \tight components $\tc(e)$ over all $e \in G^{(k-1)}$.

\begin{definition}[Connectivity]
	Let $\dcon$ be the set of {shift-closed} $[k]$-digraphs~$G$ with $k\geq2$ such that  $\adh(G)$ is a single vertex-spanning \tight component.
\end{definition}

A \emph{fractional matching} in a $[k]$-digraph $G$ is a function $\omega\colon E(G) \to [0,1]$ such that $\sum_{e\colon v \in e} \omega (e)\geq 1$ for every vertex $v \in V(G)$.
We say that $\omega$ is \emph{perfect} if $\sum_{e \in E(G)} \omega (e) = n/k$.

\begin{definition}[Space]
	Let $\dspa$ be the set of $[k]$-digraphs $G$ with $k\geq 2$ such that $\adh(G)$ has a perfect fractional matching.
\end{definition}

\begin{definition}[Aperiodicity]
	Let $\dape$ be the set of $[k]$-digraphs $G$ with $k\geq 2$ such that $G^{(k)}$ contains a closed walk $W$ of order coprime to $k$.
\end{definition}

\subsection{Robustness}
Next, we recover the concepts introduced in \cref{sec:robustness}.
It should be emphasised that the definition of the `property graph' for directed graphs is indeed a graph, not a directed graph.

\begin{definition}[Property graph -- directed]\label{def:property-graph-digraph}
	For a $[k]$-digraph $G$ and a family of $[k]$-digraphs $\P$, the \emph{property graph}, denoted by $\PG{G}{\P}{s}$, is the $s$-graph on vertex set $V(G)$ with an edge $S \subset V(G)$ whenever the induced subgraph $G[S]$ \emph{satisfies}~$\P$, that is $G[S] \in \P$.
\end{definition}

The following definition of robustness corresponds to directed hypergraphs.

\begin{definition}[Robustness, directed]\label{def:robustenss-detailed}
	For a family of $[k]$-digraphs $\P$, a $[k]$-digraph $G$ \emph{$(\delta,r,s)$-robustly satisfies} $\P$ if the minimum $r$-degree of the property $s$-graph $\PG{G}{  \P   }{s}$ is at least $\delta  \tbinom{n-r}{s-r}$.
\end{definition}

For convenience, we abbreviate $(1-1/s^2,r,s)$-robustness to \emph{$(r,s)$-robustness}.
Moreover, \emph{$s$-robustness} is short for $(r,s)$-robustness with $r=2k$ as in the context of \cref{def:robustenss-detailed}.

\subsection{Sufficient conditions}

Now we are ready to state the directed versions of our main results.
For a $k$-graph $G$, we denote by $\ori{C}(G)$ the $k$-digraph on $V(G)$ obtained by adding all possible orientations of every edge of $G$.
Recall that, for a $k$-graph $G$, we denote by $\partial G$ the $(k-1)$-sets of $V(G)$ that are contained in an edge of $G$.

The next result corresponds to a directed version of \cref{thm:framework-bandwidth}.
Its proof is given in \cref{sec:proof-main-result-bandwidth}.
\begin{theorem}[Directed bandwidth]\label{thm:framework-bandwidth-directed}
	For  all $k$, $t$ and $s$, there are $c$ and $n_0$ such that the following holds.
	Let $G$ be a $k$-graph on $n \geq n_0$ vertices such that the $t$-clique-graph $\hat K = K_t(G)$ admits a $[t]$-digraph $K \subset \ori{C}(\hat K) \cup \ori{C}(\partial \hat K)$ that  $s$-robustly satisfies $\dcon \cap \dspa \cap \dape$.
	Let $H$ be an $n$-vertex blow-up of the $(t-k+1)$st power of a $k$-uniform cycle with clusters of size at most~$(\log \log n)^c$.
	Then $H \subset G$.
\end{theorem}

Next, we turn to a version of our main result that guarantees Hamilton connectedness.
So instead of finding a Hamilton cycle we aim to connect any two suitable vertex tuples with a Hamilton path.

\begin{restatable}[Hamilton connectedness]{definition}{defhamiltonconnectedness}\label{def:hamilton-connectedness}
	We denote by $\hamcon$ the set of $[k]$-digraphs $G$, for every $k\geq 2$, such that $G^{(k-1)}$ contains at least two vertex-disjoint edges and $G^{(k)}$ contains a \tight Hamilton $(e,f)$-path for all vertex-disjoint $e,f \in G^{(k-1)}$.
\end{restatable}

The following result guarantees robust Hamilton-connectedness under the same assumptions as in \cref{thm:framework-bandwidth} for $k=t$.
Its proof is given in \cref{sec:robust-ham-connect}.

\begin{restatable}[Hamilton connectedness]{theorem}{thmframeworkconnectednessrobust}\label{thm:framework-connectedness-robust}
	Let $1/k, 1/r \geq {1/s_1} \gg 1/s_2 \geq 1/s_3 \gg 1/n$.
	Let $\P$ be a family of $s_1$-vertex $k$-graphs that admits a Hamilton framework.
	Let $G$ be a $k$-graph on $n$ vertices that satisfies $s_1$-robustly $\P$.
	Then there is an $n$-vertex $[k]$-digraph $G' \subset \ori{C}(G) \cup \ori{C}(\partial G)$ that satisfies $(r,s)$-robustly $\hamcon$ for every $s_2 \leq s \leq s_3$.
\end{restatable}

The outcome of \cref{thm:framework-connectedness-robust} turns out to be a suitable interface for relating results on Hamilton cycles in a dense environment to the so-called random robust setting.
In particular, it follows from our work with Joos~\cite{JLS23} that all results from \cref{sec:background+applications} on Hamilton cycles and their powers are still valid in the random robust setting and moreover admit counting versions.
Moreover, as already mentioned after \cref{con:thresholds}, it follows that \cref{con:thresholds} implies a conjecture of Kelly, Müyesser and Pokrovskiy~\cite[Conjecture 8.1]{KMP23} on optimal thresholds for Hamiltonicity in the random robust setting.

\subsection{Deriving the main results}

To derive \cref{thm:framework-bandwidth} from \cref{thm:framework-bandwidth-directed}, we process the original conditions using the following result.

\begin{proposition}[Transition] \label{pro:framework-undirected-to-directed}
	Let $1/k,\, 1/s_1 \gg 1/s_2 \gg 1/n$.
	Let $\P$ be a family of $s_1$-graphs that admits a Hamilton framework.
	Let $G$ be a $k$-graph on $n$ vertices that $s_1$-robustly satisfies $\P$.
	Then there is an $n$-vertex $[k]$-digraph $G' \subset \ori{C}(G) \cup \ori{C}(\partial G)$ that $s_2$-robustly satisfies $\dcon \cap \dspa \cap \dape$.
\end{proposition}

The proof of \Cref{pro:framework-undirected-to-directed} is given in \Cref{sec:framework-undirected-to-directed}.
Assuming this proposition, we can derive our main result.

\begin{proof}[Proof of \cref{thm:framework-bandwidth}]
	Given $k,t,s$, we obtain $s_2$ from \cref{pro:framework-undirected-to-directed} (with $s$ playing the rôle of $s_1$).
	We obtain $c$ from \cref{thm:framework-bandwidth-directed} (with $s_2,c$ playing the rôles of $s,c$).
	We then obtain $n_0$ from both of these two theorems and also sufficiently large with respect to $c$.

	Let $\P$ be a $k$-graph property that admits a Hamilton framework.
	Let $G$ be a $k$-graph on $n \geq n_0$ vertices such that $K=K_t(G)$ satisfies $s$-robustly $\P$.
	Let $H$ be an $n$-vertex blow-up of the $(t-k+1)$st power of a $k$-uniform cycle with clusters of size at most $(\log \log n)^{c}$.
	By \cref{pro:framework-undirected-to-directed}, there is an $n$-vertex $[t]$-digraph $K' \subset \ori{C}(K) \cup \ori{C}(\partial K)$ that satisfies $s_2$-robustly $\dcon \cap \dspa \cap \dape$.
	By \cref{thm:framework-bandwidth-directed}, we have $H \subset G$.
\end{proof}

\section{Main steps of the proofs}\label{sec:proof-main-result}

The proofs of \cref{thm:framework-connectedness-robust,thm:framework-bandwidth-directed} each come in two parts.
First we show that every $[k]$-digraph $G$ on $n$ vertices which $s$-robustly satisfies a property $\P$ can be covered with a collection of suitable blow-ups $R(\cV)$ whose clusters~$\cV$ are of size $\poly (\log \log n)$ except for one exceptional singleton cluster.
Moreover, the reduced $[k]$-digraphs $R$ satisfy $\P$ and are of order $O(1)$.
This framework is encapsulated in \rf{pro:blow-up-cover}.

Once such a cover is established, the remainder of the argument consists in spelling out the allocation of the desired structure to the blow-ups $R(\cV)$ for $\P=\Del(\dcon \cap \dspa \cap \dape)$, which can be done independently for each blow-up.
This is formalised in \rf{pro:allocation-Hamilton-path} and \rf{pro:allocation-bandwidth-frontend}.

We present results on covering and allocation in the following three subsections and then combine everything to obtain the proofs of \cref{thm:framework-connectedness-robust,thm:framework-bandwidth-directed}.

\subsection{Covering with blow-ups}\label{sec:blow-up-cover-statments}

Consider a \emph{set family} $\cV$, which is by convention of this paper a family of pairwise disjoint sets.
A set $X$ is \emph{$\cV$-partite} if it has at most one vertex in each part of $\cV$.
We say that $\cV$ is \emph{$m$-balanced} if $|V| = m$ for  every $V \in \cV$.
Similarly, $\cV$ is \emph{$(1\pm\eta)m$-balanced} if $(1-\eta)m \leq |V| \leq  (1 + \eta) m$ for  every $V \in \cV$.
Moreover, $\cV$ is \emph{quasi $(1\pm\eta)m$-balanced} if there is an \emph{exceptional} set $V^\ast \in \cV$ with $|V^\ast|=1$ and $\cV \sm \{V^\ast\}$ is $(1\pm \eta)m$-balanced.
We say that another set family $\cW$ \emph{hits} $\cV$ with subfamily $\cW' \subset \cW$ if each part of $\cV$ contains (as a subset) exactly one part of $\cW'$.
(In particular, the parts of $\cW \sm \cW'$ are disjoint from the parts of $\cV$.)

\begin{definition}[Cover]
	Let $F$ be a graph and $V$ be a set.
	We say that $\{\cV^x\}_{x \in V(F)}$ and $\{\cW^{e}\}_{e \in F}$ form an \emph{$(s_1,s_2)$-sized $(m_1,m_2,\eta)$-balanced cover} of $V$ with \emph{shape} $2$-graph $F$ if the following holds:
	\begin{enumerate}[(C1)]
		\item Each $\cV^x$ is a quasi $(1\pm \eta)m_1$-balanced set family of size $s_1$.
		\item Each $\cW^e$ is an $m_2$-balanced set family of size $s_2$. Moreover, the vertex sets $\bigcup \cW^e$ and $\bigcup \cW^f$ are pairwise disjoint for distinct $e,f \in F$.
		\item For each $x \in e$, the set family $\cW^{e}$ hits $\cV^{x} \sm \{V^\ast\}$ with subfamily $\cW^{e}_x \subset \cW^{e}$, where $V^\ast$ is the (exceptional) singleton cluster of $\cV^x$.
		\item The family $\bigcup_{x \in V(F)} \cV^x \cup \bigcup_{xy \in F} \cW^{xy} \sm (\cW^{xy}_x \cup \cW_y^{xy})$ partitions~$V$.
	\end{enumerate}
\end{definition}

Consider a $[k]$-digraph $R$, and let $\cV=\{V_x\}_{x \in V(R)}$ be a set family.
We write $R(\cV)$ for the \emph{blow-up} of $R$ by $\cV$, which is a $[k]$-digraph with vertex set $\bigcup \cV$ and whose edge set is the union of $V_{e(1)} \times \dots \times V_{e(j)}$ over all $j$-edges $e \in R$ with $j \in [k]$.
In this context, we refer to the sets of $\cV$ as \emph{clusters} and to $R$ as the \emph{reduced (di)graph}.
If $R$ is in a family of $[k]$-digraphs $\P$, we call $R(\cV)$ a \emph{$\P$-blow-up}.
We write $R(m)$ for a blow-up $R(\cV)$ whose underlying partition $\cV$ is $m$-balanced.
For a $[k]$-digraph $G$ with $\bigcup\cV \subset V(G)$, we denote by $G[\cV]$ the graph on vertex set $\bigcup \cV$ that contains all $\cV$-partite edges of $G$.

\begin{definition}[Blow-up cover]
	Let $G$ be a $[k]$-digraph.
	Let $\P$ be a family of $[k]$-digraphs.
	An $(s_1,s_2)$-sized $(m_1,m_2,\eta)$-balanced cover formed by $\{\cV^v\}_{v \in V(F)}$ and $\{\cW^{e}\}_{e \in F}$ of $V(G)$ with shape $F$ is called a \emph{$\P$-cover} of $G$ if the following holds:
	\begin{enumerate}[\upshape (B1)]
		\item For each $v \in V(F)$, there is an $s_1$-vertex $R^{v} \in \P$ with $R^{v}(\cV^{v}) = G[\cV^v]$.
		\item For every $e \in F$, there is an $s_2$-vertex $R^e \in \P$ with $R^e(\cW^e) = G[\cW^e]$.
	\end{enumerate}
\end{definition}


\definecolor{ftpink}{RGB}{235, 94, 141}
\definecolor{ftteal}{RGB}{0, 154, 168}
\definecolor{ftgreen}{RGB}{157, 191, 87}
\definecolor{ftblue}{RGB}{32, 143, 206}
\definecolor{ftdarkblue}{RGB}{15, 84, 153}
\definecolor{ftred}{RGB}{206, 49, 64}
\definecolor{ftdarkred}{RGB}{127, 6, 46}

\newcommand{\fc}[1]{}
\pgfdeclarelayer{background}
\pgfdeclarelayer{main}
\pgfdeclarelayer{foreground}
\pgfsetlayers{background,main,foreground}
\begin{figure}
	\centering
	\begin{tikzpicture}[scale=0.65, every node/.style={transform shape}]

		\tikzstyle{bigset} = [minimum size=60,circle,line width=0.5,draw opacity=1, fill=white]
		\tikzstyle{smallset} = [minimum size=16,circle,line width=0.5,draw opacity=1, fill=white]
		\tikzstyle{vertex} = [draw=black!60, fill=black=60!,circle,minimum size = 6,inner sep=0]

		\tikzstyle{edge} = [fill opacity=0.3]
		\tikzstyle{smalledge} = [line width=11, draw opacity=0.3]
		\tikzstyle{bigedge} = [line width=40, draw opacity=0.3]

		\coordinate (x1) at (-2.50,-3.50) {};
		\coordinate (x2) at (-0.00,-2.50) {};
		\coordinate (x3) at (0.50,-5.00) {};
		\coordinate (x4) at (-1.50,-6.00) {};
		\coordinate (x5) at (-2.00,-1.50) {};

		\node[bigset, draw=ftgreen, label={90: \fc{$x_1$}}] (X1) at (x1) {};
		\node[bigset, draw=ftgreen, label={90: \fc{$x_2$}}] (X2) at (x2) {};
		\node[bigset, draw=ftgreen, label={90: \fc{$x_3$}}] at (x3) {};
		\node[bigset, draw=ftgreen, label={90: \fc{$x_4$}}] at (x4) {};
		\node[vertex, ftgreen, draw=ftgreen, label={90: \fc{$x_5$}}] (X5) at (x5) {};

		\begin{pgfonlayer}{background}
			\draw[draw=ftgreen, bigedge] (x1) -- (x2);
			\draw[draw=ftgreen, bigedge] (x2) -- (x3);
			\draw[draw=ftgreen, bigedge] (x3) -- (x4);
			\draw[draw=ftgreen, bigedge] (x3) -- (x1);
			\fill[ftgreen, edge] (X5.west) to (X1.west) -- (X1.east) -- (X5.east) -- cycle;
			\fill[ftgreen, edge] (X5.south) to (X2.south) -- (X2.north) -- (X5.north) -- cycle;
		\end{pgfonlayer}

		\coordinate (y1) at (6.50,-3.00) {};
		\coordinate (y2) at (9.00,-4.00) {};
		\coordinate (y3) at (7.50,-6.00) {};
		\coordinate (y4) at (4.50,-5.00) {};
		\coordinate (y5) at (8.00,-2.00) {};

		\node[bigset, draw=ftgreen, label={110: \textcolor{ftgreen}{\huge $V_1$}}] (Y1) at (y1) {};
		\node[bigset, draw=ftgreen, label={85: \textcolor{ftgreen}{\huge $V_2$}}] (Y2) at (y2) {};
		\node[bigset, draw=ftgreen, label={320: \textcolor{ftgreen}{\huge $V_3$}}] at (y3) {};
		\node[bigset, draw=ftgreen, label={250: \textcolor{ftgreen}{\huge $V_4$}}] at (y4) {};
		\node[vertex, ftgreen, label={90: \textcolor{ftgreen}{\huge$V_5$}}] (Y5) at (y5) {};

		\begin{pgfonlayer}{background}
			\draw[draw=ftgreen, bigedge] (y1) -- (y2);
			\draw[draw=ftgreen, bigedge] (y2) -- (y3);
			\draw[draw=ftgreen, bigedge] (y3) -- (y4);
			\draw[draw=ftgreen, bigedge] (y4) -- (y1);
			\fill[ftgreen, edge] (Y5.north) to (Y1.north) -- (Y1.east) -- (Y5.south) -- cycle;
			\fill[ftgreen, edge] (Y5.east) to (Y2.east) -- (Y2.west) -- (Y5.west) -- cycle;
		\end{pgfonlayer}

		\coordinate (z1) at (14.00,-3.00) {};
		\coordinate (z2) at (13.50,-5.50) {};
		\coordinate (z3) at (16.00,-6.00) {};
		\coordinate (z4) at (16.25,-2.75) {};
		\coordinate (z5) at (15.00,-2.00) {};

		\node[bigset, draw=ftgreen, label={90: \fc{$z_1$}}] (Z1) at (z1) {};
		\node[bigset, draw=ftgreen, label={90: \fc{$z_2$}}] at (z2) {};
		\node[bigset, draw=ftgreen, label={90: \fc{$z_3$}}] at (z3) {};
		\node[bigset, draw=ftgreen, label={90: \fc{$z_4$}}] (Z4) at (z4) {};
		\node[vertex, ftgreen, draw=ftgreen, label={90: \fc{$z_5$}}] (Z5) at (z5) {};

		\begin{pgfonlayer}{background}
			\draw[draw=ftgreen, bigedge] (z1) -- (z2);
			\draw[draw=ftgreen, bigedge] (z2) -- (z4);
			\draw[draw=ftgreen, bigedge] (z3) -- (z4);
			\draw[draw=ftgreen, bigedge] (z3) -- (z1);
			\fill[ftgreen, edge] (Z5.north) to (Z1.north) -- (Z1.east) -- (Z5.south) -- cycle;
			\fill[ftgreen, edge] (Z5.north) to (Z4.north) -- (Z4.west) -- (Z5.south) -- cycle;
		\end{pgfonlayer}

		\node[label={180: \textcolor{ftgreen}{\huge{$R^u(\cV^{u})$}}}] (X) at (1.0,-8) {};
		\node[label={180: \textcolor{ftgreen}{\huge{$R^v(\cV^{v})$}}}] (Y) at (8.5,-8) {};
		\node[label={180: \textcolor{ftgreen}{\huge{$R^w(\cV^{w})$}}}] (Z) at (16.5,-8) {};

		\node[label={180: \textcolor{ftpink}{\Large{$R^{uv}(\cW^{uv})$}}}] (Z) at (4,-1) {};
		\node[label={180: \textcolor{ftteal}{\Large{$R^{vw}(\cW^{vw})$}}}] (Z) at (13,-1) {};

		\coordinate (w1) at (4.25,-4.75) {};
		\coordinate (w2) at (6.25,-2.75) {};
		\coordinate (w3) at (8.75,-3.75) {};
		\coordinate (w4) at (8.00,-6.50) {};

		\coordinate (w5) at (0.25,-2.25) {};
		\coordinate (w6) at (-2.25,-3.25) {};
		\coordinate (w7) at (0.25,-4.75) {};
		\coordinate (w8) at (-1.25,-6.25) {};

		\coordinate (w9) at (3.75,-1.75) {};
		\coordinate (w10) at (2.75,-3.75) {};
		\coordinate (w11) at (2.50,-5.25) {};
		\coordinate (w12) at (3.75,-2.75) {};
		\coordinate (w13) at (2.50,-2.50) {};

		\begin{pgfonlayer}{foreground}
			\node[smallset, draw=ftpink, label={[label distance=2mm] 270: \textcolor{ftpink}{\Large $W_4$}}] at (w1) {};
			\node[smallset, draw=ftpink, label={60: \textcolor{ftpink}{\Large$W_1$}}] at (w2) {};
			\node[smallset, draw=ftpink, label={[label distance=-3mm] 50: \textcolor{ftpink}{\Large$W_2$}}] at (w3) {};
			\node[smallset, draw=ftpink, label={180: \textcolor{ftpink}{\Large$W_3$}}] at (w4) {};
			\node[smallset, draw=ftpink, label={180: \fc{$W_1$}}] at (w5) {};
			\node[smallset, draw=ftpink, label={180: \fc{$W_2$}}] at (w6) {};
			\node[smallset, draw=ftpink, label={180: \fc{$W_3$}}] at (w7) {};
			\node[smallset, draw=ftpink, label={180: \fc{$W_4$}}] at (w8) {};
			\node[vertex, draw=ftpink, fill=ftpink, label={180: \fc{$W_9$}}] (A9) at (w9) {};
			\node[smallset, draw=ftpink, label={180: \fc{$W_{10}$}}] at (w10) {};
			\node[smallset, draw=ftpink, label={180: \fc{$W_{11}$}}] at (w11) {};
			\node[smallset, draw=ftpink, label={180: \fc{$W_{12}$}}] (A12) at (w12) {};
			\node[smallset, draw=ftpink, label={180: \fc{$W_{13}$}}] (A13) at (w13) {};
		\end{pgfonlayer}

		\begin{pgfonlayer}{main}
			\draw[draw=ftpink, smalledge] (w1) -- (w2);
			\draw[draw=ftpink, smalledge] (w3) -- (w2);
			\draw[draw=ftpink, smalledge] (w3) -- (w4);
			\draw[draw=ftpink, smalledge] (w1) -- (w4);
			\draw[draw=ftpink, smalledge] (w1) -- (w10);
			\draw[draw=ftpink, smalledge] (w1) -- (w11);
			\draw[draw=ftpink, smalledge] (w12) -- (w11);
			\draw[draw=ftpink, smalledge] (w12) -- (w2);
			\draw[draw=ftpink, smalledge] (w10) -- (w13);
			\draw[draw=ftpink, smalledge] (w5) -- (w13);
			\draw[draw=ftpink, smalledge] (w5) -- (w7);
			\draw[draw=ftpink, smalledge] (w7) -- (w11);
			\draw[draw=ftpink, smalledge] (w5) -- (w6);
			\draw[draw=ftpink, smalledge] (w8) -- (w7);
			\draw[draw=ftpink, smalledge] (w6) -- (w7);

			\fill[ftpink, edge] (A9.north) to (A13.north) -- (A13.east) -- (A9.south) -- cycle;
			\fill[ftpink, edge] (A9.west) to (A12.west) -- (A12.east) -- (A9.east) -- cycle;
		\end{pgfonlayer}

		\coordinate (b1) at (6.75,-3.25) {};
		\coordinate (b2) at (4.75,-5.25) {};
		\coordinate(b3) at (7.25,-5.50) {};
		\coordinate (b4) at (9.25,-4.25) {};
		\coordinate (b5) at (13.75,-3.25) {};
		\coordinate (b6) at (13.25,-5.25) {};
		\coordinate (b7) at (16.25,-2.75)  {};
		\coordinate (b8) at (15.75,-5.75)  {};
		\coordinate (b9) at (10.75,-2.75)  {};
		\coordinate (b10) at (12.25,-2.25)  {};
		\coordinate (b11) at (11.75,-3.75)  {};
		\coordinate (b12) at (10.75,-5.25)  {};
		\coordinate (b14) at (11.50,-4.50)   {};

		\begin{pgfonlayer}{foreground}
			\node[smallset, draw=ftteal, label={90: \fc{$b_1$}}] at (b1) {};
			\node[smallset, draw=ftteal, label={90: \fc{$b_2$}}] at (b2) {};
			\node[smallset, draw=ftteal, label={90: \fc{$b_3$}}] at (b3) {};
			\node[smallset, draw=ftteal, label={90: \fc{$b_4$}}] at (b4) {};
			\node[smallset, draw=ftteal, label={90: \fc{$b_5$}}] at (b5) {};
			\node[smallset, draw=ftteal, label={90: \fc{$b_6$}}] at (b6) {};
			\node[smallset, draw=ftteal, label={90: \fc{$b_7$}}] at (b7) {};
			\node[smallset, draw=ftteal, label={90: \fc{$b_8$}}] at (b8) {};
			\node[smallset, draw=ftteal, label={90: \fc{$b_9$}}] at (b9) {};
			\node[smallset, draw=ftteal, label={90: \fc{$b_{10}$}}] at (b10) {};
			\node[smallset, draw=ftteal, label={90: \fc{$b_{11}$}}] (B11) at (b11) {};
			\node[smallset, draw=ftteal, label={90: \fc{$b_{12}$}}] (B12) at (b12) {};
			\node[vertex, fill=ftteal, draw=ftteal, label={90: \fc{$b_{14}$}}] (B14) at (b14) {};
		\end{pgfonlayer}

		\begin{pgfonlayer}{main}
			\draw[draw=ftteal, smalledge] (b1) -- (b2);
			\draw[draw=ftteal, smalledge] (b4) -- (b1);
			\draw[draw=ftteal, smalledge] (b4) -- (b3);
			\draw[draw=ftteal, smalledge] (b2) -- (b3);

			\draw[draw=ftteal, smalledge] (b6) -- (b5);
			\draw[draw=ftteal, smalledge] (b6) -- (b7);
			\draw[draw=ftteal, smalledge] (b8) -- (b7);
			\draw[draw=ftteal, smalledge] (b8) -- (b5);

			\draw[draw=ftteal, smalledge] (b3) -- (b12);
			\draw[draw=ftteal, smalledge] (b4) -- (b9);
			\draw[draw=ftteal, smalledge] (b12) -- (b9);
			\draw[draw=ftteal, smalledge] (b12) -- (b6);
			\draw[draw=ftteal, smalledge] (b11) -- (b10);
			\draw[draw=ftteal, smalledge] (b9) -- (b11);
			\draw[draw=ftteal, smalledge] (b5) -- (b10);
			\draw[draw=ftteal, smalledge] (b11) -- (b6);

			\fill[ftteal, edge] (B14.north) to (B12.west) -- (B12.east) -- (B14.east) -- cycle;
			\fill[ftteal, edge] (B14.west) to (B11.west) -- (B11.east) -- (B14.east) -- cycle;
		\end{pgfonlayer}
	\end{tikzpicture}

	\caption{
		A blow-up cover whose shape is the path $uvw$.
		The blow-ups $R^u(\cV^u)$, $R^v(\cV^v)$  and $R^w(\cV^w)$ are coloured green.
		The blow-ups $R^{uv}(\cW^{uv})$ and $R^{vw}(\cW^{vw})$ are coloured red and blue, respectively.
		The filled nodes present the singleton clusters.
		The $2$-graph $R^v$ has vertices $1,2,3,4,5$ and edges $12,23,34,14,15,25$.
		The corresponding set family is $\cV^v=\{V_1,V_2,V_3,V_4,V_5\}$ with exceptional cluster $V_5$.
		The set family $\cV^{v} \sm \{V_5\}$ is hit by the family $\cW^{uv}_v = \{W_1,W_2,W_3,W_4\} \subset \cW^{uv}$.
	}
	\label{fig:blow-up-cover}
\end{figure}
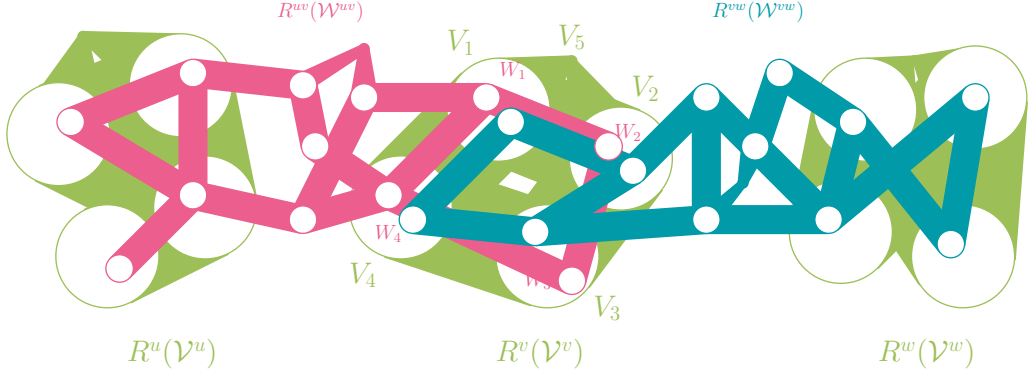


An illustration of a blow-up-cover is given in \cref{fig:blow-up-cover}.

{Our key technical result is that $\P$-covers exist in hypergraphs which satisfy $\P$ in a sufficiently robust way; and that we can even find such covers with clusters of size polynomial in $\log \log n$.}

\begin{proposition}[Blow-up cover]\label{pro:blow-up-cover}
	Let $1/k,\, 1/\Delta,\,  1/s_1 \gg 1/s_2 \gg c,\,\eta \gg 1/n$, $m_1= (\log n)^{c}$ and $m_2= (m_1)^{c}$.
	Let  $G$ be an $n$-vertex $[k]$-digraph.
	Let $\P$ be a family of $[k]$-digraphs.
	Suppose that $G$ satisfies $\P$ both $(1,s_1)$-robustly and {$(2s_1-2,s_2)$}-robustly.
	Then there exists $\ell \geq 1$ such that for every $\ell$-vertex graph $F$ with $\Delta(F)\leq \Delta$, it follows that $G$ has an $(s_1,s_2)$-sized $(m_1,m_2,\eta)$-balanced $\P$-cover of shape $F$.
\end{proposition}

{Our $\P$-covers with shape $F$ will play a rôle similar to the partition and reduced graph produced by the Regularity Lemma. Regular partitions generally require that the (reciprocal of the) error term $1/\eps$ is much smaller than the number of clusters. On the other hand, our output ensures that the parameter $1/\eta$ (which measures the imbalancedness of the clusters) is much larger than $s_1, s_2$, which are the number of clusters in each of the blow-ups of the cover. This is meaningful, because in our embedding applications later we will treat each blow-up separately, and this order of the constants is crucial for our approach to work.}

\subsection{Boosting}\label{sec:booster-lemma}

It is convenient to work with properties that survive the deletion of a few vertices.
For a $[k]$-digraph family $\P$, denote by $\Del_q(\P)$ the set of digraphs $H \in \P$ such that $H-X \in \P$ for every set $X \subset V(H)$ of at most $q$ vertices.
As it turns out, one can harden the properties ${\dcon}$, ${\dspa}$ and ${\dape}$ of \cref{thm:framework-bandwidth-directed} against vertex deletion.

We denote the minimum vertex-degree threshold for perfect $k$-uniform matchings by $\delta_{1}^{(s)}(\mat)$.
More precisely, $\delta_{1}^{(s)}(\mat)$ is the infimum over all $\delta \in [0,1]$ such that for every $\eps > 0$, there is an $n_0$ such that all $s$-graphs on $n \geq n_0$ vertices, with $n$ divisible by $s$, and $\delta_1(G) \geq (\delta + \eps) \binom{n-1}{k-1}$, contain a perfect matching.

\begin{lemma}[Booster]\label{lem:booster}
	Let $1/k,\, 1/s_1,\, 1/r_2, \, \eps, \, 1/q \gg 1/s_2 \gg 1/n$.
	Set $r_1=2k-2$ and $\delta_1 = {\big(\delta_{1}^{(s_1)}(\mat) + \eps\big)}$ and $\delta_2 = 1-\exp(-\sqrt{s_2})$.
	Then every $[k]$-digraph on $n$ vertices that $(\delta_1,r_1,s_1)$-robustly satisfies $\dcon \cap \dspa \cap \dape$ also $(\delta_2,r_2,s_2)$-robustly satisfies $\Del_q(\dcon \cap \dspa \cap \dape)$.
\end{lemma}

We prove \cref{lem:booster} in \cref{ssec:boosting}.
To place \cref{lem:booster} in the context of \cref{def:robustness,def:robustenss-detailed}, we recall the following fact due to Daykin and Häggkvist~\cite{DH81}, which follows from a simple double counting argument.

\begin{theorem}\label{fact:matchingthresholds}
	We have $\delta_{1}^{(s)}(\mat) \leq 1- 1/s$ for all $s \geq 2$.
\end{theorem}

In light of this, the choice of parameters in \cref{def:robustness,def:robustenss-detailed} can be explained as follows.
We chose $\delta = 1-1/s^2$ to guarantee a (robust) perfect matching in the property graph.
Moreover, we chose $r=2k \geq 2k-2$ to ensure connectivity at the $k$-uniform level for which one needs to consider (and connect) two disjoint $(k-1)$-sets.
This is done by finding a common edge in the property graph $P$, which is possible if $P$ has positive minimum $(2k-2)$-degree (see the proof of \cref{lem:booster}).
Finally, we recall the monotone behaviour of minimum degrees, which allows us to transition between different degree types.
\begin{fact}\label{fact:monotone-degrees}
	For an $n$-vertex $k$-graph $G$ and $d \leq d'$, we have $\frac{\delta_d(G)}{\binom{n-d}{k-d}} \geq \frac{\delta_{d'}(G)}{\binom{n-d'}{k-d'}}.$
\end{fact}

{Using these observations, the following corollary of \Cref{lem:booster} is immediate from \cref{def:robustenss-detailed} (and the definitions afterwards).}

\begin{corollary}[Booster]\label{corollary:booster2}
	Let $1/k,\, 1/s_1,\, 1/q,\, 1/r_2 \gg 1/s_2 \gg 1/n$,
	and set $\delta = 1 - \exp(-\sqrt{s_2})$.
	Every $[k]$-digraph on $n$ vertices that $s_1$-robustly satisfies $\dcon \cap \dspa \cap \dape$ also $(\delta, r_2, s_2)$-robustly satisfies $\Del_q(\dcon \cap \dspa \cap \dape)$. \hfill \qedsymbol
\end{corollary}

\subsection{Allocation}

The next two results implement the allocation step in the proofs of \cref{thm:framework-connectedness-robust,thm:framework-bandwidth-directed}.
Note that in each case, we only allocate to a single blow-up.
We begin with the allocation of a Hamilton path.

\begin{proposition}[Hamilton path allocation]\label{pro:allocation-Hamilton-path}
	Let $1/k, 1/s \gg \eta, 1/m$.
	Let $R \in \Del_k (\dcon \cap \dspa \cap \dape)$ be an $s$-vertex $[k]$-digraph.
	Let $\cV=\{V_x\}_{x \in V(R)}$ be a quasi $(1 \pm \eta)m$-balanced partition with exceptional cluster $V^\ast$.
	Let $f_1,f_2 \in (R(\cV)-V^\ast)^{(k-1)}$ be vertex-disjoint.
	Then $R(\cV)$ has a \tight Hamilton $(f_1,f_2)$-path.
\end{proposition}

The following result allocates a blow-up of a path.

\begin{proposition}[Path blow-up allocation] \label{pro:allocation-bandwidth-frontend}
	Let $1/k,\, 1/t,\, 1/s \gg \eta \gg \rho \gg 1/m$, with $s-1$ coprime to $t$.
	Let~$R$ be an $s$-vertex $k$-graph with $t$-clique-graph $\hat K = K_t(R)$.
	Suppose that $K \subset \ori{C}(\hat K) \cup \ori{C}(\partial \hat K)$ is an $s$-vertex $[t]$-digraph that satisfies $\Del_{t+1}(\dcon \cap \dspa \cap \dape)$.
	Let $\cV=\{V_x\}_{x \in V(R)}$ be a quasi $(1\pm \eta)m$-balanced partition on $n$ vertices with exceptional cluster $V_{x^\ast}$.
	Let $H$ be an $n$-vertex blow-up of the $(t-k+1)$st power of a path with cluster sizes at most $\rho m$.
	Then the blow-up $R(\cV)$ contains $H$.

	Moreover, we can ensure the following.
	Suppose that $e, f$ are two ordered $t$-edges of $K^{(t)}-x^\ast$.
	Let $H$ start with clusters $E_1,\dots,E_t \subset V(H)$ and end with clusters $F_1,\dots, F_t \subset V(H)$.
	Then there is an embedding $\phi$ of $H$ into $R(\cV)$ so that $\phi(E_i) \subseteq V_{e(i)}$ and $\phi(F_i) \subseteq V_{f(i)}$ for each $i\in [t]$.
\end{proposition}

\Cref{pro:allocation-Hamilton-path,pro:allocation-bandwidth-frontend} are proven in \Cref{sec:allocation-cycle} and \Cref{sec:allocation-path-blow-up}, respectively.

\subsection{Putting everything together}\label{sec:proof-main-result-bandwidth}

In the following, we derive our bandwidth theorem.

\begin{proof}[Proof of \cref{thm:framework-bandwidth-directed}]
	Set $\Delta = 2$. Given $k$, $t$ and $s$, we introduce constants $s_1, s_2, c', c, \eta, \rho, n_0$ such that
	\begin{equation*}
		1/k,\, 1/t,\, 1/s,\, 1/\Delta \gg 1/s_1 \gg 1/s_2 \gg c' ,\,\eta \gg \rho \gg 1/n_0.
	\end{equation*}
	More precisely, we choose $s_1$ so that $s_1-1$ is coprime to $t$, and large enough for \cref{corollary:booster2} applied with $1,s,s_1,t+1$ playing the rôles of $r_2, s_1, s_2, q$, respectively.
	We then choose $s_2$ so that $s_2 - 1$ is coprime to $t$ and large enough for \cref{corollary:booster2} applied with $2s_1-2,s,s_2,t+1$ playing the rôles of $r_2,s_1,s_2,q$, respectively.
	We also take $s_2$ to be large enough for \cref{pro:blow-up-cover} applied with $\Delta=2$.
	Next, fix $c'$ playing the rôle of $c$ in \cref{pro:blow-up-cover}, and let $c = (c')^3$.
	Fix $\eta, \rho$ in accordance with \cref{pro:blow-up-cover} and \cref{pro:allocation-bandwidth-frontend}, where in the latter $s_1$ and $s_2$ are playing the rôle of $s$.
	Then fix $\rho$ according to \cref{pro:allocation-bandwidth-frontend}.
	Finally, take $n_0$ to be large enough so that it can play the rôle of $n$ in all previous statements.
	Moreover, we also ensure that $(\log n_0)^{c^2}$ is so large that it can play the rôle of $m$ in \cref{pro:allocation-bandwidth-frontend} and large enough so that
	$(\log \log n)^c \leq \rho (\log ((\log n)^{c'}))^{c'}$ holds.

	Let $n \geq n_0$ be arbitrary, and let $m_1 = (\log n)^{c'}$ and $m_2 = (\log m_1)^{c'}$.
	Let $G$ be a $k$-graph on $n \geq n_0$ vertices with $t$-clique-digraph $K' = K_t(G)$.
	Suppose that $K'$ admits a $[t]$-digraph $K \subset \ori{C}(K') \cup \ori{C}(\partial K')$ that  $s$-robustly satisfies $\dcon \cap \dspa \cap \dape$.
	Let $H$ be an $n$-vertex blow-up of the $(t-k+1)$st power of a $k$-uniform \tight cycle with cluster sizes at most $m_3 (\log \log n)^c \leq \rho m_2$.
	Our goal is to show that $G$ contains $H$.

	We abbreviate $\P = \Del_{t+1}(\dcon \cap \dspa \cap \dape)$.
	Note that $K$ satisfies both $(1,s_1)$-robustly and $(2s_1-2,s_2)$-robustly $\P$ by two applications of \cref{corollary:booster2}.
	Hence, \cref{pro:blow-up-cover} implies that there exists $\ell \geq 1$ so that~$K$ has an $(s_1,s_2)$-sized $(m_1,m_2,\eta)$-balanced $\P$-cover formed by $\{\cV^v\}_{v \in V(F)}$ and $\{\cW^{e}\}_{e \in F}$ whose shape~$F$ is an $\ell$-vertex cycle.

	We denote the corresponding reduced $[t]$-digraphs of $\P$ by $\hat K^v$ and $\hat K^e$ for $v \in V(F)$ and $e \in F$.
	Note that for every $v \in V(F)$, there is a $k$-graph $R^v$ with $t$-clique-graph  $\hat K^v = K_t(R^v)$ such that $  K^v \subset \ori{C}(\hat K^v) \cup \ori{C}(\partial \hat K^v)$.
	Moreover, the same holds for $e \in F$ with $k$-graphs $R^e$ and $t$-clique-graph $\hat K^e$.
	We identify the vertices of the cycle $F$ with $V(F) = \{1,\dots,\ell\}$ following the natural cyclic ordering.

	Now we prepare the vertices of the guest graph $H$ for the allocation inside the blow-up cover.
	Let $\{U_1, \dotsc, U_L\}$ be a partition of $V(H)$ corresponding to the clusters of the $(t-k+1)$st power of a $k$-uniform cycle, according to the natural cyclic ordering.
	We will consider vertex-disjoint $(t-k+1)$st power of $k$-uniform \tight paths $ H^v , H^{e} \subset H$ for each $v \in V(F)$ and edge $e \in F$ (so $V(H)$ is partitioned by the vertex sets of the $k$-graphs $H^v$ and $H^{e}$.)
	We choose those subgraphs
	such that for every $v \in V(F)$ and $e \in F$, we have
	\begin{equation*}
		\text{$v(H^v) = \left|\bigcup \cV^v\right|-2s_1 m_2$ and $v(H^{e}) = \left|\bigcup \cV^e\right| = s_2 m_2$.}
	\end{equation*}
	We say that a pair $(x, y)$ is \emph{consecutive} if $x \in V(F)$ and $y = x(x+1) \in E(F)$; or if $y \in V(F)$ and $x = (y-1)y \in E(F)$.
	We also ensure that for each $x \in V(F) \cup E(F)$, there exists a (cyclic) interval $I_x \subseteq \{1, \dotsc, L\}$ such that
	\begin{enumerate}[\upshape{(\roman*)}]
		\item $V(H^x) \subseteq \bigcup_{i \in I_x} U_i$, and
		\item for each consecutive pair $(x,y)$, the last $t$ elements of the interval $I_x$ coincide with the first $t$ elements of the interval $I_{y}$.
	\end{enumerate}
	We can find the desired $H^v, H^e$ by a greedy argument, identifying one subgraph at a time following the cyclic ordering of $F$.

	Our goal is to embed each blow-up $H^v$ and $H^{e}$ into the corresponding blow-up $R^v(\cV^v)$ and $R^e(\cW^{e})$, respectively.
	This gives an embedding of $H$ into $G$, provided that beginning and terminal clusters are placed correctly.

	We begin by embedding the blow-ups $H^e$.
	For each $u \in V(F)$ and $v = u+1$ (index computations modulo $\ell$), let us fix $t$-edges $e^{uv} \in \hat K^{uv} \cap \hat K^u$ and $f^{uv} \in \hat K^{uv} \cap \hat K^v$ that avoid the exceptional clusters.
	Suppose the blow-up $H^{uv}$ starts with clusters $E_1^{uv},\dots,E_t^{uv}$ and ends with clusters $F_1^{uv},\dots, F_t^{uv}$.
	Write $\cW^{uv} = \{W_{x}^{uv}\}_{x \in V(R^{uv})}$.
	By our choice of $c$ and $n_0$, we have that each cluster of $H^e$ has size at most $(m_3 \leq \rho m_2$.
	It follows by \cref{pro:allocation-bandwidth-frontend} that there is an embedding $\phi^{uv}\colon H^{uv} \rightarrow R^{uv}(\cW^{uv})$ so that $\phi(E_i^{uv}) \subseteq V_{e^{uv}(i)}$ and $\phi(F_i^{uv}) \subseteq V_{f^{uv}(i)}$ for each $i\in [t]$.

	We embed the blow-ups $H^v$ in the same way.
	Given $v \in V(F)$, let $u=v-1$ and $w=v+1$.
	Write $\cV^v=\{V^v_x\}_{x \in V(R^v)}$.
	Let $H^{v}$ start with clusters $F_1^{v},\dots,F_t^{v}$ and end with clusters $E_1^{v},\dots, E_t^{v}$.
	By our choice of $c$ and $n_0$, we have that each cluster of $H^e$ has size at most $m_3 \leq \rho m_1$.
	It follows by \cref{pro:allocation-bandwidth-frontend} that there is an embedding $\phi^v \colon H^v \rightarrow R^{v}(\cV^{v})$ so that $\phi(F_i^{v}) \subseteq V^v_{f^{uv}(i)}$ and $\phi(E_i^{v}) \subseteq V^v_{e^{vw}(i)}$ for each $i\in [t]$.

	In combination, this gives the desired embedding of $H$ into $G$.
\end{proof}

\section{Blow-up covers}\label{sec:blow-up-covers}

In this section, we give a proof of \rf{pro:blow-up-cover} that comes in two parts.
We first cover all vertices with quasi balanced blow-ups.
Then we add the connections.

\begin{lemma}[Simple blow-up cover]\label{lem:simple-blow-up-cover}
	Let $1/s \gg c,\eta \gg 1/n$ and $m = (\log n)^c$.
	Let  $G$ be an $n$-vertex $[k]$-digraph.
	Let $\P$ be a family of $[k]$-digraphs.
	Suppose that $G$ satisfies $(1,s)$-robustly $\P$.
	Then the vertices of $G$ are covered by vertex-disjoint quasi $(1\pm\eta)m$-balanced blow-ups whose reduced graphs satisfy $\P$ and have order $s$ each.
\end{lemma}

A different way to put the outcome of \cref{lem:simple-blow-up-cover} is to say that $G$ contains an $(s_1,\cdot)$-sized $(m_1,\cdot,\eta)$-balanced $\P$-cover of empty shape.
The next result allows us to realise non-trivial shapes by adding a single edge to the current shape.

\begin{lemma}[Connecting blow-ups]\label{lem:connecting-blow-ups}
	Let $\eps,  1/s_1 \gg 1/s_2 \gg c \gg 1/n$.
	Suppose $m_1 \leq n/s_2$ and $m_2= (\log m_1)^c$ are positive integers.
	Let  $G$ be an $n$-vertex $[k]$-digraph.
	Let $\P$ be a family of $[k]$-digraphs.
	Suppose that $G$ satisfies  $(\eps,s_1,s_2)$-robustly $\P$.
	Let $\cV$ be an $m_1$-balanced set family in $V(G)$ with $s_1$ clusters.
	Then $G$ contains an $s_2$-sized quasi $m_2$-balanced $\P$-blow-up $T(\cW)$ such that $\cW$ hits $\cV$, and such that the exceptional vertex of~$\cW$ is not contained in $\cV$.
\end{lemma}

Given this, we can already derive the main result of this section.

\begin{proof}[Proof of \rf{pro:blow-up-cover}]
	We apply \cref{lem:simple-blow-up-cover} to cover the vertices of $G$ by $\ell$ vertex-disjoint quasi $(1\pm\eta/2)m_1$-balanced blow-ups whose reduced graphs satisfy $\P$ and have order $s_1$ each.
	In other words, we have an $(s_1,\cdot)$-sized $(m_1,\cdot,\eta/2)$-balanced $\P$-cover of $G$ whose shape is the empty graph on $\ell$ vertices.
	We turn this shape into an arbitrary graph of maximum degree $\Delta$ by adding one edge after another.

	We describe the process of adding one edge to the shape of the current $\P$-cover of $G$.
	Let $\mathcal{V}_{\textrm{init}}$ be the collection of non-exceptional clusters in the $\P$-cover we have found initially; and let $V_{\textrm{exc}}$ be the set of vertices of the exceptional clusters.
	Consider two of the blow-ups, say $R_1(\cV_1)$ and $R_2(\cV_2)$, where an edge of $F$ needs to be added.
	Let us define a set of vertices $V_{\textrm{avoid}}$ which we want to avoid as follows.
	Let $V_{\textrm{used}}$ be the set of vertices already used in previous blow-ups $R^e(\mathcal{W}^e)$ corresponding to some other edges $e \in F$.
	Let $\mathcal{V}_{\textrm{sat}} \subseteq \mathcal{V}_{\textrm{init}} \setminus (\mathcal{V}_1 \cup \mathcal{V}_2)$ be the set of clusters which intersect $V_{\textrm{used}}$ in at least $\eta m_1 / 4$ vertices.
	Let $V_{\textrm{sat}} = \bigcup \mathcal{V}_{\textrm{sat}}$ be the vertices contained in them.
	Let $V_{\textrm{avoid}} = V_{\textrm{used}} \cup V_{\textrm{sat}} \cup V_{\textrm{exc}}$.
	We assume for now (and justify later) that $|V_{\textrm{avoid}}| \leq \eta n$.
	Because of this, $G-V_{\textrm{avoid}}$ still satisfies $(\delta,2s_1-2,s_2)$-robustly $\P$ with $\delta = 1-1/s_2-\eta$
	(as specified in \cref{def:robustenss-detailed}).
	By \cref{lem:connecting-blow-ups}, $G-V_{\textrm{avoid}}$ contains an $s_2$-sized quasi $m_2$-balanced $\P$-blow-up $T(\cW)$ such that $\cW$ hits both $\cV_1$ and $\cV_2$, and the exceptional cluster of $\cW$ is not in $\cV_1 \cup \cV_2$.
	This finishes the process of adding one edge.

	Let us now justify that $|V_{\textrm{avoid}}| \leq \eta n$.
	Since each cluster in $\mathcal{V}_{\textrm{init}}$ has at least $(1 - \eta)m_1$ vertices, and each blow-up has $s_1 - 1$ non-exceptional clusters, we get that $\ell \leq n/((s_1-1)(1 - \eta) m_1) \leq 2 n / (s_1 m_1)$.
	Since each blow-up has one exceptional cluster, we also have that $|V_{\textrm{exc}}| \leq \ell \leq 2 n / (s_1 m_1) \leq \eta n / 3$.
	Since $\Delta(F) \leq \Delta$ we have  $|E(F)| \leq \ell \Delta / 2 \leq \Delta n / (s_1 m_1)$.
	Since each added blow-up uses $(s_2-1)m_2+1 \leq s_2 m_2$ vertices, it follows that $|V_{\textrm{used}}| \leq s_2 m_2 |E(F)| \leq \Delta s_2 m_2 n / (s_1 m_1) \leq \eta n / 3$, where the last step uses $\Delta = o(m_1/m_2)$.
	Note that~$\mathcal{V}_{\textrm{sat}}$ can have at most $|V_{\textrm{used}}| / (\eta m_1 / 4)$ clusters, and since each cluster in $\mathcal{V}_{\textrm{init}}$ has at most $(1 + \eta)m_1$ vertices, we obtain that $|V_{\textrm{sat}}| \leq (1 + \eta)m_1 |V_{\textrm{used}}| / (\eta m_1 / 4) \leq 4 (1 + \eta) |V_{\textrm{used}}| \eta^{-1} \leq \eta n / 3$, again from $\Delta = o(m_1/m_2)$.
	Thus, we have $|V_{\textrm{avoid}}| \leq \eta n$ as desired.

	The above procedure allows us to stepwise add edges to $F$, ensuring that the blow-ups of the edges are pairwise disjoint.
	Let $U^\ast$ be the final set after incorporating every edge of $F$.
	Note that, for any cluster $X \in \mathcal{V}_{\textrm{init}}$, we have $|U^\ast \cap X| \leq \eta m_1 / 2$ at the end of the construction.
	Indeed, suppose $X$ belongs to a blow-up $R_1(\cV_1)$.
	After adding an edge between the clusters $R_a(\cV_a)$ and $R_b(\cV_b)$ distinct from $R_1(\cV_1)$; if $X$ is not saturated then at most $\eta m_1 / 4 + s_2 m_2 \leq \eta m_1 / 3$ vertices will be in $U \cap X$, after becoming saturated we do not add more vertices to it.
	For the edges which do contain $R_1(\cV_1)$ we use at most $\Delta m_2$ vertices in $X$, but this is at most $\eta m_1 / 6$ (if $\Delta = o(m_1/m_2)$).
	Thus indeed the total number of vertices used in $X$ is at most $\eta m_1/2$.

	By removing from each original cluster the vertices from $U^\ast$, we transform the initial $(s_1, s_2)$-sized $(m_1,m_2,\eta/2)$-balanced $\P$-cover (with empty shape) into a $(s_1, s_2)$-sized $(m_1,m_2,2\eta)$-balanced $\P$-cover with shape $F$.
\end{proof}

It remains to show \cref{lem:simple-blow-up-cover,lem:connecting-blow-ups}, which is done in the next two sections.
We shall also prove the following related result, which tracks edges that robustly extend to many small blow-ups and finds its application in the proof of \cref{thm:framework-connectedness-robust}.

\begin{lemma}\label{lem:blow-up-support}
	Let $1/k,\,\eps,\, 1/s,\,1/m \gg \beta \gg 1/n$ and $1 \leq d \leq k$.
	Let~$G$ be an $n$-vertex $[k]$-digraph.
	Let~$\P$ be a family of $s$-vertex $[k]$-digraphs with $R^{(d)}$ non-empty for each $R \in \P$.
	Suppose that $\PG{G}{  \P   }{s}$ has at least $\eps n^s$ edges.
	Then there are at least $\beta n^{d}$ edges $e \in G^{(d)}$ such that there are $\beta n^{ms-d}$ many $m$-balanced $\P$-blow-ups in~$G$ that contain $e$ as an edge.
\end{lemma}

\subsection{Connecting blow-ups}

At the heart of the proof of \cref{pro:blow-up-cover} sits an old insight of Erdős~\cite[Theorem 2]{Erd64}, which ensures that dense uniform hypergraphs contain complete blow-ups of polylogarithmic size.

\begin{theorem}\label{thm:erd64}
	For all $s\geq 2$  and $\eps > 0$, there are $c,n_0>0$ such that every $s$-partite $s$-graph $P$ with $n \geq n_0$  vertices in each part and at least $\eps n^s$  edges contains an $m$-blow-up of an edge with $m = c\cdot (\log n)^{1/(s-1)}$.
\end{theorem}

\begin{proof}[Proof of \cref{lem:connecting-blow-ups}]
	Let $P = \PG{G}{  \P   }{s_2}$ be the property $s_2$-graph with minimum $s_1$-degree at least $\eps \tbinom{n-s_1}{s_2-s_1}$.
	(This crude bound suffices for our purposes.)
	We begin by switching to the partite setting.
	Select a random family $\cV'$ of $s_2-s_1$ pairwise disjoint $m_1$-sets in $V(G) \sm \bigcup \cV$.
	Let $Q \subset P$ be the subgraph obtained by keeping only $(\cV \cup\cV')$-partite edges.
	A standard probabilistic argument guarantees that with positive probability $\deg_{Q}(X) \geq (\eps/2) m_1^{s_2-s_1}$ for every $\cV$-partite $s_1$-set~$X \subset V(Q)$.
	In particular, $Q$ has at least $(\eps/2) m_1^{s_2}$ edges.
	Fix~$Q \subset P$ with this property.

	Note that every edge $Y \in Q$ corresponds to a labelled copy of a $[k]$-digraph $R \in \P$ on vertex set $[{s_2}]$.
	There are at most $2^{k{s_2}^k}$ many forms that $R$ can take.
	Hence we can find a specific $R \in \P$ and a subgraph $Q' \subset Q$, such that $Q'$ contains $2^{-k{s_2}^k}\eps m_1^{s_2}$ edges $Y$ that correspond (uniformly) to $R$.
	We apply \cref{thm:erd64} to $Q'$ to find an $s_2$-sized $m_2$-balanced blow-up of an edge in $Q$.
	Note that this already gives an $m_2$-balanced $\P$-blow-up whose cluster family hits $\cV$.
	To finish, we need to guarantee that this blow-up has a singleton cluster outside of $\cV$.
	We simply select a cluster of the blow-up which does not intersect $\cV$ and remove all but one vertex.
\end{proof}

The proof of \cref{lem:blow-up-support} follows along the same lines.
We use the following result, which follows from \cref{thm:erd64} and the supersaturation phenomenon~\cite[Lemma 2.1]{Kee11a}.

\begin{theorem}\label{thm:erd64-supersaturated}
	For all $s\geq 2$, $m \geq 1$ and $\eps > 0$, there are $\beta ,n_0>0$ such that every $s$-partite $s$-graph $P$ with $n \geq n_0$  vertices in each part and at least $\eps n^s$ edges contains at least $\beta n^{sm}$ $m$-blow-ups of an edge.
\end{theorem}

\begin{proof}[Proof of \cref{lem:blow-up-support}]
	Introduce $\beta'$ with $1/k,\,\eps,\, 1/s,\,1/m \gg \beta' \gg \beta$.
	Without loss of generality, we may assume that $G$ has $sn$ vertices.
	To begin, we select a random $n$-balanced partition $\cV$ of $V(G)$ with $s$ parts.
	Let $Q \subset P$ be the subgraph obtained by keeping only $\cV$-partite edges.
		{A standard probabilistic argument guarantees that $Q$ has at least $(\eps/2) n^s$ edges.}

	Note that every edge $Y \in Q$ corresponds to a labelled copy of a $[k]$-digraph $R \in \P$ on vertex set $[s]$.
	There are at most $2^{k{s}^k}$ choices for $R$.
	Hence we can find a specific $R \in \P$ and a subgraph $Q' \subset Q$, such that $Q'$ contains $2^{-k{s}^k}(\eps/2) n^{s}$ edges $Y$ that correspond (uniformly) to $R$.
	We apply \cref{thm:erd64-supersaturated} to $Q'$ to find $\beta' n^{ms}$ many $m$-balanced blow-ups of an edge.

	Note that each of these blow-ups corresponds to an $m$-balanced $\P$-blow-up.
	Moreover, each blow-up contains at least $m^d$ edges $e \in G^{(d)}$, since $R^{(d)}$ is non-empty for each $R \in \P$.
	We may therefore conclude with a basic averaging argument.
\end{proof}

We also record (without proof) the following consequence of \cref{thm:erd64}, that can be shown analogously.

\begin{lemma}[Rooted blow-ups]\label{lem:rooted-blow-ups}
	Let $1/k,\,  1/s \gg c \gg n$, $1 \leq m_1 \leq (\eps/2)n$ and $m_2 = \lfloor (\log m_1)^c \rfloor$.
	Let $G$ be an $n$-vertex $[k]$-digraph.
	Let $\P$ be a family of $[k]$-digraphs.
	Suppose that $G$ satisfies $(1,s)$-robustly $\P$.
	Let $\cV$ be an $m_1$-balanced set family in $V(G)$ with a single cluster.
	Then $G$ contains an $m_2$-balanced $\P$-blow-up $T(\cW)$ such that $\cW$ hits~$\cV$.
	\qed
\end{lemma}

\subsection{Simple blow-up covers}

The proof of \cref{lem:simple-blow-up-cover} is based on the following result, which has been adapted from earlier work~\cite[\S 4.2]{Lan23}.

\begin{lemma}[Almost blow-up cover]\label{lem:almost-blow-up-cover}
	Let $1/k,\,  1/s, \, \eta\gg c \gg 1/n$ and $1 \leq m \leq (\log n)^c$.
	Let  $G$ be an $n$-vertex $[k]$-digraph.
	Let $\P$ be a family of $[k]$-digraphs.
	Suppose that $G$ satisfies  $(1,s)$-robustly $\P$.
	Then all but at most~$\eta n$ vertices of $G$ may be covered with quasi $m$-balanced $\P$-blow-ups, each of order $s$.
\end{lemma}

We defer the details of the proof of \cref{lem:almost-blow-up-cover} to \cref{sec:almost-blow-up-cover}.

\begin{proof}[Proof of \rf{lem:simple-blow-up-cover}]
	Introduce $c_1,c_2,c_3$ with $\eps,  1/s \gg c_1,\eta \gg c_2 \gg c_3 \gg 1/n$ as well as $\eta'$ with $\eta \gg \eta' \gg 1/n$.
	Then set $m_i = (\log n)^{c_i}$ for $i \in [3]$.
	By \cref{lem:almost-blow-up-cover}, all but at most $\eta' n$ vertices of $G$ may be covered with pairwise vertex-disjoint $m_1$-balanced $\P$-blow-ups $R_1(\cV_1),\dots,R_{p}(\cV_p) \subset G$, where each~$R_i$ has order $s$.
	The uncovered vertices, denoted by $W$, are added onto blow-ups in two steps.

	We first cover most vertices of $W$ as follows.
	While $|W| \geq \eta'n/m_1$, iteratively apply \cref{lem:rooted-blow-ups} to find vertex-disjoint blow-ups $T^1(\cW_1),\dots,T^{p}(\cW_q)$ such that each $\cW_i$ is an $m_2$-balanced set family that hits $\{W\}$.
	Each blow-up covers $(s-1)m_2$ vertices of $G-W$ and $m_2$ vertices of $W$.
	So the process requires at most $(s-1)|W|/m_2 \leq n/m_2$ vertices of $G-W$, which means that we can greedily choose these blow-ups pairwise vertex-disjoint.
	Moreover, since $(s-1)m_2 \leq \eta' m_1$, we can also ensure that no cluster of $\cV_1 \cup \dots \cup \cV_p$ is intersected in more than $(\eta/8)m_1$ vertices.
	Denote the remaining uncovered vertices by $U \subset W$.

	To finish the covering, we apply \cref{lem:rooted-blow-ups} to add each $u \in U$ onto a quasi $m_3$-balanced $\P$-blow-up $S^u(\cU_u)$, where $S^u$ has order $s$ and $\{u\} \in \cU_u$.
	(The requirements of the lemma are met by replacing $u$ with $\eps n $ copies of itself, which are dropped afterwards.)
	Note that this blow-up covers $(s-1)m_3$ vertices of $G-U$ and $1$ vertex of $U$.
	Since $(s-1)m_3|U| \leq \eta' n$, we can select the blow-ups $S^u(\cU_u)$ ($u \in W$) to be pairwise disjoint.
	Moreover, since $(s-1)m_3 \leq \eta' m_2$, we may ensure that any cluster of $\cV_1,\dots,\cV_p,\cW_1 \cup \dots \cup \cW_q$ is intersected in at most $(\eta/8)m_2$ vertices.

	We delete the vertices of each $\cW_i$ from each $\cV_j$, keeping the names for convenience.
	Afterwards, we delete the vertices from $\cU_\ell$ from each of $\cW_i$ and $\cV_j$, keeping again the names.
	Two issues remain.
	Firstly, the clusters of the families $\cW_i$ and $\cV_i$ are still much larger than the clusters of $\cU_i$.
	Secondly,
	we need to obtain quasi balanced blow-ups, for which we need to create exceptional vertices in those blow-ups.
	We address both of those issues as follows.
	First, we split each $\cW_i$ and each $\cV_j$ into $(1\pm \eta/2)((s-1)m_3)$-balanced families.
	(This can be done greedily.)
	Finally, we partition each of these $(1\pm \eta/2)((s-1)m_3)$-balanced families into $s$ quasi $(1\pm \eta)m_3$-balanced families.
	To do this, assuming the family is $X_1, \dotsc, X_s$, we can partition each $X_i$ into $s$ sets, one of size $1$, and the other ones of size as equal as possible, that is, of size $(1 \pm \eta)m_3$.
	Then we form $s$ new families by including exactly one singleton in each, and $s-1$ many $(1 \pm \eta)m_3$-sized sets.
	Hence every family is now quasi $(1\pm \eta)m_3$-balanced, and we can finish with $c=c_3$.
\end{proof}

\subsection{Covering most vertices with blow-ups}\label{sec:almost-blow-up-cover}

Here we show \cref{lem:almost-perfect-regular-tuple-tiling}.
We require the following notion of (weak) quasirandomness.
For  an $s$-graph~$P$, the \emph{density} of a tuple $(V_1,\dots,V_k)$ of pairwise disjoint vertex sets is
\begin{equation*}
	d_P(V_1,\dots,V_k) = \frac{e_P(V_1,\dots,V_k)}{|V_1|\dots|V_k|}\,.
\end{equation*}
We say that $(V_1,\dots,V_k)$ is \emph{$(\eps,d)$-lower-regular} if for $d=d_P(V_1,\dots,V_k)$ and all choices of $X_1\subset V_1,$ $\dots,$ $X_k \subset V_k$ satisfying $|X_1| \geq \eps |V_1|,\,\dots,\,|X_k| \geq \eps |V_k|$, we have
$ d_P(X_1,\dots,X_k)  \geq d- \eps$.
Denote by $\cQ(s,m,\eps,d)$ the set of $m'$-balanced $(\eps,d')$-lower-regular $s$-tuples with $d'\geq d$ and $m' \geq m$.

Our next lemma shows that one can partition most of the vertices of a hypergraph with large enough minimum degree into balanced lower-regular tuples, where we allow distinct tuples to have distinct part sizes.
This can easily be derived from the Weak Hypergraph Regularity Lemma (see, e.g.~\cite[\S 2]{KNR+10}), which, however, leads to tower-type dependencies on the constant hierarchy.
To avoid this, we use an alternative approach, which comes with milder constant dependencies and an overall shorter proof.

For an $s$-graph $P$ and a family $\cF$ of $s$-graphs, an \emph{$\cF$-tiling} is a set of pairwise vertex-disjoint $k$-graphs $F_1,\dots,F_\ell \subset H$ with $F_1,\dots,F_\ell \in \cF$.

\begin{lemma}[Almost perfect tiling]\label{lem:almost-perfect-regular-tuple-tiling}
	Let $1/s, \eta \gg d,\eps \gg \alpha \gg 1/n$ and $m=\alpha n$.
	Then every $n$-vertex $s$-graph $P$ with $\delta_1(P) \geq \left(1-1/s^2 \right) \binom{n-1}{s-1}$ contains a $\cQ(s,m,\eps,d)$-tiling that covers all but $\eta n$ vertices.
\end{lemma}

Given this, we can easily derive the main result of this section.

\begin{proof}[Proof of \rf{lem:almost-blow-up-cover}]
	Introduce $d,\eps,\alpha$ with $1/s,\, \eta \gg d ,\,\eps \gg \alpha \gg c$ where $\eps = d/2$.
	By \cref{lem:almost-perfect-regular-tuple-tiling} applied with $\eta/2$ playing the rôle of $\eta$, the $s$-graph $P$ contains a $\cQ(s,\alpha n,\eps, d)$-tiling that covers all but $\eta n / 2$ vertices.
	Consider one of the $(\eps,d)$-lower-regular $s$-tuples $\cV=\{V_1,\dots,V_s\}$, which is $m'$-balanced for some $m' \geq \alpha n$.
	Note that each $\cV$-partite edge of~$P$ corresponds to some labelled $s$-vertex $[k]$-digraph in $\P$.
	So there is an $s$-vertex $[k]$-digraph $R_1 \in \P$, which appears at least $(d -\eps)2^{-ks^k} (m')^s$ times.
	By \cref{thm:erd64}, we may find a subgraph $R_1(m) \subset G$ with $\cV$-partite edges.
	We repeat this procedure to obtain pairwise vertex-disjoint blow-ups $R_1(m),\dots,R_{\ell}(m)$ until all but $\eps m'$ vertices are covered in each cluster of $\cV$.
	{This is possible, because the tuple $(V_1,\dots,V_s)$ is $(\eps, d)$-lower-regular, so the set of remaining vertices forms a dense partite $s$-graph which allows us to apply the previous argument again.
	After iterating this over all regular tuples, we have thus covered with blow-ups all but $\eps n \leq \eta n/2$ of those vertices that were covered by the tuples.
	Together with the at most $\eta n / 2$ vertices not covered by the tuples, this gives the desired result.}
\end{proof}

It remains to show \cref{lem:almost-perfect-regular-tuple-tiling}.
We shall use the following quasirandom analogue of \cref{thm:erd64}, which states that every (reasonably) large dense hypergraph admits a regular tuple of linear order.

\begin{lemma}[Regular tuple]\label{lem:density-regular-tuple}
	Let $1/s, \, d,\, \eps \gg 1/m$ and $\gamma = \exp(-\eps^{-2s})$.
	Let $P$ be an $m$-balanced $s$-partite $s$-graph with $e(P) \geq d m^s$.
	Then $P$ contains an $(\eps,d)$-lower-regular $m_1$-balanced $s$-tuple, where $\gamma m \leq m_1$.
\end{lemma}
\begin{proof}
	Without loss of generality, we can assume that $m' =\eps m$ is an integer.
	(Otherwise, decrease~$\eps$ accordingly.)
	By assumption, the vertex set of $P$ consists of an $m$-balanced tuple $(V_1,\dots,V_s)$ of density $d_P(V_1,\dots,V_s)\geq d$.

	Suppose that $(V_1,\dots,V_s)$ is not $(\eps,d)$-lower-regular.
	Then there exists $X_1\subset V_1,$ $\dots,$ $X_s \subset V_s$ with $|X_1|,\dots,|X_s| \geq m'$ and $d_P(X_1,\dots,X_s) < d - \eps$.
	By averaging, we can find $Y_1\subset X_1,$ $\dots,$ $Y_s \subset X_s$ with $|Y_1|\,,\dots,\,|Y_s| = m'$ and $d_P(Y_1,\dots,Y_s)  < d - \eps$.
	Partition each $V_i \setminus Y_i$ into $m'$-sized sets; by adding $Y_i$ we obtain a partition $\mathcal{P}_i$ of $V_i$ into exactly $m/m' = \eps^{-1}$ many sets.
	Writing $t = \eps^{-1}$, there are $t^s$ many choices for a tuple $(Z_1, \dotsc, Z_s)$ of sets where $Z_i \in \mathcal{P}_i$.
	By averaging, we can find a tuple $(Z_1,\dots,Z_s)$ of parts such that $d_P(Z_1,\dots,Z_s) \geq d + \gamma$ for $\gamma = (\eps/2) (t)^{-s} = (\eps/2) \eps^{s}$.
	If $(Z_1,\dots,Z_s)$ is $(\eps,d)$-lower-regular, we are done.
	Otherwise, we iterate this process within $P[Z_1,\dots,Z_s]$.
	In each step, the density increases by $\gamma$ and the part size decreases by a factor of at most $\eps/2$.
	The process stops after at most $1/\gamma$ steps with part sizes of at least $(\eps/2)^{1/\gamma} m \geq \exp(-\eps^{-2s}) m$.
\end{proof}

We then derive \cref{lem:almost-perfect-regular-tuple-tiling} by iteratively applying the following result.
{We apply the lemma with {$\mu = 16 \eta$}, so that $1 - 1/s^2 \geq \th_{1}^{}(\mat_s)+ \mu$.}
In each step, we turn a $\cQ$-tiling into another $\cQ$-tiling with smaller parts such that additional $(\nu\mu/16)n = \eta^2n$ vertices are covered.
So we arrive at \cref{lem:almost-perfect-regular-tuple-tiling} after at most $\eta^{-2}$ steps.

\begin{lemma}[Regular tuple tiling increment]\label{lem:larger-matching}
	Let $ 1/s,\, \mu,\, \eta, \, d,  \,\eps \gg \alpha \gg 1/n$ with $d \leq \mu/16$ and $\eps \leq d/2$.
	Set $\gamma = \exp(-\eps^{-2s})$ and $m=\alpha n$.
	Let $\cQ_1 = \cQ(s,m,\eps,{d})$ and $\cQ_2 = \cQ(s, m',\eps',d')$,
	where $m' = (\gamma \eta \mu^2/128)m$,
	$\eps'= 2\eps$,
	 and $d' = d/2$.
	Let $P$ be an $s$-graph on $n$ vertices with $\delta_1(P) \geq \left(\th_{1}^{}(\mat_s)+ \mu \right) \binom{n-1}{s-1}$.
	Suppose that $P$ contains a $\cQ_1$-tiling $Q_1$ on $\lambda n$ vertices with $0 \leq \lambda \leq 1-\eta$.
	Then $P$ contains a $\cQ_2$-tiling on at least $(\lambda + \mu \eta /16) n$ vertices.
\end{lemma}

\COMMENT{
	Before we come to the proof of \cref{lem:larger-matching}, let us derive \cref{lem:almost-perfect-regular-tuple-tiling}.
	\begin{proof}[Proof of \cref{lem:almost-perfect-regular-tuple-tiling}]
		Set $\mu = 16 \eta$, and note that $1 - 1/s^2 \geq \th_{1}^{}(\mat_s)+ \mu$.
		Set $t = (\mu \eta/16)^{-1} = \eta^{-2}$.
		For $0 \leq i \leq t$, we set $d_i = (\mu/16)/2^i$, $\eps_i = d_i/2$, $\gamma_i = \exp(-\eps_i^{-2s})$, $\alpha_0 = 1$, $\alpha_i = \gamma_i \eta \mu^2/128$, $m_i = \alpha_i n$, $\lambda_0 = 0$ and $\lambda_i = \lambda_{i-1} + \eta^2$.
		We begin with $\cQ_1 = \cQ(s,\alpha_0 n,\eps_0,d_0)$ and an empty $\cQ_1$-tiling $Q_1$.
		We then apply \cref{lem:larger-matching} iteratively until all but $\eta n$ vertices are covered.
		So at the beginning of step $i$, we have a $\cQ_1$-tiling $Q_1$ on $\lambda_i n$ vertices in $P$, where $\cQ_{1} = \cQ(s,m_i,\eps_i,d_i)$.
		And at the end of step $i$, we have a $\cQ_{2,i}$-tiling on at least $(\lambda_i + \mu \eta /16) n = (\lambda_i + \eta^2)n = \lambda_{i+1} n$ vertices, where $\cQ_{2,i} = \cQ(s, m_{i+1},\eps_{i+1},d_{i+1})$.
		Clearly, this process stops after at most $t$ steps.
		So we can finish with a {$\cQ(s,\alpha_t n,\eps_t,d_t)$-tiling} that covers all but $\eta n$ vertices.
	\end{proof}
}

Let us briefly explain how we intend to prove \cref{lem:larger-matching}.
Our plan is to extract additional regular tuples from the $(1-\lambda)n \geq \eta n$ vertices not covered by $Q_1$.
We shall do this by using the assumption that the minimum degree is $\mu \binom{n-1}{s-1}$ above the threshold for a perfect matching.
To track these gains, we define $\nu = \mu/16$ and $\beta = \nu \eta /2$.
In a first step, we find a tiling $\cQ_{\text{fresh}}$ that covers $4\beta \eta n$ additional vertices outside of $Q_1$ and some vertices inside of $Q_1$.
In a second step, we `recycle' the remainder of $Q_1$ using another tiling $\cQ_{\text{rec}}$ such that $\cQ_{\text{fresh}} \cup \cQ_{\text{rec}}$ together cover all but $2\beta \eta n$ vertices in $Q_1$.
We note that the leftover comes from the fact that removing $\cQ_{\text{fresh}}$ leaves parts of $Q_1$ a bit unbalanced.
In any case, the relative difference $2\beta n = \mu \eta /16$ presents the share of additional covered vertices at the end of the argument.

Recall that the input part sizes are $m$ and $m'$.
Since the parts of $Q_1$ may be of different size, it will be convenient to use an intermediate part size.
We therefore define $m_1 = m$, $m_2 =  (\mu/8) m_1$ and $m_3 = m' = \gamma 2\beta  m_2$.
Our approach is then to refine the parts of $Q_1$, which have size at least $m_1$, into parts of uniform size $m_2$.
We remark that $m_2$ is chosen small enough to take advantage of the excess minimum degree $\mu \binom{n-1}{s-1}$.
Moreover, $m_3$ is chosen small enough to keep the number of lost vertices due to imbalancedness (discussed above) under $2\beta n$.
Now come the details.

\begin{proof}[Proof of \cref{lem:larger-matching}]
	We set $\nu = \mu/16$ and $\beta = \nu \eta /2$.
	Moreover, let $m_1 = m$, $m_2 =  (\mu/8) m_1$ and $m_3 = m' = \gamma 2\beta  m_2$.
	We begin by defining a family $\cU$ of disjoint sets in $V(P)$, as follows.
	For each $s$-tuple in $Q_1$ with parts $\{V_1, \dotsc, V_s\}$, pick a maximal number of vertex-disjoint sets of size $m_2$ inside each $V_i$, and add all of them to $\cU$.
	It follows that the sets of $\cU$ cover all but $m_2 = (\mu/8)m_1 \leq (\mu/8) |V_i|$ vertices in each $V_i$.
	Next, we extend $\cU$ by subdividing the vertices outside of $Q_1$ into a maximal number of vertex-disjoint $m_2$-sets such that $r = |\cU|$ is divisible by $s$.
	This implies that outside of $V(Q_1)$, the family $\cU$ covers all but at most $sm_2$ vertices.
	We remark that $1/s,\mu$ are much larger than $1/r$ by the choice of $\alpha$.
	Now, let $R$ be an $s$-graph with vertex set~$\cU$ and an edge $X$ if $P$ contains at least $4\nu m_2^{s}$ $X$-partite edges.
	{So $R$ plays the rôle of a reduced graph in the context of a Regularity Lemma.}
	It is therefore not surprising that $R$ inherits the minimum degree condition of $P$.
	Indeed, simple counting shows that $\delta_1(R) \geq \left( \th_{1}^{}(\mat_s)+\mu/4 \right) \binom{r-1}{s-1}$.\COMMENT{
	Fix a vertex $x$ in $R$.
	There are at most $r^{s-2} m_2^{s} \leq (\mu/4) m_2^{s} \binom{r-1}{s-1}$ edges of $G$, which have one vertex in the part of $x$ and more than two vertices in some part.
	(For the inequality, we used that $1/s,\mu$ are much larger than $1/r$.)
	Moreover, there are at most $(\mu/4) m_2^{s} \binom{r-1}{s-1}$ edges of $G$, which have one vertex in the part of $x$ and one vertex that is not covered by $\cU$.
	Put differently, there are $m_2^s (1-1/s+\mu/2) \binom{r-1}{s-1}$ edges in $G$, which are $\cU$-partite and have a vertex in $x$.
	Every edge of $R$ incident to $x$ can at host at most $m_2^{s}$ of these edges.
	Every non-edge of $R$ incident to $x$ can host at most $4\nu m_2^s$ of these edges.
	Thus $m_2^{s} \deg(x) + 4\nu m_2^s \binom{r-1}{s-1} \geq m_2^s (1-1/s+\mu/2) \binom{r-1}{s-1}$.
	Since $4\nu = \mu/4$, we can solve for $\deg(x)$ to obtain the desired bound.
	}
	Since $r$ is divisible by $s$, there is a perfect matching $\cM$ in~$R$.

	{We will construct a $\cQ_2$-tiling by finding tuples inside each edge $X \in \mathcal{M}$, as follows.
	Each edge $X$ of $\cM$ corresponds to an $s$-partite family $\{U_1, \dotsc, U_s\}$.
	We construct a $\cQ_2$-tiling in $X$ greedily, by finding one tuple after another.
	Note that as long as the number of used vertices in each part $U_i$ is at most $3\nu m_2$, there remain at least $\nu m_2^s \geq d m_2^s$ edges among the leftover vertices.
	(Here we used that $d \leq \mu/16 = \nu$.)
	Within the unused vertices in $X$, select subsets $U'_i \subseteq U_i$, for all $1 \leq i \leq s$, of size $\beta m_2$ each such that $d_P(U'_1, \dotsc, U'_s) \geq d$.
	(This can be done using an averaging argument).
	We apply \cref{lem:density-regular-tuple} with $\beta m_2$ playing the rôle of $m$, to obtain an $(\eps, d)$-lower-regular $\tilde m$-balanced $s$-tuple, where $m_3 =  2 \gamma \beta  m_2 \leq \tilde m \leq \beta m_2$.
	Note that this lower-regular $s$-tuple belongs to $\cQ_2$.
	We iterate this, to greedily find a $\cQ_2$-tiling on at least $3\nu s m_2$ vertices, and at most $3\nu s m_2 + \beta m_2$ vertices, in the $X$-partite subgraph of $P$ hosted by $\bigcup X$.}
	
	Denote the union of these `fresh' $\cQ_2$-tilings, over all $X \in \mathcal{M}$, by $\cQ_{\text{fresh}}$.
	Note that $\cQ_{\text{fresh}}$ covers at least $3\nu m_2$ and at most $(3\nu + \beta)m_2$ vertices in each part of $\cU$.
	Since there are at least $(1 - \lambda){(n-sm_2)}/m_2$ parts in $\mathcal{U}$ outside $V(Q_1)$, it follows that $\cQ_{\text{fresh}}$ covers at least ${3\nu (1 - \lambda)n - s} \geq 2\nu \eta n = 4\beta n$ vertices outside $V(Q_1)$.
	
	Next, we define a $\cQ_2$-tiling $\cQ_{\text{rec}} \subset P$ in $V(Q_1) \sm V(\cQ_{\text{fresh}})$ to `recycle' what is left of $Q_1$.
	Let $\{V_1, \dotsc, V_s\}$ be the parts of an ${\tilde m}$-balanced $(\eps,d)$-lower-regular $s$-tuple in $Q_1$.
	Recall that, $\cQ_{\text{fresh}}$ covers between $3\nu m_2$ and $(3\nu + \beta)m_2$ vertices in each part of $\cU$.
	This implies in particular that the difference of leftover vertices between any two clusters of $V_1,\dots,V_s$ is at most $\beta {\tilde m}$.
	We may therefore pick subsets $U_1 \subset V_1, \dotsc, U_s \subset V_s$ of size $m^\ast = (1 - 3\nu - \beta){\tilde m}$ among the uncovered vertices.
	Note that $(U_1,\dots,U_s)$ is still $(2\eps,d')$-lower-regular for $d' \geq d-\eps \geq d/2$ and thus in $\cQ_2$.
	Indeed, for all choices of $X_1\subset U_1,$ $\dots,$ $X_s \subset U_s$ of size at least $2\eps m^\ast \geq \tilde m$ each, we have
	$ d_P(X_1,\dots,X_s)  \geq d - \eps$ by $(\eps,d)$-regularity of $(V_1,\dots,V_s)$.
	Let $\cQ_{\text{rec}} \subset P$ be the $\cQ_2$-tiling obtained this way.
	It follows that $\cQ_{\text{rec}} \cup \cQ_{\text{fresh}}$ covers all but $\beta v(Q_1) = \beta \lambda n$ vertices of $V(Q_1)$.
	It follows that $\cQ_{\text{fresh}} \cup \cQ_{\text{rec}}$ covers at least $(1 - \beta)\lambda n + 4 \beta n \geq (\lambda + 4\beta  - 2\beta )n = (\lambda + \nu \eta)n$ vertices.
\end{proof}

\section{Robust Hamilton connectedness}\label{sec:robust-ham-connect}

In this section, we prove \cref{thm:framework-connectedness-robust}.
We require the following variation of \cref{lem:inheritance-minimum-degree} (Inheritance Lemma), whose proof follows from analogous probabilistic arguments (see \cref{sec:inheritance} for more details).

\begin{lemma}[Inheritance Lemma II]\label{lem:inheritance-minimum-degree2}
	For $1/k,\,1/r,\,\eps \gg 1/s \gg 1/n$ and $\delta \geq 0$.
	Let $G$ be an $n$-vertex $k$-graph, and let $D$ be a subset of $d$-sets in $V(G)$.
	Suppose that for each $e \in D$ there exists $G_e \subseteq G$ such that $\deg_{G_e}(e) \geq (\delta + \eps) \tbinom{n-d}{k-d}$.
	Let $P$ be the $s$-graph consisting of all $s$-sets $S$ such that, for all $e \in D$ contained in~$S$, $\deg_{G_e[S]}(e) \geq (\delta + \eps/2) \tbinom{s-d}{s-d}$.
	Then $\delta_{r}(P) \geq  (1-e^{-\sqrt{s}}    )  \tbinom{n-r}{s-r}$.\qed
\end{lemma}

For a family $\cR$ of hypergraphs, an \emph{$\cR$-blow-up} is a blow-up whose reduced graph is in $\cR$.

\begin{proof}[Proof of \cref{thm:framework-connectedness-robust}]
	Given $k,{r}\leq s_1$, introduce constants $p_1,p_2,q_1,q_2,j_2,j_1$ with
	\[ \frac{1}{k},\frac{1}{r}, \frac{1}{s_1} \gg \frac{1}{p_1} \gg \frac{1}{p_2} \gg \frac{1}{q_1} \gg \frac{1}{q_2} \gg \frac{1}{j_2} \gg \frac{1}{j_1} \gg \frac{1}{s_2} \geq \frac{1}{s_3} \gg \frac{1}{n}. \]
	For brevity, let $\mathsf{Q} = \dcon \cap \dspa \cap \dape$.
	Since $\P$ is a family of $s_1$-vertex $k$-graphs which admits a Hamilton framework and $G$ satisfies $\P$ $s_1$-robustly,
	we can apply \cref{pro:framework-undirected-to-directed} to find  an $n$-vertex $[k]$-digraph $G' \subset \ori{C}(G) \cup \ori{C}(\partial G)$ that $p_1$-robustly satisfies $\mathsf{Q}$.
	By \cref{corollary:booster2}, $G'$ satisfies $(1 - \exp(-\sqrt{p_2}), {r}, p_2)$-robustly $\Del_k(\mathsf{Q})$.

	We introduce $\beta$ such that $1/k, 1/j_1, 1/p_2 \gg \beta \gg 1/s_2$.
	We apply \cref{lem:blow-up-support} (with $\Del_k(\mathsf{Q})$ playing the role of $\P$) to reveal that there are at least $\beta n^{k-1}$ edges $e \in G'^{(k-1)}$ such that there are $\beta n^{j_1p_2-(k-1)}$ many $j_1$-balanced $\Del_k(\mathsf{Q})$-blow-ups in $G'$ that contain $e$ as an edge.
	Let $C$ be obtained from $G'$ by deleting all $(k-1)$-edges that do not have this property.
	Fix $s_2 \leq {s} \leq s_3$.
	We claim that $C$ satisfies  ${s}$-robustly $\hamcon$.

	Let $P = \PG{G'}{\mathsf{Q}}{p_2}$.
	Set $Q = \PG{P}{\DegF{{r}}{1-1/p^2_2}}{{s}}$.
	In other words, $Q$ is the ${s}$-graph with vertex set $V(G)$ and an ${s}$-edge $S$ whenever the induced $p_2$-graph $P[S]$ has minimum ${r}$-degree at least $\left(1 - 1/p^2_2\right) \binom{{s}-{r}}{p_2-{r}}$.
	Since $G'$ satisfies $\Del_k(\mathsf{Q})$ $(1 - \exp(-\sqrt{p_2}), {r}, p_2)$-robustly, we have in particular that $\delta_{{r}}(P) \geq (1 - \exp(-\sqrt{p_2})) \tbinom{n - {r}}{p_2 - {r}}$.
	Hence, by \cref{lem:inheritance-minimum-degree}, it follows that $\delta_{{r}}(Q) \geq (1-1/(2{s}^2)) \tbinom{n-{r}}{{s}-{r}}$.

	Given an ${s}$-set $S \subset V(Q)$ and $e \in C^{k-1}$, we say that $S$ is \emph{$e$-tracking} if there is a $j_1$-balanced $\Del_k(\mathsf{Q})$-blow-up in $C[S] \cup \{e\}$ that contains $e$ as an edge.
	We say that $S$ is \emph{tracking} if it is $e$-tracking for all $e \in C^{k-1}[S]$.
	Finally, we call an ${s}$-set $S \subset Q$ \emph{bueno} if $C^{k-1}[S]$ has at least two disjoint edges.
	Denote by $Q' \subset Q$ the $n$-vertex subgraph of tracking and bueno edges.

	\begin{claim}
		We have $\delta_{{r}}(Q') \geq  (1-1/{s}^2)   \tbinom{n-{r}}{{s}-{r}}$.
	\end{claim}
	\begin{proofclaim}
		Fix a set of ${r}$ vertices $D$.
		Since $\beta \gg 1/{s}$, by \cref{lem:inheritance-minimum-degree} (applied with $r=0$, $d=0$ and $\delta=0$) we deduce that there are at most $1/(4s^2)\binom{n - {r}}{{s} - {r}}$ sets $S$ containing $D$ for which the induced $[k]$-digraph $C[S]$ has fewer than $(\beta/2) {s}^{k-1}$ $(k-1)$-edges.
		In the opposite case, $C[S]$ has at least $(\beta/2) {s}^{k-1} > \binom{{s}-1}{k-2}$ edges, and hence has at least two disjoint edges.
		In other words, the ${s}$-graph of bueno edges has minimum ${r}$-degree at least $(1 - 1/(4s^2))\binom{n - {r}}{{s} - {r}}$.

		Now we wish to apply \cref{lem:inheritance-minimum-degree2} in an auxiliary hypergraph.
		Let $H$ be the $j_1 p_2$-uniform hypergraph with an edge for every $j_1$-balanced $\Del_k(\mathsf{Q})$-blow-up with $p_2$-uniform edges in $G'$.
		By construction, each $e \in D$ belongs to at least $\beta n^{j_1p_2-(k-1)} \geq 2 \beta \binom{n - (k-1)}{j_1 p_2 - (k-1)}$ edges of $H$.
		We apply \cref{lem:inheritance-minimum-degree2} in $H$, with $j_1p_2$, ${s}$, $\beta$, $C$ playing the rôles of $k$, $s$, $\varepsilon$, $D$.
		We obtain that the ${s}$-uniform graph of tracking edges has minimum ${r}$-degree at least $(1 - 1/(4 s^2)) \binom{n - {r}}{{s} - {r}}$.

		Together with the above and the fact that $\delta_{{r}}(Q) \geq (1 - 1/(2s^2))\binom{n - {r}}{{s} - 2k}$, this confirms the claim.
	\end{proofclaim}

	Consider an edge $S \in Q'$, and let $D=C[S]$.
	Let $g,f \in D^{(k-1)}$ be disjoint, which exist since $S$ is bueno.
	To finish, we show that $D$ contains a Hamilton $(g,f)$-path.

	Since $Q' \subseteq Q$, from the definition of $Q$ we get that $\mathsf{Q}$ is $p_2$-robustly satisfied by~$D$.
	Using $1/k, 1/p_2 \gg 1/q_1 \gg 1/q_2 \gg 1/{s}$, we can apply \cref{lem:booster} twice and get that $D$ satisfies $\Del_{k}(\mathsf{Q})$ both $(1,q_1)$-robustly and $(2q_1,q_2)$-robustly.
	Introduce new constants $c, \eta$ such that $1/q_2 \gg c, \eta \gg 1/j_2$ hold, and let $m_1 = (\log {s})^{c}$ and $m_2 = (\log m_1)^{c}$.
	By \cref{pro:blow-up-cover} with ${s}$ playing the rôle of $n$, $D$ has a $(q_1,q_2)$-sized $(m_1,m_2,\eta)$-balanced $\Del_{k}(\mathsf{Q})$-cover whose shape $F$ is a path.

	By definition of $Q'$, we have that $S$ is $g$-tracking.
	Hence, there is a $p_2$-sized $j_1$-balanced $\P$-blow-up $R^g(\cV^g)$ that contains $g$ as an edge.
	Let $x$ be the first vertex of $F$ with corresponding blow-ups $R^{x}(\cV^{x})$.
	Select a set of size $j_1$ inside each set of the family $\cV^{x}$; let $\cV^{x}_2$ be those sets.
	This can be done keeping the sets disjoint from the vertices of $R^g(\cV^g)$.
	In this way, we obtain a $j_1$-balanced $q_1$-sized blow-up $R^{x}(\cV_2^{x})$, so that
	$\cV_2^{x} \cup \cV^g$ is a $j_1$-balanced family with $p_2 + q_1 \leq 2q_1$ clusters.
	We apply \cref{lem:connecting-blow-ups} with $\cV_2^{x} \cup \cV^g$ in place of $\cV$.
	This gives a $j_2$-balanced $\P$-blow-up $R^{gx}(\cW^{gx})$ such that $\cW^{gx}$ hits $\cV^g$ with $\cW^{gx}_{g} \subset \cW^{gx}$ and $\cV^{x}$ with $\cW^{gx}_x \subset \cW^{gx}$.
	We repeat the same process with $f$ to obtain $\P$-blow-ups $R^f(\cV^f)$ and  $R^{fy}(\cW^{fy})$, where $y$ is the last vertex of the path $F$.

	Note that these four blow-ups might intersect each other (quite a lot) outside of the dedicated hitting areas.
	To fix this, we may select subfamilies that are $j_1'$-balanced and $j_2'$-balanced, respectively, and vertex-disjoint (except for the hitting areas), keeping the names for convenience.
	This can be done with  $j_1' = j_1/3$ and $j_2' = j_2/3$ using a random partitioning argument.
	We delete the vertices of these additional four blow-ups from the cover  $(\{\cV^v\}_{v \in V(F)},  \{\cW^{e}\}_{e \in F}))$, again keeping the names for convenience.
	Since the additional blow-ups have only $2q_2(j_1'+j_2')$ vertices, $(\{\cV^v\}_{v \in V(F)},  \{\cW^{e}\}_{e \in F}))$ is still $(m_1,m_2,2\eta)$-balanced.
	Let $F'$ be the path obtained from $F$ by extending its ends with two vertices $g$ and $f$.

	To finish, we identify the vertices of $F$ with $1,\dots,\ell$ following the order of the path.
	For each $i \in V(F)$, we select disjoint edges $e_i,f_i \in R^{i}(\cV^{i})$ with $e_1 = f$ and $e_\ell = g$.
	We then apply \cref{pro:allocation-Hamilton-path} to find a \tight Hamilton $(f_i,e_{j})$-path $P^{ij}$ in $R^{ij}(\cW^{ij})$ for each $i \in [\ell]$ and $j=i+1$.
	Note that the families $\cV^{i}$ are still $(1\pm 2\eta)m_1$-balanced (resp. $(1\pm 2\eta)j_1'$-balanced) after deleting the vertices of these paths.
	We may therefore finish by applying \cref{pro:allocation-Hamilton-path} to find a \tight Hamilton $(e_i,f_{i})$-path $P^{i}$ in $R^{i}(\cW^{i})$ for each $i \in [\ell]$ with index computations modulo $\ell$.
\end{proof}

\section{Hamilton cycle allocation}\label{sec:allocation-cycle}

In this section, we allocate Hamilton cycles to suitable blow-ups.
Every \tight Hamilton cycle contains an almost perfect matching as a subgraph (or perfect if its order is divisible by $k$).
We thus begin our allocation with an analogous result for perfect matchings.
Afterwards, we embed the full cycle, proving \cref{pro:allocation-Hamilton-path}.

\subsection{Perfect matchings}

To allocate a perfect matching into a suitable blow-up, we require a few additional concepts introduced by Keevash and Mycroft~\cite{KM15}.
Our exposition follows a setup from the context of perfect tilings~\cite{Lan23}.
For a $k$-digraph $G$ and an edge $e \in G$, we denote by $\vn_e \in \NATS^{V(G)}$ the \emph{indicator vector}, which takes value $1$ at index $v$ if $v \in e$.\COMMENT{When writing $\INTS^{V(G)}$, we implicitly fix an ordering of $V(G)$. But this does not matter for the arguments.}
We set $\vn_v = \vn_e$ for $e=(v)$.
The \emph{lattice} of $G$ is the additive subgroup $\cL(G) \subset \INTS^{V(G)}$ generated by the vectors $\vn_e$ with $e \in G$.
We say that $\cL(G)$ is \emph{complete} if it contains all $\vecb b \in \INTS^{V(G)}$ for which $\sum_{v \in V(G)} \vecb b(v)$ is divisible by $k$.
Denote by $\udiv$ the set of uniform digraphs with complete lattice.
It is convenient to characterise lattice completeness as follows.
A \emph{transferral} in a lattice $\cL(G)$ is a vector of the form $\vn_v - \vn_u$ with $u,v \in V(G)$.

\begin{observation}\label{cor:lattice}
	The lattice of a  $k$-digraph is complete if and only if its lattice contains all transferrals.
\end{observation}
\begin{proof}
	Let $G$ be a $k$-digraph with $V=V(G)$.
	Clearly, if the lattice $\cL(G)$ is complete, then it contains all transferrals.
	For the other direction, consider a vector $\vecb b \in \INTS^{V}$ for which $\sum_{v \in V} \vecb b(v)$ is divisible by $k$.
	Fix a vertex $w \in V$ and an edge $e$ of $G$.
	(Such an edge exists since $G$ has a complete lattice.)
	Consider $\mathbf {b} \in \INTS^{V}$ with $\sum_{v \in V} \vecb b(v) = qk$.
	Informally, we assemble $\vecb b$ from transferrals as follows.
	We begin with $q$ copies of $e$ and transfer their weight to $w$.
	Then, we redistribute the accumulated weight of $w$ according to the demands of $\vecb b$.
	Formally,
	\begin{equation*}
		q  \vn_e +   q \sum_{v\in e}  (\vn_{w} - \vn_{v}) + \sum_{v \in V} \mathbf {b}(v) (\vn_v - \vn_{w}) = \mathbf {b}\,.
	\end{equation*}
	Since $\cL(H)$ contains all transferrals, this implies $\mathbf {b} \in \cL(H)$, and hence $\cL(G)$ is complete.
\end{proof}

Recall that in a \tightly connected $k$-digraph every two edges are on a common closed \tight walk.
The following result allows us to find a walk of controlled length that visits all edges.
The proof is a straightforward, if somewhat tedious application of the pigeonhole principle.

\begin{lemma}[{\cite[Proposition 7.3]{LS23}}]\label{lem:tight-walk}
	Let $G$ be a \tightly connected $[k]$-digraph on $s$ vertices and $1 \leq q \leq k$.
	Then $G$ contains a \tight closed walk of order at most $k^{2s}$ that visits every edge.
	Moreover, if $R \in \ape$, then there exists such a walk of order $q \bmod k$.
\end{lemma}

The next result connects the assumptions of \cref{pro:allocation-Hamilton-path} to lattice completeness.

\begin{observation}\label{obs:con-cap-implies-div}
	Every $[k]$-digraph in $\dcon \cap \dape$ has complete lattice.
\end{observation}

\begin{proof}
	Let $G$ be a $[k]$-digraph that satisfies $\dcon \cap \dape$.
	Let $u,v \in V(G)$ be arbitrary: by \cref{cor:lattice}, it suffices to show that the lattice $\cL(G)$ contains the {transferral} $\vn_v - \vn_u$.
	By \cref{lem:tight-walk}, every two edges of $G$ are on a common closed \tight walk of order $1\bmod k$.
	Moreover, every vertex is on an edge.
	By the above, there is a \tight walk beginning with $v$ and ending with $u$ of order coprime to $k$.
	Let $v_1,\dots,v_{k\ell+1}$ be the sequence of vertices of the walk, with $v = v_1$ and $u = v_{k\ell + 1}$.
	It follows that $G$ contains the edges $e_i = v_{ik+1},\dots,v_{ik+k}$ and $f_i =v_{ik+2},\dots,v_{ik+k+1}$ for $0 \leq i < \ell$.
	Hence $\vn_v - \vn_u = \sum_{0 \leq i < \ell} (e_i - f_i) \in \cL(G)$, as desired.
\end{proof}

Finally, we state a simple auxiliary result about perfect matchings.

\begin{observation}\label{lem:balancing-matching}
	For $1/k \gg \eta, 1/m$, let  $G$ be a $(1\pm\eta)m$-balanced blow-up of a complete $k$-graph of order $k+1$.
	Then $G$ has a perfect fractional matching.\qed
\end{observation}
\COMMENT{\begin{proof}
		Let $\cV$ be the partition of $V(G)$.
		Note that we can easily find a perfect fractional matching if all clusters have the same size.
		Simply give each edge the same weight.
		Otherwise, we proceed as follows.
		Denote by $d$ the largest differences in size of two clusters.
		So $d \leq 2\eta n$ by assumption.
		To balance the clusters of $\cV$, we construct a matching $M$ as follows.
		We begin with $M = \es$.
		While there is a cluster $V \in \cV$ that is larger than all others in $G - V(M)$, select edges $e_1,\dots,e_{k}$ such that each of these edges covers a vertex of $V$ and misses a unique cluster of the others.
		So after adding these edges, the largest difference between the size of two clusters must have gone down by $1$.
		This process stops after $d$ steps, since it never uses more than $(k-1)d$ vertices of each cluster.
		We can then finish by covering $G-V(M)$ with a perfect fractional matching as described before.
	\end{proof}}

Given these preparations, we are ready to embed perfect matchings into blow-ups following a strategy introduced in the setting of perfect tilings~\cite[Lemma 4.11]{Lan23}.

\begin{proposition}[Perfect matching allocation]\label{lem:allocation-perfect-matching}
	Let $1/k,\, 1/s \gg \eta, 1/m$.
	Let $R \in \Del_1 (\dspa \cap \udiv)$ be an $s$-vertex $[k]$-digraph.
	Let $\cV=\{V_x\}_{x \in V(R)}$ be a $(1 \pm \eta)m$-balanced partition with $|\bigcup \cV|$ divisible by $k$.
	Then the blow-up $R(\cV)$ has a perfect matching.
\end{proposition}

\begin{proof}
	Introduce $q$ with {$1/k ,\, 1/s \gg 1/q \gg \eta, 1/m$}.
	In the following, an edge $f \in R(\cV)$ is called \emph{$e$-partite} for an edge $e \in R$, if $f$ has one vertex in each of the clusters corresponding to $e$.

	We first pick a matching $\cM_{\res} \subset R(\cV)$ that acts as a `reservoir' by selecting $q$ disjoint $e$-{partite} edges for every $e \in R$.
	Let $\cV'$ be obtained from $\cV$ by deleting the vertices of $\cM_{\res}$.
	Note that $\cV'$ is still $(1\pm 2\eta)m$-balanced.

	Next we cover most of the vertices of $R(\cV')$ with a matching $\cM_{\cover}$, which is constructed as follows.
	We begin by finding a perfect fractional matching in $R(\cV')$.
	To do so, apply \cref{lem:balancing-matching} to identify a family $S_1,\dots,S_\ell$ of pairwise vertex-disjoint $\cV'$-partite $(s-1)$-sets,
	such that $\bigcup \cV' = S_1\cup \dots \cup S_\ell$.
	Since  $R \in \Del_1 (\dspa)$, there is a perfect fractional matching $\omega_i$ of $(R(\cV'))[S_i]$ for each $1\leq i \leq \ell$.
	It follows that $\omega = \omega_1 + \dots + \omega_\ell$ is a perfect fractional matching of $R(\cV')$.
	Let $w_e$ be the sum of $\omega(f)$ over all $e$-partite edges $f \in R(\cV')$.
	We obtain the matching $\cM_{\cover} \subset R(\cV')$ by selecting $\lfloor w_e   \rfloor$ edges  for each $e \in R$.
	Let $\cV''$ be obtained from $\cV'$ by deleting the vertices of $\cM_{\cover}$.
	Note that $\cV''$ contains at most $e(R) \leq s^k$ vertices, where the error arises from rounding $\lfloor w_e   \rfloor$ at each edge $e \in R$.

	To finish, we construct a third matching $\cM_{\ominus} \subset R(\cV)$ that contains the remaining vertices.
	This is done by picking a matching $\cM_{\oplus} \subset \cM_{\res}$ from the reservoir such that the vertices $V(\cM_{\oplus}) \cup \bigcup \cV''$ span a perfect matching $\cM_{\ominus}$ in $R(\cV)$.
	Formally, let $\mathbf {b} \in \INTS^{V(R)}$ count the size of each part of $\cV''$, that is $\mathbf {b} (x) = |V_x''|$ for every $x \in V(R)$.
	Note that the lattice $\cL(R)$ is complete since $R \in {\udiv}$.
	Moreover, $\sum_{x \in  V(R)} \mathbf {b}(x)$ is divisible by~$k$ by assumption.
	It follows that $\mathbf {b}\in \cL(R)$.
	So there are integers $c_e$, one for every $e \in R$, such that $\mathbf {b} = \sum_{e \in R} c_e \vn_e$.
	By the choice of $q$, we can ensure $-q \leq c_e \leq q$ holds for each $e \in R$.
	Let $E_{\oplus}$ and $E_{\ominus}$ denote the sets of edges $e \in R$ for which~$c_e$ is positive and negative, respectively.
	Let $\cM_{\ominus} \subset \cM_{\res}$ be a submatching obtained by adding $|c_e|$ many $e$-partite edges from $\cM_{\res}$, for each $e \in E_{\ominus}$.
	Since
	\begin{equation*}
			\sum_{e \in E_{\oplus}} c_e \vn_e  = \mathbf {b} - \sum_{e \in E_{\ominus}} c_e \vn_e = \mathbf {b} + \sum_{e \in E_{\ominus}} |c_e|  \vn_e,
	\end{equation*}
	we may finish by selecting a matching $\cM_{\oplus} \subset R(\cV)$ on the vertices of $V(\cM_{\ominus}) \cup \bigcup \cV''$ by adding $c_e$ many $e$-partite edges for every $e \in E_{\oplus}$.
\end{proof}

We remark that the relation $1/k ,\, 1/s \gg 1/q$ in the previous proof is needed in relation with the following problem: given $\mathbf {b}\in \cL(R)$, what is the best bound $q$, in terms of $s = |V(G)|$ and $r = \Vert \mathbf{b} \Vert_\infty \geq k$, so that there exist $c_e$, one for each $e \in R$, such that $|c_e| \leq q$ and $\mathbf{b} = \sum_{e \in R} c_e \vn_e$?
A bound of the form $q \leq 2r \cdot (2s)^{s}$ can be proven using Steinitz' lemma (see \cite{barany2008power}).

\subsection{Hamilton paths}

Now we are ready to embed Hamilton paths into suitable blow-ups.

\begin{lemma}\label{lem:linked-edges}
	Suppose that $G$ is a $k$-digraph and $v \in V(G)$ such that $G, G-v \in \dcon$.
	Then there are edges~$e,f$ that intersect in $k-1$ vertices such that $v \in f$ and $v \notin e$.
\end{lemma}
\begin{proof}
	By assumption there are distinct edges $e',f' \in G$ such that $v \in f'$ and $v \notin e'$.
	Moreover, there is a \tight walk $W$ starting with $f'$ and ending with $e'$.
	We can then take $f$ to be the last edge of $W$ that contains $v$, and $e$ to be the edge that follows after $f$.
\end{proof}

\begin{proof}[Proof of \rf{pro:allocation-Hamilton-path}]
	For $1/k, 1/s \gg \eta, 1/m$, let $R \in \Del_k (\dcon \cap \dspa \cap \dape)$ be an $s$-vertex $[k]$-digraph.
	Let $\cV=\{V_x\}_{x \in V(R)}$ be a quasi $(1 \pm \eta)m$-balanced partition with exceptional cluster $V^\ast$.
	Let $f_1,f_2 \in (R(\cV)-V^\ast)^{(k-1)}$ be vertex-disjoint.
	Our goal is to find a \tight Hamilton $(f_1,f_2)$-path in $R(\cV)$.

	Let us write $V^\ast = V_{{x^\ast}}$ with $x^\ast \in V(R)$.
	For every $e \in (R-x^\ast)^{(k)}$, fix a \tight path $P_e$ in $R(\cV)$ of order $v(P_C) = 3k$ whose vertices follow the parts of $\cV$ according to the ordering of $e$.
	Since $m$ is much larger than $k$ and $s$, we can assume that these paths are pairwise disjoint, and also disjoint from $f_1$ and $f_2$.

	We handle the exceptional vertex $v^\ast$ as follows.
	By \cref{lem:linked-edges}, there are edges $e, f \in R^{(k)}$ that share $k-1$ vertices such that ${x^\ast} \in f$, ${x^\ast} \notin e$.
	We use this configuration to obtain a \tight path of order at most $6k$ that contains the vertex of $v^\ast$ and begins and ends in the clusters of $f$ (in the ordering of~$f$).
	We also ensure that~$P^\ast$ is disjoint from the paths of type $P_e$ specified above as well as disjoint from $f_1$ and $f_2$.

	Since $R-{x^\ast} \in \dape$, there is an odd closed walk $W$ of order coprime to $k$ in $R-{x^\ast}$.
	We can assume that $W$ has $w \leq 4^{2^k}$ vertices.
	We select a path $P_{\dape}$ that contains as a subpath a segment of order $(k-1)w$ that follows the clusters of $W$, meaning that $P_{\dape}$ `winds $(k-1)$-times around' $W$.
	We ensure that $P_{\dape}$ is disjoint from all paths selected so far including $f_1$ and $f_2$.
	Importantly, for any $1\leq i \leq k-1$, we may select an alternative path $P_{\dape}'$ with $V(P_{\dape}') \subset V(P_{\dape})$ and the same ends as $P_{\dape}$  such that $v(P_{\dape}) - v(P') \equiv i \bmod k$.
	Indeed, such a $P_{\dape}'$ would simply `wind around' $W$ fewer times, and then connect to the rest of $P_{\dape}$.

	Since $R \in \Del_k(\con)$, the $k$-digraphs $\adh(R)$ and $\adh(R-{x^\ast})$ are \tightly connected.
	By \cref{lem:tight-walk}, any two disjoint elements of $R-{x^\ast}$ are connected by a \tight walk~$W$ of order at most $4^{s^k}$.
	This allows us to connect up all the paths of type $P_e$, $P_{x^\ast}$ and $P_{\dape}$ to one `skeleton' $(f_1,f_2)$-path $P_{\text{skel}}$ of order at most $8^{s^k}$.
	We then ensure that $n-v(P_{\text{skel}}) \equiv 0 \bmod k$ by replacing $P_{\dape}$ with a suitable alternative path $P_{\dape}'$ as described above.

	Let $\cV'$ be obtained from $\cV$ by deleting the vertices of $P_{\text{skel}}$ (including the now empty exceptional cluster).
	Note that $\cV'$ is still $(1\pm 2\eta)m$-balanced.
	By \cref{lem:allocation-perfect-matching}, there is a perfect matching $\cM$ of $(R-x^\ast)(\cV')$.
	We then define a Hamilton $(f_1,f_2)$-path $P$ by adding $\ell k$ vertices to the `middle' of every path $P_e$, where $\ell$ counts the number of edges in $\cM$ that are in the respective parts of $\cV$ corresponding to the edge $e$.
	Note that this is possible because we chose $P_e$ of order $v(P_e)=3k$.
\end{proof}

A similar proof gives the following corollary.

\begin{corollary} \label{corollary:hamallocationforbandwidth}
	{Let $1/k, 1/s \gg \eta, 1/m$.
	Let $R \in \Del_k (\dcon \cap \dspa \cap \dape)$ be an $s$-vertex $[k]$-digraph.
	Let $\cV=\{V_x\}_{x \in V(R)}$ be a $(1 \pm \eta)m$-balanced partition.
	Moreover, let $e,f,g \in R(\cV)^{(k)}$ each containing an element of $R(\cV)^{(k-1)}$.
	Then $R(\cV)$ has a \tight Hamilton cycle that contains each of $e,f,g$ as a subpath.}
\end{corollary}

To prove the corollary, we follow the proof as above, only modifying the construction of $P_{\text{skel}}$ such that it becomes a cycle that contains $e,f,g$ as subpaths.
(This is possible by `winding around' the corresponding $k$-edges of $R(\cV')^{(k)}$).

\section{Path blow-up allocation}
\label{sec:allocation-path-blow-up}

In this section, we prove \rf{pro:allocation-bandwidth-frontend}.
The strategy follows the ideas developed for a similar result in the graph setting~\cite{LS23}.
Let us outline the argument in the case when $k=t$.
Our goal is to embed the blow-up of a path $H$ into a blow-up $R(\cV)$.
To begin, we identify a suitable blow-up $C(\cW)$ of a cycle $C$ in $R(\cV)$ such that $\cW$ refines~$\cV$.
This is possible since $R$ robustly satisfies $\con$, $\spa$ and $\ape$.
We then embed $H$ into $C(\cW)$ as follows.
To begin, we map $H$ onto $C(\cW)$ following a random walk on $C$.
	{It turns out that (using that $s-1$ is coprime to $t$) we can control the order of $C$, taking it coprime to $t$.
		We can use this to get a} homomorphism $\phi$ from $H$ to $C(\cW)$ that is nearly injective.
We then manually adjust the (images) of the end tuples of $H$ as required in the statement.
To finish the allocation, we modify $\phi$ by swapping certain vertices along the ordering of $C$.
Using again that $C$ has a controlled order, this results in the desired embedding.

The rest of this section is organised as follows.
In the next section, we prove \cref{pro:allocation-bandwidth-frontend} subject to a result on random embeddings (\cref{pro:allocation-random}).
In \cref{sec:markov}, we introduce some probabilistic terminology and results.
Finally, we prove \cref{pro:allocation-random}  in \cref{sec:allocation-random}.

\subsection{Embedding blow-ups of paths}\label{sec:allocation-blow-up-path-final}

In the following, we derive the main result of this section.
To formalise the above discussion, we introduce the concept of buffer vertices.

\begin{definition}[Buffer]
	Let $H$ be a $k$-graph, let $C$ be the $(t-k+1)$st power of a cycle $C$ on the vertices $\{c_1, \dotsc, c_s\}$, and let $\phi\colon V(H) \rightarrow V(C)$ be a homomorphism.
	We say $\phi$ has an \emph{$(\alpha, s)$-buffer} $I \subseteq V(H)$ if
	\begin{enumerate}[(B1)]
		\item no two vertices in $I$ are contained in a common edge,
		\item for every $c_i \in V(C)$ there are at least $\alpha n / s$ vertices of $I$ mapped to $c_i$, and
		\item for each vertex $v \in I$ with $\phi(v) = c_i$ and each $e \in E(H)$ which contains $v$, we have $\phi(e) \subseteq \{ c_{i+1}, \dotsc, c_{i+t-1} \}$.
	\end{enumerate}
\end{definition}

Buffer vertices will be used in the last step of the proof, where we need to manually adjust a given embedding.
The next result implements the random allocation of the second step, and gives an embedding which is almost injective, with many buffer vertices.

\begin{proposition}[Randomised allocation] \label{pro:allocation-random}
	Let $1/t \gg \alpha$ and $ 1/s,\, 1/t,\, \xi \gg \rho \gg 1/n$.
	Suppose $s, t$ are coprime.
	Let $H$ be a blow-up of the $(t-k+1)$st power of a \tight path on $n$ vertices with clusters of size at most $\rho n$.
	Let $C$ be the $(t-k+1)$st power of a cycle $C$ on $s$ vertices.
	Then there exists a homomorphism $\phi$ from $H$ to $C$, where each vertex of $C$ is the image of $(1 \pm \xi) n/s$ vertices.
	Moreover, $\phi$ has an $(\alpha, s)$-buffer.
\end{proposition}

We require a simple fact about cycles, whose proof we omit.

\begin{lemma} \label{lemma:fixedlengthwalkincycle}
	For all coprime integers $s,t \geq 2$, there is $\ell = \ell(s,t) \geq 1$ with the following property.
	Let $C$ be a $t$-uniform {oriented cycle} of order $s$.
	Then for each pair of oriented edges $e, f \in C$, there is an $(e,f)$-walk in $C$ of order exactly $\ell$. \hfill $\qed$
\end{lemma}
\COMMENT{\begin{proof}
		We show that $\ell = (t+1)st$ works.
		Consider ordered edges $e,f$
		Since $s, t \geq 2$ are coprime, there is an $(e,f)$-walk in $C$ of order $w \leq (t+1)s$ with $w \equiv 0 \bmod t$.
		We then increase the order of this walk to exactly $(t+1)st$ by winding around $e$.
	\end{proof}}

\begin{proof}[Proof of \cref{pro:allocation-bandwidth-frontend}]
	We begin by choosing constants $\xi, \alpha, p$ such that
	\[   1/t \gg \alpha \gg 1/p \gg \eta \gg \xi \gg \rho \gg 1/m, \]
	and also $1/s \gg 1/p$.
	Moreover, we also choose $p$ to be coprime to $t$.

	\medskip \noindent \emph{Step 1: Finding a Hamilton cycle.}
	Recall that $V_{x^\ast}$ is the exceptional cluster of $\cV$, and denote $\cV' = \cV \sm \{V_{x^\ast}\}$.
	Let $n$ be the order of $R(\cV)$.
	So $\mathcal{V}'$ is $(1\pm \eta)m$-balanced, and by assumption we have
	\begin{equation}
		n-1 = \sum_{x \in V(R-x^\ast)} |V_x| = (1\pm \eta) (s-1)m.
		\label{equation:sizesV}
	\end{equation}
	Since $\eta$ is sufficiently small, we deduce that $m = (1\pm 2 \eta)n/(s-1)$.
	Hence, $\mathcal{V}'$ is $(1\pm 4\eta)n/(s-1)$-balanced.

	Let $\cW$ be a partition which refines $\cV'$ by (greedily) partitioning each cluster of $\cV'$ into $p$ subclusters of approximately the same size.
	Set $s' = (s-1)p$.
	Using that $\eta, 1/p \gg 1/m$, we see that $\cW$ is $(1\pm 5\eta)n/s'$-balanced.
	For each $x \in V(R-x^\ast) $, denote by $U_x$ the set of clusters of $\cW$ that are contained in $V_x$, and set $\cU = \{U_x\}_{x \in V(R-x^\ast)}$.
	Note that $\cU$ is a $p$-balanced partition of the set $\mathcal{W}$.
	Moreover, $s'$ is the number of clusters in $\cW$.
	Both $s-1$ and $p$ are coprime to $t$, so $s'$ is also coprime to $t$.

	Now recall that $K \subset \ori{C}(\hat K) \cup \ori{C}(\partial \hat K)$ is an $s$-vertex $[t]$-digraph that satisfies $\Del_{t+1}(\dcon \cap \dspa \cap \dape)$.
	It follows by \cref{lem:linked-edges} that there are edges $g^\ast, g \in K^{(t)}$ that intersect in $t-1$ vertices and such that $x^\ast \in g^\ast$ and $x^\ast \notin g$.
	Let $y \in V(K)$ be the unique vertex in $g \setminus g^\ast$.
	Without loss of generality (by considering a cyclic shift, if necessary) we can assume the tuple $g$ begins with $y$.
	Set $K' = K-x^\ast$.
	We modify $\mathcal{W}$ by choosing any cluster in $U_y$ and adding a new vertex to it; we update $\mathcal{U}$ accordingly.
	Note that, after this final modification,
	\begin{enumerate}[  label=({A}{{\arabic*}})]
		\item \label{item:thirdallocation-Wbalanced} $\mathcal{W}$ is a $(1 \pm 6\eta)n/s'$-balanced partition,
		\item \label{item:thirdallocation-sizesnoty} for each $x \in V(R-x^\ast-y)$, we have $\sum_{W \in U_x} |W| = |V_x|$, and
		\item \label{item:thirdallocation-sizesy} $\sum_{W \in U_y} |W| = |V_y|+1$.
	\end{enumerate}

	Suppose that $e, f$ are two ordered $t$-edges of $K^{(t)}-x^\ast$ as in the statement.
	Since $e = (e_1, \dotsc, e_t) \in K'^{(t)}$,  we can select $e' = (e'_1, \dotsc, e'_t) \in K'(\mathcal{U})^{(t)}$ where $e'_i \in U_{e_i}$ is chosen arbitrarily for each $1 \leq i \leq t$.
	We also define $f', g' \in K'(\mathcal{U})^{(t)}$ analogously, starting from $f$ and $g$.

	Observe that $K' \in \Del_t(\dcon \cap \dspa \cap \dape)$.
	Thus, we may apply \cref{corollary:hamallocationforbandwidth} with $s-1, p, K', \cU, e', f', g'$ playing the rôles of $s, m, R, \cV, e, f, g$.
	We obtain a Hamilton cycle $C_K$ of $K'(\cU)$, which contains $e',f',g'$ as subpaths.
	For convenience, we write $V(C_K) = \{c_1, \dotsc, c_{s'}\}$, so that $c_1, \dotsc, c_{s'}$ respects the cyclic order of~$C_K$.
	Note that each $c_i$ corresponds to a cluster in $\cW$, which we will denote by $W_i$.
	Together with \eqref{equation:sizesV}, \ref{item:thirdallocation-sizesnoty}, \ref{item:thirdallocation-sizesy}, this gives
	\begin{enumerate}[label=({A}{{\arabic*}}), resume]
		\item \label{item:thirdallocation-sizesW} $n = \sum_{j=1}^{s'} |W_{j}|$.
	\end{enumerate}
	Set $R' = R - x^\ast$.
	Since $K' \subseteq K^{(t)} \subseteq \ori{C}(K_t(R))$ and $V(K') = V(R-x^\ast)$, the cyclic sequence $(c_1, \dotsc, c_{s'})$ of $C_K$ is also the cyclic sequence of the $(t-k+1)$st power of a Hamilton cycle in $C \subset R'[\mathcal{U}]$.

	\medskip \noindent \emph{Step 2: Approximate allocation.}
	By assumption, there is a homomorphism $\phi_1$ from $H$ to the $(t-k+1)$st power of a \tight path $P$ such that $\phi_1$ takes each cluster of $H$ to the corresponding vertex $p_1,p_2,\dots $ of $P$.
	We apply \cref{pro:allocation-random} with $2 \alpha, s'$ playing the rôles of $\alpha, s$.
	This gives a homomorphism $\phi_2\colon V(P) \rightarrow V(C)$ such that $\phi = \phi_2 \circ \phi_1$ is a homomorphism from $H$ to $C$, and
	\begin{enumerate}[label=({A}{{\arabic*}}), resume]
		\item \label{item:thirdallocation-preimages} each $c_i$ is the image of $(1 \pm \xi)n/s'$ vertices, and
		\item \label{item:thirdallocation-buffer}
		      there exists a $(2\alpha, s')$-buffer $I \subseteq V(H)$.
	\end{enumerate}

	\medskip \noindent \emph{Step 3: Adjusting the endpoints.}
	Recall that $H$ starts with clusters $E_1,\dots,E_t \subset V(H)$ and ends with clusters $F_1,\dots, F_t \subset V(H)$.
	Let $\ell \geq 1$ be the value given by \cref{lemma:fixedlengthwalkincycle} applied to $C$.
	Let $i'\in C$ be the image of the $(\ell-t+1)$st edge of $P$, formally $ i' = (i'_1, \dotsc, i'_t) = \phi_2( p_{\ell-t+1}, \dotsc, p_{\ell})$.
	By \cref{lemma:fixedlengthwalkincycle}, it follows that $C$ admits an oriented walk $w_1 \dotsb w_{\ell}$ from $e' = (e'_1, \dotsc, e'_t)$ to $i'$.
	Modify $\phi_2$ on the vertices $p_1, \dotsc, p_{\ell}$ by setting $\phi_2(p_j) = e'_j$ for each $1 \leq i \leq t$, and $\phi_2(p_j) = w_{j-t}$ for each $1 \leq j \leq \ell$.
	We perform a similar alteration on $\phi_2$ on the last $\ell$ vertices of $V(P)$, to map the last $t$ vertices of $V(P)$ to the ordered tuple $f'$.
	For convenience, we keep the names $\phi_2$ and $\phi = \phi_2 \circ \phi_1$ for the updated homomorphism after this alteration.
	By definition of~$e'$ and $f'$, we have
	\begin{enumerate}[label=({A}{{\arabic*}}), resume]
		\item \label{item:thirdallocation-correctendploints} $\phi(E_i) \subseteq V_{e(i)}$ and $\phi(F_i) \subseteq V_{f(i)}$ for each $i\in [t]$.
	\end{enumerate}

	Since the alterations only modified the images of at most $2\ell$ vertices of $P$ and each cluster of $H$ has size at most $\rho m \leq (1+\eta)n/s'$, the previous properties of $\phi$ are maintained up to a small change in the constants.
	More precisely, the number of vertices of $H$ mapped to any given vertex $c_j \in V(C)$ were altered by at most $2\ell \rho (1+\eta)n/s'\leq \xi n /s'$, where we appealed to the hierarchy of constants in the last inequality.
	Let $Y = E_1 \cup \dotsb \cup E_t \cup F_1 \cup \dotsb \cup F_t$.
	We deduce from \ref{item:thirdallocation-preimages} and \ref{item:thirdallocation-buffer} that
	\begin{enumerate}[label=({A}{{\arabic*}}), resume]
		\item \label{item:thirdallocation-preimages2} each $c_i$ is the image of $(1 \pm 2\xi)n/s'$ vertices, and
		\item \label{item:thirdallocation-buffer2}
		      there exists an $(\alpha, s')$-buffer $I' \subseteq V(H) \setminus Y$.
	\end{enumerate}

	\medskip \noindent \emph{Step 4: Achieving perfection.}
	For each $1 \leq j \leq s'$, let $z_j = |W_{j}| - |\phi^{-1}(c_j)|$, where as we recall $W_j \in \cW$ is the cluster corresponding to the vertex $c_j \in V(C)$.
	Hence, $z_j$ is the difference between our target number of vertices of $H$ allocated to $c_j$ and the current number of vertices allocated via $\phi$.
	From \ref{item:thirdallocation-sizesW} we deduce that $\sum_{j=1}^{s'} z_j = 0$.
	Moreover,  $-7 \eta n/s' \leq z_j \leq 7 \eta n/s'$ holds for each $1 \leq j \leq s$ by \ref{item:thirdallocation-Wbalanced} and \ref{item:thirdallocation-preimages2}.
	In particular, we have $\sum_{j=1}^{s'} |z_j| \leq 7 \eta n$.
	Let $D$ be the directed graph on vertex set $\{1, \dotsc, s'\}$ where the edge $(i, i+t)$ is added for every $1 \leq i \leq s'$, with indices modulo~$s'$.
	Since $t$ and $s'$ are coprime, $D$ is in fact a directed cycle of length~$s'$.

	We will define, iteratively, sets $B_j \subseteq \phi^{-1}(c_j) \cap I'$ as follows.
	The construction will be done in steps $T \geq 0$.
	In each step $T$, we shall guarantee that $|B_j| \leq T$ for all $j$.
	Initially, in step $T = 0$, all the sets are empty.
	Define, for each $1 \leq j \leq s'$,
	\[ \delta(j) = |B_{j-t}| - |B_{j}| - z_j. \]
	If $\delta(j) = 0$ for all $1 \leq j \leq s'$; or $T \geq 7 \eta n$ steps have been executed, we stop the construction.
	Otherwise, suppose that $T < 7 \eta n$ steps have been executed, and there are $j, j'$ such that $\delta(j) > 0$ and $\delta(j') < 0$.
	Since $D$ is a cycle, there is a directed $(j, j')$-path $v_1 \dotsb v_r$ in $D$.
	For each $1 \leq i < r$, add some new unused vertex from $\phi^{-1}(c_{v_i}) \cap I'$ to $B_{v_i}$.
	Indeed this is possible, since by \ref{item:thirdallocation-buffer2} each set $\phi^{-1}(c_{v_i}) \cap I'$ contains at least $\alpha n / s' > 7 \eta n > T \geq |B_i|$ vertices (in the first inequality we used crucially that $s' = (s-1)p$ and $\eta \ll 1/p, 1/s, \alpha$).
	Note that this change decreases $\delta(j)$ by one, increases $\delta(j')$ by one, and keeps every other value intact: this implies that $\sum_{j=1}^{s'} |\delta(j)|$ decreases by two.

	We claim that this construction ends with $\delta(j) = 0$ for each $j \in [s']$.
	Note that, initially, $\sum_{j=1}^{s'} |\delta(j)| = \sum_{j=1}^{s'} |z_j| \leq 7 \eta n$.
	Also, as explained, each movement decreases $\sum_{j=1}^{s'} |\delta(j)|$ by two.
	Hence, after doing at most $7\eta n/ 2$ movements, we arrive at a family of sets $B_1, \dotsc, B_{s'}$ such that $\delta(j) = 0$ for each $j \in [s']$, as required.

	We define $\phi'\colon V(H) \rightarrow V(C)$ by defining $\phi'(B_j) = c_{j+t}$ for each $1 \leq j \leq s$; and $\phi'(v) = \phi(v)$ for each other vertex.
	Note first that $B_j \subseteq I'$ for each $j \in [s']$.
	Moreover, $I' \cap Y = \emptyset$ by \ref{item:thirdallocation-buffer2}.
	So \ref{item:thirdallocation-correctendploints} implies that
	\begin{enumerate}[label=({A}{{\arabic*}}), resume]
		\item \label{item:thirdallocation-correctendplointsfinal} $\phi'(E_i) \subseteq V_{e(i)}$ and $\phi'(F_i) \subseteq V_{f(i)}$ for each $i\in [t]$.
	\end{enumerate}

	We remark that $\phi'$ is a homomorphism from $H$ to $C$.
	To see this, let $e \in E(H)$ be arbitrary.
	If $\phi(e) = \phi'(e)$, this is immediate, so assume otherwise.
	Each edge $e \in H$ contains at most one vertex from $I$, by the definition of buffer.
	Thus, if $e \in H$ is such that $\phi(e) \neq \phi'(e)$, there must exist a unique $v \in B_j \cap e$ for some $1 \leq j \leq s'$.
	Again by the definition of the buffer, it follows that $\phi(e) \subseteq \{ c_j, c_{j+1}, \dotsc, c_{j+t-1} \}$.
	Since $\phi'$ and $\phi$, restricted to $e$, only differ in the image of $v$, we have $\phi'(e) \subseteq (\phi(e) \setminus \{ \phi(v) \}) \cup \{ \phi'(v) \} = \{ c_{j+1}, \dotsc, c_{j+t} \}$.
	As $C$ is the $(t-k+1)$st power of a Hamilton cycle, this ensures that $\phi'(e)$ is an edge in $C$, as required.

	Next, note that for each $1 \leq j \leq s'$, the number of vertices of $H$ allocated to $c_j$ via $\phi'$ is
	\begin{equation}
		|(\phi')^{-1}(c_j)| = |\phi^{-1}(c_j)| - |B_j| + |B_{j-t}| = |\phi^{-1}(c_j)| + z_j = |W_{j}|.
		\label{equation:exactWj}
	\end{equation}

	\medskip \noindent \emph{Step 5: Incorporating the exceptional vertex.}
	To finish, we obtain from $\phi'$ a homomorphism $\phi''$ from $H$ into~$R$, which essentially consists in taking care of the exceptional cluster $x^\ast \in V(R)$.
	Recall that in Step 1, we selected an ordered tuple $g = (g_1, \dotsc, g_t)$ with $g_1 = y$, such that the power of a cycle $C$ contains the ordered tuple $g' = (g'_1, \dotsc, g'_{t})$, where $g'_i \in U_{g_i}$ for all $i \in [t]$.

	Select any $v^\ast \in V(H) \cap (I' \setminus Y)$ such that $\phi(v^\ast) = \phi'(v^\ast) = g'_1$.
	(Such a vertex must exist: the argument is the same we used to ensure the existence of vertices in the buffer before using \ref{item:thirdallocation-buffer2}.)
	Define $\phi''(v^\ast) = x^\ast$.
	For every $v \neq v^\ast$, we have $\phi(v) = c_j$, and there exists $x \in V(R-x^\ast)$ such that $c_j \in U_x$; we set $\phi''(v) = x$.

	To verify that $\phi''$ is a homomorphism it is enough to check the image of edges $e \ni v^\ast$.
	Since $v^\ast \in I' \setminus Y$ and $I'$ is a buffer for $\phi$, we have $\phi(v^\ast) = g'_1$ and $\phi(e \setminus \{v^\ast\}) \subseteq \{ g'_{i+1}, \dotsc, g'_{i+t-1} \}$.
	Hence, $$\phi''(e) \subseteq \{x^\ast\} \cup \{ g_{i+1}, \dotsc, g_{i+t-1} \} = g^\ast \in K^{(t)} \subseteq K_t(R),$$ so indeed $\phi''(e) \in E(R).$
	From \ref{item:thirdallocation-sizesnoty}--\ref{item:thirdallocation-sizesy} and \eqref{equation:exactWj} we get that each $x \in V(R)$ is the image of exactly~$|V_x|$ vertices from $H$.
	Hence $R[\cV]$ contains a copy of $H$, as desired.
	Lastly, we have $\phi''(E_i) \subseteq V_{e(i)}$ and $\phi''(F_i) \subseteq V_{f(i)}$ for each $i\in [t]$ by \ref{item:thirdallocation-correctendplointsfinal}.
\end{proof}

It remains to show \cref{pro:allocation-random}.
In the next section, we introduce a few probabilistic instruments for this purpose.
The proof of \cref{pro:allocation-random} follows in \cref{sec:allocation-random}.

\subsection{Markov chains}\label{sec:markov}
For the random allocation, we need some probabilistic tools.
We work with finite Markov chains.
The following exposition covers the tools and concepts needed for the proof, see also \cite[Section 10.3]{LS23} for more details.

\subsubsection{Basics}

Markov chains are random processes $\{X_i\}_{i \geq 1}$ taking values in a finite \emph{state space} $\mathcal{X}$.
The trajectories of the random process are determined by a transition matrix $(P_{ij})_{i, j \in \mathcal{X}}$, where the value $P_{ij}$ represents the probability of evolving from $i$ to $j$ in a single step.
This implies that all entries in $P$ are in $[0,1]$ and the entries of each row sum up to $1$.
The associated digraph (with loops and antiparallel edges allowed) of the chain $D = D(P, \mathcal{X})$ has vertex set $\mathcal{X}$ and an arc $(i, j)$ if $P_{ij} > 0$.
A finite Markov chain is \emph{ergodic} if $D$ is strongly connected and contains two closed walks whose lengths are coprime.
In the language of Markov chains, this is equivalent to being \emph{irreducible} and \emph{aperiodic}.
The characterization in terms of the digraph properties applies to finite Markov chains.
A probability distribution $\pi$ on $\mathcal{X}$ (understood as a row vector) is said to be \emph{stationary} for the Markov chain if $\pi = \pi P$.
A basic fact from Markov chain theory states that ergodic Markov chains admit a unique stationary distribution.

\subsubsection{Concentration}

We need the following concentration inequality for Markov-dependent random variables~\cite{FJS21} (cf. \cite[Theorem 10.2]{LS23}).
Given a function $f\colon \mathcal{X} \rightarrow \REALS$ and a probability distribution $\pi$ over $\mathcal{X}$, we write $\Exp_\pi[f] = \sum_{x \in \mathcal{X}} \pi(x)f(x)$ for the expected value.

\begin{theorem} \label{theorem:markovconcentration}
	Let $\{X_i\}_{i \geq 1}$ be a finite ergodic Markov chain with state space $\mathcal{X}$ and transition matrix $P$.
	There exists $C > 0$, depending on $P$ only, such that the following holds.
	Suppose $X_1$ is distributed according to the stationary distribution $\pi$ of the chain.
	Let $\{f_i\}_{i \geq 1}$ be functions such that $f_i\colon \mathcal{X} \rightarrow [x_i, y_i]$ for $x_i < y_i$.
	Then, for every $\eps > 0$ and $s \in \NATS$,
	\[ \Pr\left[ \left| \sum_{i=1}^s (f_i(X_i) - \Exp_\pi[f_i]) \right| > \eps  \right] \leq 2 \exp \left( - \frac{2 C \eps^2}{\sum_{i=1}^s (y_i - x_i)^2} \right). \]
\end{theorem}

\subsubsection{Multiblock chains}

Let $\{X_i\}_{i \geq 1}$ be a Markov chain with state space $\mathcal{X}$ and transition matrix $P$.
Given a positive integer $m \geq 1$, let $\mathcal{X}^m_{+} \subseteq \mathcal{X}^m$ be the $m$-tuples $(x_1, \dotsc, x_m)$ such that $P_{x_i x_{i+1}} > 0$ for each $1 \leq i < m$.
Define the random variables $\{ \smash{X^{(m)}_i} \}_{i \geq 1}$, where $\smash{X^{(m)}_i} = (X_i, X_{i+1}, \dotsc, X_{i+m-1})$ consists of $m$ consecutive states of the Markov chain $\{X_i\}_{i \geq 1}$.
Then $\{ \smash{X^{(m)}_i} \}_{i \geq 1}$ is in fact a Markov chain with state space $\mathcal{X}^m_{+}$, where the probability of transition between two `shifted' tuples $(x_1, x_2, \dotsc, x_m)$ and $(x_2, x_3, \dotsc, x_{m+1})$ is equal to $P_{x_m x_{m+1}}$, and zero otherwise.
We say that $\{ \smash{X^{(m)}_i} \}_{i \geq 1}$ is the \emph{$m$-block chain of $\{X_i\}_{i \geq 1}$}.
The following proposition is straightforward to prove.

\begin{lemma} \label{proposition:multiblockmarkov}
	Let $X = \{X_i\}_{i \geq 1}$ be a Markov chain with state space $\mathcal{X}$ and transition matrix $P$.
	Let $m \geq 1$, and let $X^{(m)} = \{ \smash{X^{(m)}_i} \}_{i \geq 1}$ be the $m$-block chain of $X$.
	If $X$ is ergodic, then $X^{(m)}$ is ergodic as well;
	and if the stationary distribution of $X$ is $\pi$, then $X^{(m)}$ has stationary distribution $\pi^{(m)}$ on $\mathcal{X}^m_{+}$ given by \[ \pi^{(m)}(x_1, x_2, \dotsc, x_m) = \pi(x_1) P_{x_1 x_2} \dotsb P_{x_{m-1} x_m} \]
	for each $(x_1, x_2, \dotsc, x_m) \in \mathcal{X}^m_{+}$.
\end{lemma}

\subsection{Random allocation}\label{sec:allocation-random}

To prove \cref{pro:allocation-random}, we consider (a slight variation of) the following randomised procedure, which defines a random homomorphism from the $(t-k+1)$st power of a path $P$ to the $(t-k+1)$st power of a cycle $C$.
Suppose the vertices of $P$ are $\{p_1, p_2, \dotsc\}$ and the vertices of $C$ are $\{c_1, \dotsc, c_s\}$, in cyclic order.
Initially, map $p_i$ to $c_i$ for all $1 \leq i < t$.
In an iteration of the procedure, suppose that the next vertex of $P$ to be embedded is $p_i$, and the $t-1$ previous vertices $p_{i-t+1}, \dotsc, p_{i-1}$ were mapped to $c_{j-t+1}, \dotsc, c_{j-1}$, respectively.
(Here and in the following, index computations are taken modulo $s$ for the vertices of $C$.)
Flip a fair coin.
\begin{enumerate}
	\item If the result is heads, then map $p_i$ to $c_j$,
	\item If the result is tails, then map the $t-1$ vertices $p_i, p_{i+1}, \dotsc, p_{i+t-2}$ to $c_{j-t}, c_{j-t+1}, \dotsc, c_{j-2}$ respectively.
\end{enumerate}
This procedure ensures that each $t$ consecutive vertices in $P$ are embedded in $t$ consecutive vertices in $C$, which in turn ensures that it defines a valid homomorphism from $P$ to $C$.

To analyse this random algorithm, we use a Markov chain $X = \{(a_i, b_i)\}_{i \geq 0}$, defined as follows.
The state space is $\mathcal{X}_{s, t} = [s] \times [t-1]$.
The transition probabilities $P_{(a,b)(a',b')}$ are given (for each $a \in [s]$) by
$P_{(a,1), (a+1,1)} = P_{(a,1), (a,2)} = 1/2$,
$P_{(a,b)(a,b+1)} = 1$ for $b \in \{2, \dotsc, t-2\}$,
$P_{(a,t-1), (a-1, 1)} = 1$, and $0$ in any other case.
The way these transitions capture the above algorithm is that the allocation uses the vertex $c_j$ of the cycle in the $j$th step if $(a_i, b_i) = (i,1)$ or if $(a_i, b_i) = (j+t+1-b, b)$ for some $1 < b < t$.
It is not hard to compute the stationary distribution of this Markov chain.

\begin{lemma}[{\cite[Lemma 10.4]{LS23}}] \label{proposition:markovstationary}
	If $s$ and $t$ are coprime, then $X$ is ergodic and the stationary distribution~$\pi$ of $X$ is given, for all $a \in [s]$, by $\pi_{(a,1)} = 2/(st)$ and $\pi_{(a,b)} = 1/(st)$ for $b \neq 1$.
\end{lemma}

Armed with these facts, we can prove \Cref{pro:allocation-random}.

\begin{proof}[Proof of \cref{pro:allocation-random}]
	Introduce $C$ with $1/t,1/s \gg C \gg \rho$ such that \cref{theorem:markovconcentration} applies to chains over $[s] \times [t-1]$ and $(2t-1)$-block Markov chains over $[s] \times [t-1]$.
	Let us write $H = P(\cU)$.
	So $P$ is the $(t-k+1)$st power of a \tight path on $q$ vertices $p_1, \dotsc, p_q$, and $\cU = \{U_i\}_{p_i \in V(P)}$ is a set family with clusters of size at most $h \leq \rho n$.

	Now we consider a random embedding $\phi_2 \colon V(P) \rightarrow V(C)$ defined as follows.
	Let $X = \{(a_i, b_i)\}_{i \geq 1}$ be the Markov chain over $[s] \times [t-1]$ defined above, where $(a_1, b_1)$ is drawn according to the stationary distribution~$\pi$ given by \cref{proposition:markovstationary}.
	Then, for each $i \geq 1$, we define
	\[\phi(v_i) = \begin{cases}
			c_a         & \text{if } (a_i, b_i) = (a, 1)                        \\
			c_{a-t+b-1} & \text{if } (a_i, b_i) = (a, b) \text{ and } b \geq 2,
		\end{cases}\]
	where indices of $c_j$ are understood modulo $s$ as before.
	Then $\phi_2$ is a homomorphism from $P$ to $C$.
	Let $\phi = \phi_2 \circ \phi_1$ be the homomorphism from $H$ to $C$, where $\phi_1$ maps the vertices $U_i$ to $p_i$.
	To finish, we verify that $\phi$ satisfies the required properties with positive probability.
	The analysis is separated into two steps concerning the allocation to the clusters and the allocation of the buffer, respectively.

	\medskip
	\noindent \emph{Step 1: Estimating allocation to clusters.}
	For each $1 \leq j \leq s$, let
	\[ S_j = \{ (j, 1) \} \cup \{ (j + t + 1 - b, b) \colon 2 \leq b < t \} \subseteq \mathcal{X}_{s,t}, \]
	and note that $\phi(v_i) = c_j$ if and only if $(a_i, b_i) \in S_j$, by definition.
	Now, fix $1 \leq j \leq s$, and define functions $f^j_1, \dotsc, f^j_q\colon \mathcal{X}_{s,t} \rightarrow [0, h]$ such that, for all $1 \leq i \leq q$,
	\[f^j_i((a,b)) = \begin{cases}
			|U_i| & \text{if } (a, b) \in S_j \\
			0     & \text{otherwise.}
		\end{cases}\]
	Since $U_i$ will be embedded via $\phi = \phi_2 \circ \phi_1$ in $c_j$ if and only if $(a_i, b_i) \in S_j$, we have
	\[ | \phi^{-1}(c_j)| = \sum_{i=1}^q f^j_i((a_i, b_i)). \]
	As the Markov chain was started from its stationary distribution, we have $\Pr[(a_i, b_i) = (a,b)] = \pi_{(a,b)}$ for each $1 \leq i \leq q$ and $(a,b) \in \cX_{s,t}$.
	Then, using the values of $\pi$ given by \cref{proposition:markovstationary}, we get
	\[ \Pr[ \phi_2(v_i) = c_j ] = \Pr[ (a_i, b_i) \in S_j ] = \frac{2}{st} + (t-2) \frac{1}{st} = \frac{1}{s}, \]
	which implies that $\Exp_\pi[ f^j_i ] = \frac{|U_i|}{s}$ and therefore that
	\[ \sum_{i=1}^q \Exp_\pi[ f^j_i ] = \frac{n}{s}. \]
	We have that $\sum_{i=1}^q |U_i| = n$ and $|U_i| \leq h$ for each $1 \leq i \leq q$.
	A convexity argument reveals that the sum $\sum_{i=1}^q |U_i|^2$ is maximised when as many entries $|U_i|$ as possible take the value $h$, at most one takes a value between $1$ and $h$, and the remaining values are zero.
	In this situation, at most $(n-1)/h + 1 \leq 2n/h$ entries take value $h$.
	Therefore, we have
	\[ \sum_{i=1}^q |U_i|^2 \leq (2n/h)h^2 = 2nh. \]

	The chain $X$ is finite and ergodic by \cref{proposition:markovstationary}.
	We apply \cref{theorem:markovconcentration} with $f^j_i, \xi n / s$ and $[0,|U_i|]$ playing the rôles of $f_i$, $\eps$ and $[x_i,y_i]$ respectively, to get
	\begin{align*}
		\Pr\left[ \left| |\phi^{-1}(c_j)| - \frac{n}{s} \right| > \frac{\xi n}{s} \right]
		 &  =
		\Pr\left[ \left| \sum_{i=1}^q f^j_i((a_i, b_i)) - \Exp_\pi[f^j_i] \right| > \eps \frac{n}{s} \right] \\
		  & \leq 2 \exp \left( - \frac{2 C \xi^2 n^2}{s^2 \sum_{i=1}^q |U_i|^2} \right)                       
		  \leq 2 \exp \left( - \frac{C \xi^2 n}{s^2 h} \right) < \frac{1}{3s},
	\end{align*}
	where in the last step we used $h \leq \rho n < C \xi^2 n / (s^2   \ln(8s))$.
	This allows us to use a union bound over the $s$ choices of $j$.
	So with probability at least $2/3$, we have $|\phi^{-1}(c_j)| =  (1 \pm \xi){n}/{s}$ for each $1 \leq j \leq s$.

	\medskip
	\noindent \emph{Step 2: Estimating buffer allocation.}
	For each $1 \leq c \leq t$, let $H_c \subseteq V(H)$ be the vertices contained in a cluster~$U_i$ with $i \equiv c \bmod t$.
	Thus $n = \sum_{c=1}^t |H_c|$, and therefore by the pigeonhole principle there exists $c$ such that
	\begin{equation}
		|H_c| \geq n/t. \label{equation:hc}
	\end{equation}
	Fix such a value of $c$ from now on.
	Note that there is no pair of vertices from $H_c$ which belong to a common edge of~$H$.
	Our goal is to find an $(\alpha, s)$-buffer $I$ for $\phi$ contained in $H_c$.

	Given $1 \leq j \leq s$, we say that $v \in V(H)$ is a \emph{buffer vertex for $c_j$} if there exists $1 \leq i \leq q$ with $i \equiv c \bmod t$ such that $\phi_1(v) = p_i$ and $\phi_2(p_{i} + r ) = c_{j+|r|}$ holds for each $-t < r < t$.
	The reason for this definition is that for each buffer vertex $v$ for $c_j$, it follows that each $v \in H_c$, $\phi(v_i) \in c_j$, and for each edge $e \in H$ which contains~$v$, we have $\phi(e) \subseteq \{ c_{j+1}, \dotsc, c_{j+t-1} \}$.
	Thus, we need to show that for each $1 \leq j \leq s$, there are at least $\alpha n / s$ buffer vertices for $c_j$.

	Consider the $(2t-1)$-block chain $X^{(2t-1)} = \{ \smash{X^{(2t-1)}_i} \}_{i \geq 1}$ of the chain $\{X_i\}_{i \geq 1}$ over the state space $\mathcal{X}^{2t-1}_+$.
	By \cref{proposition:multiblockmarkov}, $X^{(2t-1)}$ is ergodic and has stationary distribution $\pi^{(2t-1)}$ on $\mathcal{X}^{2t-1}_{+}$ given by \[ \pi^{(2t-1)}(x_1, x_2, \dotsc, x_{2t-1}) = \pi(x_1) P_{x_1 x_2} \dotsb P_{x_{2t-2} x_{2t-1}} \]
	for each $(x_1, x_2, \dotsc, x_{2t-1}) \in \mathcal{X}^{2t-1}_{+}$.
	Since the chain $X$ is started from its stationary distribution, we have
	\begin{align*}
		\Pr[(X_1, \dotsc, X_{2t-1}) = (x_1, \dotsc, x_{2t-1})]
		 & = \pi(x_1) P_{x_1 x_2} \dotsb P_{x_{2t-2} x_{2t-1}}
		= \pi^{(2t-1)}(x_1, \dotsc, x_{2t-1}),
	\end{align*}
	so $X^{(2t-1)}$ also starts from its stationary distribution.

	Next, given $1 \leq j \leq s$, define $\mathbf{x}_j \in \mathcal{X}^{2t-1}_{+}$ as
	\begin{align*}
		\mathbf{x}_j & = ( (j+1, 1), (j+2,1), \dotsc, (j+t-1, 1),            \\
		             & \qquad  (j+t-1, 2), (j+t-1, 3), \dotsc, (j+t-1, t-1), \\
		             & \qquad  (j+t-2, 1), (j+t-1, 1) ).
	\end{align*}
	Note that if $X^{(2t-1)}_{i-t+1} = \mathbf{x}_j$, then $\phi_2(p_i) = c_j$, and also the $t-1$ vertices before $p_i$, and also the $t-1$ vertices after $p_i$ were embedded in $\{c_{j+1}, \dotsc, c_{j + t-1} \}$.
	Recall that $\phi = \phi_2 \circ \phi_1$ is the homomorphism from $H$ to $C$, where $\phi_1$ maps the vertices $U_i$ to $p_i$.
	The above discussion implies that every vertex in $U_i \cap H_c$ is a buffer vertex for $c_j$.
	Shifting the indices, it follows that if $X^{(2t-1)}_i = \mathbf{x}_j$, then the vertices in $U_{i+t-1} \cap H_c$ are buffer vertices for $c_j$.

	With this in mind, define, for each $1 \leq j \leq s$, a sequence of functions $g^j_1, \dotsc, g^j_{q-2t+2}\colon \mathcal{X}^{(2t-1)}_{+} \rightarrow [0, h]$ such that, for all $1 \leq i \leq q$,
	\[g^j_i(\mathbf{y}) = \begin{cases}
			|U_{i+t-1} \cap H_c| & \text{if } \mathbf{y} = \mathbf{x}_j \\
			0                    & \text{otherwise.}
		\end{cases}\]
	By the above discussion, $c_j$ has at least
	$  \sum_{i=1}^{q-2t+2} g^j_i(X^{(2t-1)}_i) $
	many buffer vertices.

	Recall \cref{proposition:markovstationary}.
	Since the Markov chain was started from its stationary distribution, it follows that
	\begin{align*}
		\Pr[X^{(2t-1)}_i = \mathbf{x}_j]
		 & = \pi^{(2t-1)}_{\mathbf{x}_j}
		= \pi_{(j+1,1)} P_{(j+1,1)(j+2,2)} \dotsb P_{(j+t-2,1),(j+t-1,1)} = \frac{2}{st} \left( \frac{1}{2} \right)^{t}
	\end{align*}
	for each $1 \leq i \leq q-2t+2$.
	This implies $\Exp_\pi[ g^j_i ] = \frac{|U_{i+t-1} \cap H_c|}{2^{t-1}st}$.
	We therefore have (with explanations to follow)
	\begin{align*}
		\sum_{i=1}^{q-2t+2} \Exp_\pi[ g^j_i ]
		 & = \frac{1}{2^{t-1}st} \left| H_c \setminus \{ U_1 \cup \dotsb \cup U_{t-2} \cup U_{q-t+1} \cup \dotsc \cup U_{q} \} \right|            \\
		 & \geq \frac{1}{2^{t-1}st} \left( |H_c| - 2h \right) \geq \frac{1}{2^{t-1}st} \left( \frac{n}{t} - 2h \right) \geq \frac{2 \alpha n}{s}.
	\end{align*}
	In the first equality, we use that $|U_{i+t-1} \cap H_c|$ for $1 \leq i \leq q-2t+2$ ranges from every $U_i$ except the first $t-2$ and the last $t-2$ indices.
	The first inequality follows from the fact that $H_c$ (since it consists only of clusters $U_i$ which are spaced $t$ apart) contains at most one of the first and last $t-2$ many $U_i$'s; and each $U_i$ that is missed contains at most $h$ many vertices, by assumption.
	The second inequality follows from equation \eqref{equation:hc}, and the last inequality follows from the assumptions on $h$ and $\alpha$.
	Also, the same convexity argument as in the previous step shows that
	\[ \sum_{i=1}^{q-2t+2} |U_{i+t-1} \cap H_c|^2 \leq 2nh. \]

	We apply \cref{theorem:markovconcentration} to the chain $X^{(2t-1)}$ with $g^j_i, \alpha n / s$, $0$ and $|U_{i+t-1} \cap H_c|$ playing the rôles of $f_i$,~$\eps$,~$x_i$ and $y_i$ respectively, to get
	\begin{align*}
		\Pr\left[ \sum_{i=1}^{q-2t+2} g^j_i(X^{(2t-1)}_i) < \frac{\alpha n}{s} \right]
		 & \leq
		\Pr\left[ \left| \sum_{i=1}^{q-2t+2} g^j_i(X^{(2t-1)}_i) - \Exp_\pi[ g^j_i] \right| > \alpha \frac{n}{s} \right] \\
		 & \leq 2 \exp \left( - \frac{2 C \alpha^2 n^2}{s^2 \sum_{i=1}^s |U_{i+t-1} \cap H_c|^2} \right)                 
		  \leq 2 \exp \left( - \frac{C \alpha^2 n}{2s^2h} \right) < \frac{1}{3s},
	\end{align*}
	where in the last step we used $h \leq \rho n < C \alpha^2 n / (s^2  \ln(8s))$.
	Thus we can take a  union bound over all the $s$ choices of $j$.
	So with probability at least $2/3$, it follows that $\sum_{i=1}^{q-2t+2} g^j_i(X^{(2t-1)}_i) \geq \alpha n /s$ for each $1 \leq j \leq s$;
	which implies that each $c_j$ has at least $\alpha n / s$ many buffer vertices, as required.

	Therefore, the calculations done in both steps show that there exists an embedding $\phi_2$ from $P$ to $C$ which satisfies all properties simultaneously.
\end{proof}

\section{Boosting and orientations}\label{sec:boosting+orientations}

Here we show \cref{pro:framework-undirected-to-directed,lem:booster}.
We begin with the proof of the latter lemma in \cref{ssec:boosting}.
For the proof of \cref{pro:framework-undirected-to-directed}, we need some preparations, which take place in the following two subsections.
\cref{pro:framework-undirected-to-directed} is then shown in \cref{sec:framework-undirected-to-directed}.

\subsection{Proof of the Booster Lemma} \label{ssec:boosting}
We require the following fact.

\begin{observation}\label{obs:fractional-threshold}
	Let $1/r,\,1/s,\,\eps \gg 1/n$.
	Let $P$ be an $n$-vertex $s$-graph with $\delta_r(P) \geq (\delta_{1}^{(s)} + \eps) \binom{n-r}{s-r}$.
	Then~$P$ has a perfect fractional matching.
\end{observation}
\begin{proof}
	By monotonicity of the degree-types (\cref{fact:monotone-degrees}), we have $\delta_1(P) \geq (\delta_{1}^{(s)} + \eps) \binom{n-1}{s-1}$.
	Write $n = m k + q$ with $0 \leq q \leq k-1$.
	By the choice of $n$, we note that $P$ still has a perfect matching after deleting any set of $q$ vertices.
	We may thus obtain a perfect fractional matching of $P$ by averaging over all possible choices of $q$ vertices.
\end{proof}

\begin{proof}[Proof of \cref{lem:booster}]
	We abbreviate $\P= \dcon \cap \dspa \cap \dape$.
	Let $G$ be a $[k]$-digraph on $n$ vertices that $(\delta_1,r_1,s_1)$-robustly satisfies $\P$.
	Set $P = \PG{H}{\P}{s_1}$ and $\delta=\delta_{1}^{(s_1)}(\mat)$.
	Hence, $\delta_{r_1}(P) \geq (\delta+\eps) \binom{n-r_1}{s_1-r_1}$.
	Now set $Q = \PG{P}{\DegF{1}{\delta+\eps/2}}{s_2}$.
	In other words, $Q$ is the $s_2$-graph on $V(P)=V(G)$ with an $s_2$-edge $S$ whenever the $s_1$-graph $P[S]$ has minimum ${r_1}$-degree at least $\left(\delta+\eps/2\right) \binom{s_2-r_1}{{s_1-r_1}}$.
	By \cref{lem:inheritance-minimum-degree}, it follows that $\delta_{r_2}(Q) \geq  \delta_2 \tbinom{n-r_2}{s_2-r_2} \geq (1-1/s_2^2) \tbinom{n-r_2}{s_2-r_2}$.

	Fix an arbitrary $s_2$-edge $S' \in E(Q)$.
	It is enough to show that $G[S']$ satisfies $\Del_q(\P)$.
	To this end, let $D \subset S'$ be a set of at most $q$ vertices, and let $S = S' \sm D$.
	Our goal is now to show that $G[S]$ satisfies $\P = \dcon \cap \dspa \cap \dape$.

	For the space property, observe that $P[S]$ has a perfect fractional matching by \cref{obs:fractional-threshold}.
	Furthermore, note that $G[R]$ has a perfect fractional matching for every $s_1$-edge $R$ in $P[S]$ by definition of $P$.
	We may therefore linearly combine these matchings to a perfect fractional  matching of $G[S]$.

	For the connectivity property, consider disjoint $e,f \in G[S]^{(k-1)}$.
	We have to show that $e$ and $f$ are on a common closed \tight walk in $G[S]$.
	Note that $e$ and $f$ together span ${r_1}=2k-2$ vertices, and $\delta_{r_1}(P[S]) \geq (\delta+\eps) \binom{n-r_1}{s_1-r_1}$.
	Hence there is an $s_1$-edge $R \in P[S]$ that contains both $e$ and $f$.
	Moreover, since $G[R] \in \dcon$, it follows that $e$ and $f$ are on a common closed walk.

	Finally, the aperiodicity property follows simply because there is an edge in $R \in P[S]$.
	Since $G[R]$ satisfies $\dape$, it follows that $G[R]$ contains a closed \tight walk of order $1\bmod k$.
\end{proof}

\subsection{Orientations}
It is convenient to work with undirected versions of the properties $\dcon$, $\dspa$ and $\dape$ introduced in \cref{sec:directed-setupt-necessary-conditions}.
A \emph{$k$-bounded hypergraph} (or \emph{$[k]$-graph} for short) $G$ consists of a set of vertices $V(G)$ and a set of edges $E(G)$, where each edge $e \subset V(G)$ has size $1 \leq |e| \leq k$.
Denote by $G^{(i)}$ the $i$-graph consisting of the edges of uniformity $i$.
We say that $G$ is \emph{\tightly connected} if $G^{(k)}$ is \tightly connected (as defined in \cref{sec:necessary-conditions}).

Given this, the concepts introduced in \cref{sec:directed-setupt-necessary-conditions} can be immediately transferred to the undirected setting.
The \emph{(undirected) \tight adherence} $\adh(G) \subset G^{(k)}$ is obtained by taking the union of the \tight components $\tc(e)$ over all $e \in G^{(k-1)}$.

\begin{definition}[Connectivity]
	Let $\ucon$ be the set of $[k]$-graphs $G$ with $k\geq2$ such that $\adh(G)$ is a single vertex-spanning \tight component.
\end{definition}

\begin{definition}[Space]
	Let $\uspa$ be the set of $[k]$-graphs $G$ with $k\geq 2$ such that $\adh(G)$ has a perfect fractional matching.
\end{definition}

\begin{definition}[Aperiodicity]
	Let $\uape$ be the set of $[k]$-graphs $G$ with $k\geq 2$ such that $\adh(G)$ contains a closed walk $W$ of length $1 \bmod k$.
\end{definition}

The next lemma allows us to transition from the undirected to the directed setting.

\begin{lemma}[Orientation]\label{lem:orientation}
	For $1/k,\, 1/s_1 \gg 1/s_2 \gg 1/n$, let $G$ be a $[k]$-graph on $n \geq n_0$ vertices that $s_1$-robustly satisfies $\ucon \cap \uspa \cap \uape$.
	Then there exists an $n$-vertex $[k]$-digraph $D \subseteq \ori{C}(G) \cup \ori{C}(\partial G)$ which $s_2$-robustly satisfies $\dcon \cap \dspa \cap \dape$.
\end{lemma}

The rest of this subsection is dedicated to the proof of \cref{lem:orientation}.
Let us briefly outline the argument.
Given a $k$-graph $G$ as in the statement, it is not hard to find a suitable orientation of $G[S]$ for any $s_1$-uniform edge $S$ in the property graph with respect to  $\ucon \cap \uspa \cap \uape$.
The main challenge is then to find an orientation that works globally, that is, simultaneously for every choice of $S$.
We achieve this by exploiting a highly connected substructure in an auxiliary graph whose vertices are the oriented $(k-1)$-edges of $G$.

To formalise this discussion, we introduce a few tools on ($2$-uniform) graphs.
Say that an $n$-vertex graph is \emph{$(\rho, L)$-robustly-connected} if for every pair of vertices $x, y$ there exists $1 \leq \ell \leq L$ such that there are at least $(\rho n)^\ell$ many $(x,y)$-paths with exactly $\ell$ inner vertices each.
The following standard observation states that a lower-regular pair (as defined in \cref{sec:almost-blow-up-cover}, for $s=2$) contains a large robustly-connected graph.

\begin{lemma} \label{lemma:lowerregular-robustlyconnected}
	Let $d \gg \eps \gg 1/m$.
	Let $P=(V_1, V_2)$ be an $(\varepsilon, d)$-lower-regular pair in a $2$-graph $G$, with $|V_1| = |V_2| = m$.
	Then there is a $(d/3, 3)$-robustly-connected $P' \subset P$ on at least $m$ vertices.
\end{lemma}

\begin{proof}
	For each $i \in [2]$, let $X_i \subseteq V_i$ be the set of vertices with fewer than $(d-\eps)m$ neighbours in $V_{3-i}$.
	Since $d_P(X_1, V_2) < d-\eps$, we deduce by lower-regularity that $|X_1| < \eps m$; similarly $|X_2| < \eps m$.
	Let $V'_1 = V_1 \setminus X_1$ and $V'_2 = V_2 \setminus X_2$.

	Note that each $v \in V'_1$ has at least $(d - \eps)m$ neighbours in $V_2$, and therefore at least $(d - \eps)m - |X_2| \geq (d - 2 \eps)m$ neighbours in $V'_2$.
	Repeating this argument gives that each $v \in V'_2$ has at least $(d - 2 \eps)m$ neighbours in $V'_1$.

	Now, let $x \in V_1$ and $y \in V_2$.
	Then $U_2 = N(x, V_2) \setminus \{y\}$ has at least $(d - 2 \eps)m - 1 \geq (d - 3 \eps)m \geq \eps m$ vertices, and the same is true for $U_1 = N(y, V_1) \setminus \{x\}$.
	By lower-regularity  there are at least $(d - \eps)|U_1||U_2| \geq (d - 3 \eps)^2 m^2$ edges between $U_1$ and $U_2$, and each of them yields an $(x,y)$-path with exactly two inner vertices.
	If $x, y \in V_1$, we can form $(x,y)$-paths with three inner vertices by first taking any vertex $x' \in N(x, V_2)$ and any $(x',y)$-path which does not intersect $x$, giving at least $(d - 3 \eps)^3 m^3$ many choices.
	If $x, y \in V_2$, the argument is the same.
	Since $|V'_1 \cup V'_2| \leq 2m$, we deduce that $G[V'_1 \cup V'_2]$ is $(\rho, 3)$-robustly-connected with $\rho = (d - 3 \eps)/2 \geq d/3$, as desired.
\end{proof}

We also need a fact about \tight walks in \tightly connected $k$-graphs~\cite[Proposition 7.2]{LS23}.

\begin{lemma} \label{proposition:keyorientation}
	Let $G$ be a \tightly connected $k$-graph, and let $(w_1, \dotsc, w_k)$ be an orientation of an edge in $G$.
	Then there exists a closed \tight walk $W$ which contains $(w_1, \dotsc, w_k)$ as a subwalk and visits every edge of $G$.
\end{lemma}

Now we are ready to show the main result of this subsection.

\begin{proof}[Proof of \cref{lem:orientation}]
	Introduce $\eps$ with $1/k,1/s_2 \gg \eps \gg 1/n$.
	Given $G$ as in the statement, let $P = \PG{G}{\ucon \cap \uspa \cap \uape}{s_1}$ on vertex set $V(G)$.
	By assumption, $P$ has minimum $2k$-degree at least $(1 - 1/s_1^2) \binom{n - 2k}{s_1 - 2k}$.
	We begin with an observation about walks and vertex orderings.
	Say that a pair of $(k-1)$-tuples $(\ori{x}, \ori{y})$ is \emph{consistent} for $S \in \P$,  if there is a closed \tight walk $W$ {of length $1 \bmod k$} in $G[S]$ that contains $\ori{x}$ and $\ori{y}$ as a subwalk.
	\begin{claim}\label{cla:consistent-pair}
		For all $S \in P$ and $x, y \in G^{(k-1)}$ contained in $S$, there are orientations $\ori{x}$ of $x$ and $\ori{y}$ of $y$ such that $(\ori{x}, \ori{y})$ is consistent for $S$.
	\end{claim}
	\begin{proofclaim}
		By $\ucon$ and $\uape$ it follows that $\adh(G[S])$ is a single vertex-spanning \tight component which contains a closed \tight walk $C$ of length $1 \bmod k$.
		In particular, $\tc(x)[S]$ and $\tc(y)[S]$ belong to the same \tight component.
		Let $(c_1, \dotsc, c_k)$ be a subwalk of $C$.
		By \cref{proposition:keyorientation}, there exists a closed \tight walk $W$ which contains $(c_1, \dotsc, c_k)$ as a subwalk and visits every edge of $G[S]$, in particular it must visit $k$-edges which contain $x$ and $y$ respectively.
		By extending $W$ (winding around $x$ and $y$, and then attaching copies of $C$, if necessary), we can assume that~$W$ has length $1 \bmod k$, visits every edge of $G[S]$, and contains as a subsequence some orientation $\ori{x}$ of~$x$ and some orientation $\ori{y}$ of $y$, respectively.
	\end{proofclaim}

	Let $\beta = 1/(2(k-1)!^2)$.
	We say that a pair of $(k-1)$-tuples $(\ori{x}, \ori{y})$ on $t = |x \cup y|$ vertices is \emph{robustly consistent}, if there are at least $\beta\binom{n - t}{s_1 - t}$ many $s_1$-edges $S \in P$ such that $(\ori{x}, \ori{y})$ is consistent for~$S$.

	\begin{claim}\label{cla:robustly consistent}
		For all $x, y \in G^{(k-1)}$, there are orientations $\ori{x}$ of $x$ and $\ori{y}$ of $y$ such that $(\ori{x},\ori{y})$ is robustly consistent.
	\end{claim}
	\begin{proofclaim}
		Let $x, y \in G^{(k-1)}$and $t = |x \cup y|$.
		By \cref{fact:monotone-degrees}, there are at least $(1/2) \binom{n - t}{s_1 - t}$ edges $S \in P$ with $x \cup y \subset S$.
		It follows by \cref{cla:consistent-pair} that for each such $S$, there exists a consistent pair $(\ori{x}, \ori{y})$ obtained by orienting $x$ and $y$, respectively.
		Hence the claim follows by averaging over the possible orientations.
	\end{proofclaim}

	We define an auxiliary $2$-graph $H$ as follows.
	Let $V(H) = \ori{C}(G^{(k-1)})$ be the set of all possible orientations of every edge in $G^{(k-1)}$.
	We add an edge between $\ori{x}$ and $\ori{y}$ to $H$ if $(\ori{x},\ori{y})$ is robustly consistent.
	Let $N = |V(H)|$, and note that $N = (k-1)!|G^{(k-1)}|$.
	We can also obtain the bound $N \geq (k-1)!/(2\binom{s_1}{k-1})\binom{n}{k-1}$.
	This follows from the fact that $P$ has edge density at least $1/2$ and $G^{(k-1)}[S]$ is non-empty for every $s_1$-edge $S \in P$.
	It also is not hard to see that $H$ is dense.

	\begin{claim}\label{cla:many-edges-H}
		There are at least $\binom{|G^{(k-1)}|}{2} \geq \beta \binom{N}{2}$ edges in $H$.
	\end{claim}
	\begin{proofclaim}
		For $x \in G^{(k-1)}$, denote by $U_x$ the set of the $(k-1)!$ orientations of $x$.
		Observe that the vertex set of $H$ can be\ partitioned into the $|G^{(k-1)}|$ many sets $U_x$, one for each $x \in G^{(k-1)}$.
		By \cref{cla:robustly consistent}, there is an edge in $H$ between $U_x$ and $U_y$ for any distinct $x,y \in G^{(k-1)}$.
		We deduce that $H$ has at least $\binom{|G^{(k-1)}|}{2} \geq \beta \binom{N}{2}$ edges.
	\end{proofclaim}

	Next, we show that $H$ contains a robustly-connected subgraph.
	Let $d = 4 \beta / 30$ and $\alpha = \exp( - \eps^{-4})/4$.

	\begin{claim}\label{cla:robustly-connected-subgraph}
		There is a subgraph $H_0 \subseteq H$ on at least $\alpha N$ vertices which is $(d, 3)$-robustly-connected.
	\end{claim}
	\begin{proofclaim}
		Let $H' \subseteq H$ be a bipartite subgraph with two parts of size  $m = \lfloor N/2\rfloor$ with at least $6d m^2$ edges.
		(Finding such an $H'$ is a basic application of the first moment method.)
		Recall that $1/k \gg \eps \gg 1/n$.
		Moreover, $\beta$ and $d$ are functions of $k$ while $N$ grows with $n$ by \cref{cla:many-edges-H}.
		Apply \cref{lem:density-regular-tuple} with $2, \eps, m, H'$ in place of $s, \eps, m, P$ to obtain an $(\eps, 6d)$-lower-regular $4\alpha m$-balanced pair $(X,Y)$.
		Now apply \cref{lemma:lowerregular-robustlyconnected} with $X, Y, 4\alpha m, \eps/2$ playing the rôles of $V_1, V_2, m,\eps$ to obtain a subgraph $H_0 \subseteq H$ on at least $4\alpha m \geq \alpha N$ vertices which is $(d, 3)$-robustly-connected.
	\end{proofclaim}

	We define a $[k]$-bounded orientation $D \subseteq \ori{C}(G)$ on $V(G)$ by setting $D^{(k-1)} = V(H_0)$ and $D^{(k)} = \ori{C}(G^{(k)})$.
	In other words, every possible orientation of every $k$-edge of $G$ is included in $D^{(k)}$, and $D^{(k-1)}$ consists of the orientations which participate in the auxiliary graph $H_0$.
	Our goal is to show that $D$ satisfies $s_2$-robustly $\P = \dcon \cap \dspa \cap \dape$ as detailed in \cref{def:robustenss-detailed}.
	We begin by identifying an auxiliary $s_2$-graph, which captures the relevant properties.

	\begin{claim}\label{cla:orientation-aux-graph}
		There is an $s_2$-graph $Q$ on $V(G)$ with $\delta_{2k}(Q) \geq  {(1-1/s_2^2)}  \tbinom{n-2k}{s_2-2k}$ such that for any edge $R \in Q$ it follows that
		\begin{enumerate}[\upshape{(\roman*)}]
			\item \label{itm:aux-graph-deg} $P[R]$ has minimum $2k$-degree at least $\left(\delta_1^{s_1} + \eps \right) \binom{s_2-2k}{s_1-2k}$; and
			\item \label{itm:aux-graph-dense} $D[R]^{(k-1)}$ is not empty.
		\end{enumerate}
		Furthermore, for any two $x,y \in D[R]^{(k-1)}$, we have that
		\begin{enumerate}[\upshape{(\roman*)}, resume]
			\item \label{itm:aux-graph-constitent} if $xy \in E(H_0)$, then there is an $S \in P[R]$ that is consistent with $(x,y)$; and
			\item \label{itm:aux-graph-suitable} there is $1 \leq \ell \leq L$ and $z_1,\dots,z_\ell \in D^{(k-1)}[R]$ such that $H_0$ contains the path $x z_1 \dotsb z_\ell y$.
		\end{enumerate}
	\end{claim}
	\begin{proofclaim}
		We define four types of auxiliary hypergraphs that track the relevant objects.
		To begin, recall that $P = \PG{G}{\ucon \cap \uspa \cap \uape}{s_1}$ has minimum $2k$-degree at least $(1 - 1/s_1^2) \tbinom{n - 2k}{s_1 - 2k}$.
		Moreover, $1 - 1/s_1^2 \geq \th_1^{(s_1)} +\eps$ by \cref{fact:matchingthresholds}.
		Set $Q_1 = \PG{P}{\DegF{2k}{\th_1^{(s_1)}+\eps}}{s_2}$.
		In other words, $Q_1$ is the $s_2$-graph with vertex set $V(G)$ and an $s_2$-edge $R$ whenever the induced $s_1$-graph $P[R]$ has minimum vertex-degree at least $\left(\th_1^{(s_1)}+\eps\right) \tbinom{s_2-1}{s_1-1}$.
		By \cref{lem:inheritance-minimum-degree}, it follows that $\delta_{2k}(Q_1) \geq (1-e^{-\sqrt{s_2}})  \tbinom{n-2k}{s_2-2k} \geq (1 - 1/(4s^2_2)) \tbinom{n-2k}{s_2-2k}$.

		Next, define $Q_2$ as the $s_2$-graph on vertex set $V(G)$ with an edge $R$ whenever $D^{(k-1)}[R]$ is non-empty.
		Since \cref{cla:robustly-connected-subgraph} guarantees that $|D^{(k-1)}| \geq \alpha N = \Omega(n^{k-1})$, it follows by \cref{lem:inheritance-minimum-degree} that $\delta_{2k}(Q_2) \geq (1-e^{-\sqrt{s_2}/2}) \tbinom{n-2k}{s_2-2k} \geq (1 - 1/(4s^2_2)) \tbinom{n-2k}{s_2-2k}$ edges.

		To define the graph $Q_3$ which tracks \ref{itm:aux-graph-constitent}, we proceed as follows.
		First, for each integer $t$ between $k-1$ and $2k-2$, we let $\mathcal{D}_t$ be the set of pairs of tuples $(x,y) \in V(H_0)^2$ such that $|x \cup y| = t$.
		We also let $\mathcal{D}'_{t} \subseteq \mathcal{D}_t$ be the set of pairs which form an edge in $E(H_0)$.
		By definition, we know that for each pair $(x,y) \in \mathcal{D}'_{t}$ there is a set $P_{x,y} \subseteq P$ consisting of at least $\beta \tbinom{n-t}{s_1 - t}$ many $s_1$-edges $S \in P$ such that $(x,y)$ is consistent for $S$.
		Let $Q_3^t$ be the $s_2$-graph consisting of the $s_2$-sets $S$ such that for all $(x,y) \in \mathcal{D}'_t$ with $x \cup y \subseteq S$, there is at least one $s_1$-set $S' \in P_{x,y}$ such that $S' \subseteq S$.
		Then \cref{lem:inheritance-minimum-degree2}
		implies that $Q_3^t$ satisfies $\delta_{2k}(Q^t_3) \geq (1 - e^{\sqrt{s_2}}) \tbinom{n-2k}{s_2 - 2k}$.
		We then let $Q_3$ be the intersection of the $s_2$-graphs $Q_3^t$, for each $k-1 \leq t \leq 2k-2$.
		We have $\delta_{2k}(Q_3) \geq (1 - ke^{\sqrt{s_2}}) \tbinom{n-2k}{s_2 - 2k} \geq (1 - 1/(4s^2_2)) \tbinom{n-2k}{s_2-2k}$.

		To define the $s_2$-graph $Q_4$ which tracks \ref{itm:aux-graph-suitable} we need a bit of preparation.
		Let $(x,y) \in \mathcal{D}_t$.
		Since $x,y \in V(H_0)$, and since $H_0$ is $(d,3)$-robustly connected, there exist $1 \leq \ell \leq 3$ such that there are at least $(\alpha |V(H_0)|)^\ell = \Omega(n^{\ell (k-1)})$ many paths $x z_1 \dotsb z_\ell y$ in $H_0$.
		For each $1 \leq \ell \leq 3$, let $\mathcal{D}^\ell_t \subseteq \mathcal{D}_t$ be the set of pairs where the paths described above have $\ell$ inner vertices.
		For a given $(x,y) \in \mathcal{D}^\ell_t$, we can discount the at most $O(n^{\ell(k-1)-1})$ paths $x z_1 \dotsb z_\ell y$ paths where $x \cup y, z_1,\dots,z_\ell$ are not pairwise disjoint.
		By doing so we get a family of $\ell(k-1)$-sets $P'_{x,y}$, consisting of $\Omega(n^{\ell (k-1)})$ many sets $z_1 \cup \dotsb \cup z_\ell$, with the property that $x z_1 \dotsb z_\ell y$ is a path in $H_0$.
		For each $k-1 \leq t \leq 2k-2$ and $1 \leq \ell \leq 3$, we let $Q_4^{t, \ell}$ be the $s_2$-graph consisting of the $s_2$-sets $S$ such that for all $(x,y) \in \mathcal{D}^\ell_t$, there is at least one $\ell(k-1)$-set in $P'_{x,y}$ which is contained in $S$.
		Then \cref{lem:inheritance-minimum-degree2}
		implies that $Q_4^{t, \ell}$ satisfies $\delta_{2k}(Q^{t, \ell}_4) \geq (1 - e^{\sqrt{s_2}}) \tbinom{n-2k}{s_2 - 2k}$.
		We then let $Q_4$ be the intersection of the $s_2$-graphs $Q_4^{t, \ell}$, for each $k-1 \leq t \leq 2k-2$ and $1 \leq \ell \leq 3$.
		We have $\delta_{2k}(Q_4) \geq (1 - 3ke^{\sqrt{s_2}}) \tbinom{n-2k}{s_2 - 2k} \geq (1 - 1/(4s^2_2)) \tbinom{n-2k}{s_2-2k}$.

		Now let $Q$ be the intersection of the $s_2$-graphs $Q_1$, $Q_2$, $Q_3$ and $Q_4$.
		We have that $\delta_{2k}(Q) \geq (1 - 1/s_2^2) \tbinom{n-2k}{s_2 - 2k}$,
		and by construction, it follows that $Q$ satisfies the claimed properties.
	\end{proofclaim}

	Let $Q$ be as in \cref{cla:orientation-aux-graph}, and fix an edge $R \in Q$.
	To finish, we have to show that $D[R]$ satisfies $\P = \dcon \cap \dspa \cap \dape$.
	Let us begin with the connectivity property.

	\begin{claim}
		$D[R]$ satisfies $\dcon$.
		Moreover, $\adh(D[R])$ contains at least one orientation of every edge of~$G[R]^{(k)}$.
	\end{claim}
	\begin{proofclaim}
		Let $W$ be a closed \tight walk in $D[R]$ which contains some $x \in D^{(k-1)}$ as guaranteed by property~\ref{itm:aux-graph-dense}.
		Suppose $W$ visits a maximum number of $k$-edges in $D[R]$.
		It follows that $W$ contains all edges of $D^{(k-1)}$.
		Indeed, for any $y \in D^{(k-1)}$, it follows by property~\ref{itm:aux-graph-suitable} that there are $z_1, \dotsc,z_\ell \in D^{(k-1)}[R]$ with $1 \leq \ell \leq 3$ such that $H_0$ contains a path $x z_1 \dotsb z_\ell y$.
		So by property~\ref{itm:aux-graph-constitent} and by definition of consistency, we can find a \tight walk from $x$ to $y$ in $D[R]$.
		Hence $W$ must also contain $y$ by \cref{proposition:keyorientation}.
		This shows that $\adh(D[R])$ is \tightly connected, and hence $D[R] \in \con$.

		To finish, we claim that $\adh(D[R])$ contains at least one orientation of every $k$-edge in $G[R]$.
		To this end, consider an edge $e \in G[R]^{(k)}$.
		By property~\ref{itm:aux-graph-deg} and \cref{fact:monotone-degrees}, there is an edge $S\in P[R]$ that contains the vertices of both $e$ and $x$.
		Since $G[R] \in \uape$, and by \cref{proposition:keyorientation}, it follows that $W$ contains an orientation of $e$ as a subwalk.
	\end{proofclaim}

	Next, we check the aperiodicity property.
	Take any $x, y \in D^{(k-1)}$ contained in $R$.
	By definition of $D$, there exists $S \in P$ which is consistent with $(x,y)$.
	Hence $G[S]$ contains a closed walk of length $1 \bmod k$ from $x$ to~$y$.
	Since such a walk belongs to $\adh(D[R])$, it follows that $D[R]$ satisfies $\ape$.

	Finally, the space property $\spa$.
		{By property~\ref{itm:aux-graph-deg} and \cref{obs:fractional-threshold}, $P[R]$ contains a perfect fractional matching.}
	Moreover, each $G^{(k)}[S]$ with $S \in P[R]$ contains a perfect fractional matching.
	We may therefore linearly combine these matchings to a perfect fractional matching in $G^{(k)}[R]$.
	Since $\adh(D[R])$ contains every $k$-edge in $G^{(k)}[R]$, we are done.
\end{proof}

\subsection{Dense \tight components}

For the proof of \cref{pro:framework-undirected-to-directed}, we require two simple lemmata about dense \tight components.

\begin{lemma}\label{lem:dense-component}
	Let $1/k,\, \eps \gg 1/n$ and $\delta \in [0,1]$.
	Then every $k$-graph  $G$ with $e(G)\geq (\delta +\eps) \binom{n}{k}$ contains a \tight component with at least $\delta^k\binom{n}{k}$ edges.
\end{lemma}

\begin{proof}
	Let $C_1,\dots,C_\ell$ be the \tight components of $G$.
	Suppose that the component $C$ maximises $e(C_i) / e(\partial C_i)$ over all $1 \leq i \leq \ell$.
	By \cref{obs:tight-connectivity}, we have $\sum_{i=1}^\ell e(C_i) = e(G)$ and $\sum_{i=1}^\ell e(\partial C_i) = e(\partial G)$.
	It follows that
	\begin{equation*}
		e(G) = \sum_{i=1}^\ell e(C_i) \leq \frac{e(C)}{e(\partial C)} \sum_{i=1}^\ell e(\partial C_i) = \frac{e(C)}{e(\partial C)}  e(\partial G).
	\end{equation*}
	Let us write $e(C) = \nu \binom{n}{k}$ and $e(\partial C) = \nu' \binom{n}{k-1}$.
	By Lovász's formulation of the Kruskal--Katona theorem, we have  $\nu' \geq \nu^{(k-1)/k} - \eps^2$.
	Together with the above, this gives $\nu \geq \nu' (\delta + \eps) \geq \nu^{(k-1)/k} \delta$ and hence $\nu \geq \delta^k$.
\end{proof}

\begin{lemma}\label{lem:minimum-degree-component}
	Let $1/k,\, \eps \gg 1/n$  and $\delta \geq 2^{-1/(k-d)}$ for $1\leq d < k$.
	Let $G$ be a $k$-graph with $\delta_d(G)\geq (\delta +\eps) \binom{n-d}{k-d}$.
	Then $G$ contains a \tight component $C \subset G$ with $\delta_d(C)\geq \delta^{k-d} \binom{n-d}{k-d}$.
\end{lemma}
\begin{proof}
	For each $d$-set $e \subset V(G)$, use \cref{lem:dense-component} to select a \tight component $C_e \subset L_{G}(e)$ of density at least~$\delta^{k-d}$.
	Let $C \subset G$ be the vertex-spanning subgraph with an edge $e \cup f$ for every $d$-set $e$ and $f \in C_e$.
	Note that for every $d$-set $e$ the edges $e \cup f$ with $f \in C_e$ are in a common \tight component of $C$, which we denote by $F_e$.
	Moreover, for $d$-sets $e,e'$ with $|e \cap e'| = d-1$, there is a common $f \in E(C_e) \cap E(C_{e'})$ since $\delta^{k-d} \geq 1/2$.
	Thus $F_{e} = F_{e'}$, and by extension $C$ is \tightly connected.
\end{proof}

\subsection{From undirected to directed frameworks}\label{sec:framework-undirected-to-directed}

Now we are ready to show the main result of this section.

\begin{proof}[Proof of \cref{pro:framework-undirected-to-directed}]
	Introduce $\eps,s_3$ with  $1/k,\,1/s_1 \gg \eps \gg s_2 \gg s_3 \gg 1/n$.
	Let $Q = \PG{G}{\P}{s_1}$, and note that by assumption $\delta_{2k} (Q) \geq (1-1/s_1^2) \binom{n-2k}{s_1-2k}$.
	Denote by $F$ the Hamilton framework of $\P$ that we are guaranteed by the assumption.
	For each $T \in Q$, let $C_T$ be the $[k]$-graph with $C_T^{(k)} = F(G[T])$ and no further edges (for the moment).
	Let $C$ be the union of the $[k]$-graphs $C_T$ over all edges $T \in Q$.
	Note that for each $e \in \partial(C^{(k-1)})$ there is at least one edge $T \in Q$ such that $e \in \partial(C_T)$.
	For our purposes, we require a bit more.
	To this end, let us say that $e \in \partial(C^{(k-1)})$ is \emph{well-connected} if there are at least $\eps \binom{n-(k-1)}{s_1-(k-1)}$ edges $T \in Q$ such that $e \in \partial(C_T)$.
	A crude double-counting argument reveals that there are at least $\eps \binom{n}{k-1}$ well-connected sets in $\partial(C)$.
	Now we define a $k$-bounded hypergraph (which we still call $C$) by adding to $C$ all well-connected $(k-1)$-tuples.
	We claim that $C$ satisfies $s_2$-robustly $\ucon \cap \uspa \cap \uape$.

	To see this, define an $s_2$-graph $P$ on $V(G)$ by adding an edge $S$ whenever
	\begin{enumerate}[label = \rmlabel]
		\item \label{itm:Q[S]-large-degree} $Q[S]$ has minimum $(k-1)$-degree at least $(1-e^{-\sqrt{s_2}})\binom{n-(k-1)}{s_1-(k-1)}$,
		\item \label{itm:Q[S]-well-connected-existence} $S$ contains a well-connected $(k-1)$-set and
		\item \label{itm:Q[S]-well-connected-degree} every well-connected $(k-1)$-set $e \subset S$ is contained in $\partial(C_T)$ for at least $(\eps/2) \binom{s_2-(k-1)}{s_1-(k-1)}$ edges $T \in Q[S]$.
	\end{enumerate}
	As in the proof of \cref{cla:orientation-aux-graph}, it follows by \cref{lem:inheritance-minimum-degree,lem:inheritance-minimum-degree2} that $P$ has minimum $2k$-degree at least $(1-1/s_2^2) \binom{n-2k}{s_2-2k}$.
	(We omit repeating the details.)
	Consider an edge $S \in P$.
	In the following, we show that~$G[S]$ satisfies $\ucon \cap \uspa \cap \uape$.
	This is enough, because we may then conclude the proof by applying \cref{lem:orientation} with $s_2,s_3$ playing the rôle of $s_1,s_2$ to obtain an $n$-vertex $[k]$-digraph $G' \subset \ori{C}(G) \cup \ori{C}(\partial G)$ which $s_3$-robustly satisfies $\dcon \cap \dspa \cap \dape$.

	We begin with connectivity.
	Recall that the {\tight adherence} $\adh(C[S]) \subset (C[S])^{(k)}$ is obtained by taking the union of the \tight components $\tc(e)$ over all $e \in (C[S])^{(k-1)}$.
	We have to show that $\adh(C[S])$ is a single vertex-spanning \tight component.
	To this end, we first identify a \tight component $J \subset C[S]$ and then show that $J$ contains the components $\tc(e)$.
	By choice, each $C_T$ is itself \tightly connected.
	{Moreover, for all $T,T' \in Q$ with $|T \cap T'| = s_1-1$, the edges of $C_{T}^{(k)} = F(G[T])$ and $C_{T'}^{(k)} = F(G[T'])$ are in the same \tight component of $G[S]$ thanks to the consistency condition of \cref{def:hamilton-framework}.}
	{This motivates us to select an $s_2$-vertex \tight component $Q'_S \subset Q[S]$ of minimum $(k-1)$-degree at least $$(1-e^{-\sqrt{s_2}})^k \binom{s_2-(k-1)}{s_1-(k-1)} \geq (1-e^{-\sqrt{s_2}/2}) \binom{s_2-(k-1)}{s_1-(k-1)},$$ which is possible by property~\ref{itm:Q[S]-large-degree} and \cref{lem:minimum-degree-component}.}
	Let $J \subset C[S]$ be formed by the union of the $k$-graphs $C_{T}$ with $T \in Q'_S$.
	As observed above, $J$ is \tightly connected.

	Now consider $e \in (C[S])^{(k-1)}$, which exists by property~\ref{itm:Q[S]-well-connected-existence}.
	To prove that $C$ satisfies $\ucon$, we show that $\tc_{C[S]}(e) = J$.
	By the above it suffices to show that $e$ is contained in an edge $T \in Q'_S$.
	Recall that by definition $e$ is well-connected in $G$, and hence $e$ is contained in $\partial(C_T)$ for at least $(\eps/2) \binom{s_2-(k-1)}{s_1-(k-1)}$ edges $T \in Q[S]$ by property~\ref{itm:Q[S]-well-connected-degree}.
	Since $\eps/2 > e^{-\sqrt{s_2}/2}$, this means that one of these edges is in $Q'_S$, and we are done.
	It follows that $C$ satisfies $\ucon$.

	For the space property, observe that $Q'_S$ has a perfect fractional matching.
	Indeed, by \cref{fact:matchingthresholds,obs:fractional-threshold}, the $s_2$-vertex $s_1$-graph $Q'_S$ still has a perfect fractional matching.
	Furthermore, note that $C[T]$ has a perfect fractional matching for every $s_1$-edge $T$ in $Q'_S$.
	We may therefore linearly combine these matchings to a perfect fractional matching of $C[S]$.

	Finally, the aperiodicity property follows simply because there is an edge in $T \in Q'_S$.
	Since $C_T$ satisfies~$\uape$, it follows that $C_T$ contains an odd walk of order $1 \bmod k$.
\end{proof}

\section{Conclusion}\label{sec:conclusion}

In this paper, we introduced a new approach to construct spanning structures in hypergraphs, which was used to prove a hypergraph bandwidth theorem.
We conclude with a series of remarks and related open questions.

\subsection{Comparison to past work}

In earlier work~\cite[Theorem 3.5]{LS23}, we proved a version of \cref{thm:framework-bandwidth} in the graph setting, where the quantification of bandwidth and maximum degrees coincides with~\rf{thm:bandwidth-graphs-linear}.
The main conceptual difference between the two results is the notion of robustness, which in the earlier instance is defined via edge and vertex deletions.
The results are equivalent, in the sense that one can derive one from the other with different constants using the Regularity Lemma.

\subsection{Quo vadis hypergraph Hamiltonicity?}

\cref{thm:framework-bandwidth} allows us to decompose the problem of finding a \tight Hamilton cycle in a dense host graph into three distinct subproblems, corresponding to connectivity, space and aperiodicity.
Each of these problems invokes new questions, which merit study on their own.

\subsubsection*{Connectivity}

Given a $k$-graph $G$ of edge density $d$, what is the largest number of edges a \tight component is guaranteed to have?
For $k=2$, this is easy to solve.
However, already for $k=3$ not much is known beyond the cases $d=o(1)$~\cite{ABCM17} and $d=5/8$~\cite{LSV24}.
In the latter, one can find a component which has density $1/2$ relative to the order of $G$.
We conjecture that this extends to all odd uniformities.
\begin{conjecture}
	For $k = 2j+1 \geq 3$, let $G$ be an $n$-vertex graph of density $1 - \frac{1}{2^k} \frac{k!}{(j+1)!j!}$.
	Then $G$ contains a \tight component $C$ with $e(G) \geq (1/2-o(1)) \binom{n}{k}$.
\end{conjecture}
Note that this is best possible, which can be seen by taking a partition of $V(G)$ into two parts of equal size and adding all edges, except those which have $j+1$ vertices in the first part.

Other questions in this direction arise from the context of spanning spheres, where we are interested in vertex-spanning \tight components under minimum degree conditions~\cite{GHM+21,ILM+24}.

\subsubsection*{Space}

Let us call a matching in a $k$-graph $G$ \emph{connected} if it is contained in a \tight component.
If $G$ has density $d$, what is the largest connected matching we are guaranteed to find?
Without the connectedness, this is known as Erdős' Matching Conjecture, which predicts two extremal configurations.
However, in the connected setting the problem becomes more delicate and already for $k=3$ more than two extremal configurations emerge.
For additional context, we refer the reader to the work of Lang, Schacht and Volec~\cite{LSV24}.

\subsubsection*{Aperiodicity}

The existence of cycles of specific length in dense hypergraphs has been studied in the Turán-setting.
In the known cases we know the threshold for the existence of sufficiently long cycles;
this was achieved for $k=3$ by
Kam{\v{c}}ev, Letzter and Pokrovskiy~\cite{KLP24}, and for $k=4$ by Sankar~\cite{San24}.
We conjecture the following for larger uniformities.

\begin{conjecture}\label{con:odd-cycles}
	For $k \geq 2$, $\eps >0$ and $n$ large, let $G$ be an $n$-vertex graph of density $1/2+\eps$.
	Then $G$ contains a cycle of length coprime to $k$.
\end{conjecture}

The conjecture holds for all $2 \leq k \leq 4$, by the aforementioned results.
We remark that the conjecture, if true, is best possible when $k$ is even.
To see this consider a partition of $G$ into two parts $X$ and $Y$ of the same size and add all edges that intersect in $X$ with an odd number of vertices.
(The density is easily checked by calculating the codegrees.)
For odd $k$, one can take $X$ and $Y$ of distinct sizes and recurse the construction inside $Y$.
When $k=3$, this is known to lead to the optimal density~\cite{KLP24}.
It is plausible that the same recursive construction is optimal for all odd $k \geq 5$.

\subsubsection*{Natural frameworks}

As discussed in \cref{sec:natural-frameworks}, we believe that Hamiltonicity in hypergraphs is best understood by focusing on the densest components of the link graphs.
We therefore reiterate our main conjecture on the matter:

{\connaturalframework*}

\subsection{Pósa's theorem for triple systems}

An old conjecture of Pósa states that every graph with relative minimum degree at least $2/3$ contains a square of a Hamilton cycle.
This was confirmed for large graphs by Komlós, Sárközy and Szemerédi~\cite{KSS98}.
For $3$-graphs, Bedenknecht and Reiher showed that a relative minimum codegree of slightly above $4/5$ gives a square of a \tight Hamilton cycle, while a lower bound of at least $3/4$ is necessary.
We believe that the latter is the correct threshold.

\begin{conjecture}
	For every $\eps > 0$, there is $n_0$ such that every $3$-graph $G$ of order $n\geq n_0$ with $\delta_2(G) \geq (3/4 + \eps)n$ contains a square of a \tight Hamilton cycle.
\end{conjecture}

A first step towards this conjecture would be to prove that everything is connected.

\begin{conjecture}
	For every $\eps > 0$, there is $n_0$ such that for every $3$-graph $G$ of order $n\geq n_0$ with $\delta_2(G) \geq (3/4 + \eps)n$, the $4$-clique graph $K_4(G)$ is \tightly connected.
\end{conjecture}

\subsection{Directed hypergraphs}

Dirac-type results have also been studied in directed hypergraphs.
\rf{pro:blow-up-cover} also applies to this setting.
However, it is not clear whether one can extend \cref{thm:framework-bandwidth} to a directed version.
One particular problem is that in the proof of \cref{pro:allocation-bandwidth-frontend} one can no longer modify the allocation by flipping vertices between clusters (since their edges would then point in the wrong direction).
On the other hand, if we restrict the output of \cref{thm:framework-bandwidth} to a (directed) Hamilton cycle, the proof should in principle go through as before.
But then again, a direct adaptation of our arguments only applies to shift-closed directed graphs (by the setup explained in \cref{sec:directed-setupt-necessary-conditions}).
This certainly limits the applications, as it does not apply, for instance, to oriented graphs.
It would therefore be appropriate to also generalise \cref{def:hamilton-framework} accordingly, for instance by replacing the perfect matching in the space condition with a perfect tiling into cycles.

\subsection{Other spanning structures}

Tight Hamilton cycles are not the only spanning structures that have been studied in hypergraphs, and our techniques can be used or extended to be applicable in other situations.

\subsubsection*{Loose cycles}

Dirac-type problems have also been studied for $\ell$-cycles in which consecutive edges intersect in $1 \leq \ell \leq k-1$ vertices.
The minimum codegree conditions for loose cycles have been determined by Kühn, Mycroft and Osthus~\cite{KMO10}.
Using the main structural result of this paper (\cref{pro:blow-up-cover}), we reconfigured our framework to address $\ell$-Hamiltonicity in a separate publication~\cite{LS25}.

\subsubsection*{Trees}

A classic result of Komlós, Sárközy and Szemerédi~\cite{KSS95, KSS01b} determines asymptotically tight minimum degree conditions which ensure the existence of spanning trees of logarithmic maximum degree.
Similar results have been obtained in the hypergraph setting by Pavez-Signé, Sanhueza-Matamala and Stein~\cite{PSS23}.
We wonder whether an approach using blow-up covers can also be used to recover and extend these results?

\subsubsection*{Spheres}

Another geometrical interpretation of cycles was studied by Georgakopoulos,  Haslegrave,  Montgomery and Narayanan~\cite{GHM+21}, who determined the minimum codegree threshold for $3$-uniform spheres.
Recently, this was generalised for `supported codegree' to $k$-uniform spheres by Illingworth, Lang, Müyesser, Parczyk and Sgueglia~\cite{ILM+24} using a basic version of \rf{pro:blow-up-cover}.

\subsubsection*{Zero-freeness}

Böttcher, Schacht and Taraz~\cite{BST09} also proved a more general form of the Bandwidth theorem, which allows us to embed graphs that admit so-called zero-free colourings.
Our framework can be adapted to reproduce these results for hypergraphs by adding to \cref{def:hamilton-framework} the condition of $F(G)$ containing a clique of size $k+1$.
Moreover, this condition is easily seen to be necessary, since graphs with zero-free colourings can contain $(k+1)$-cliques and there exist $K_{k+1}$-free Hamilton frameworks.

\subsubsection*{Randomness}

Our main result on Hamilton connectedness (\cref{def:hamilton-connectedness}), which we restate here, can be combined with the work of Joos, Lang and Sanhueza-Matamala~\cite{JLS23} to give extensions to the random robust setting.

\thmframeworkconnectednessrobust*

A more challenging problem concerns the (random) resilience of Hamilton cycles~\cite{APP21,AKL+23}.
One obstacle in this line of research has thus far been the lack of a `robust' form of Hamiltonicity, which can be used as a black box.
We believe that the notion of Hamilton frameworks could be a suitable starting point for further research in this area.

\subsubsection*{Counting}

Combining \cref{thm:framework-connectedness-robust} with our former work~\cite{JLS23} also gives (somewhat) crude estimates on counting (powers of) Hamilton cycles in hypergraphs.
Similar bounds have been obtained by Montgomery and Pavez-Signé~\cite{MPS24} by exploiting the inheritance of minimum degree conditions.
It would be very interesting to understand whether this can be further sharpened as in the work of Cuckler and Kahn~\cite{cuckler2009hamiltonian}.

\subsection{Blow-up covers}

Our approach rests on the idea of covering the vertices of a hypergraph with blow-ups of suitable reduced graphs.
This framework shares a common interface with the setup derived from a combination of the Regularity Lemma and the Blow-up Lemma.
This allows us to tackle embedding problems relying on past insights, while avoiding many of the technical details.
We believe that there are plenty of opportunities to test and extend these techniques.

\subsubsection*{Deregularisation}
When using the Regularity Lemma, the order of the host graphs is a tower function of height growing with the input parameters.
There have been efforts to reduce these requirements by avoiding the Regularity Lemma, see for instance the survey of Simonovits and Szemerédi~\cite{SS19}.
Our work contributes to this line of research, since \rf{pro:blow-up-cover} applies to graphs whose order is a tower function of constant height.
Admittedly, this is still too large to be of practical relevance, but it is somewhat smaller nonetheless.

\subsubsection*{Cluster sizes}

The hypergraph extension of the Erdős--Stone theorem (\cref{thm:erd64}) allows us to find blow-ups of constant size hypergraphs with cluster sizes $\poly (\log n)$.
Random constructions show that this is essentially best possible.
On the other hand, the cluster sizes in \cref{thm:bandwidth-graphs-linear} cannot be larger than $\poly (\log \log n)$.
This bound stems from a somewhat crude application of \cref{thm:erd64} in the proof of \cref{lem:connecting-blow-ups}.
We believe that it can be improved to the bound suggested by the Erdős--Stone theorem.

\begin{conjecture}
	\cref{thm:framework-bandwidth} remains valid for $H$ with clusters of size at most $\poly (\log n)$.
\end{conjecture}

{In an accompanying paper, we show a weaker version of this conjecture in the classic setting of Dirac's theorem~\cite{LS24b}.}

\subsubsection*{Inheritance}\label{sec:conclusion-inheritance}

Our framework can be applied to families that are approximately closed under taking small induced subgraphs.
This inheritance principle applies to minimum degree conditions as detailed in \cref{lem:inheritance-minimum-degree}.
There are many related conditions that have been studied in the context of Hamiltonicity, which exhibit inheritance.
Let us illustrate this with the example of degree sequences.

An old result of Pósa~\cite{Pos62} states that a graph on $n$ vertices contains a Hamilton cycle provided that its degree sequence $d_1\leq \dots \leq d_n$ satisfies $d_i > i $ for all $i < n/2$.
To generalise this to powers of cycles, we define the graph class $\DegSeq_{t,\eps}$ as all $n$-vertex graphs~$G$ with degree sequences $d_1 \leq \dots \leq d_n$ such that $d_i > (t-2)n/t + i + \eps n$ for every $i \leq n/t$.
In an extension of Pósa's theorem it was recently established that large enough graphs in $\DegSeq_{t,\eps}$ contain $(t-1)$st powers of Hamilton cycles~\cite{LS23,ST17}, which is sharp as constructions show.
Using standard probabilistic tools, it is not hard to see that $\DegSeq$ is approximately closed under taking small induced subgraphs.
(The details of the proof are found in \cref{sec:inheritance}.)

\begin{lemma}[Degree sequence inheritance] \label{lem:inheritance-degree-sequence}
	Let $1/t,\eps \gg 1/s \gg 1/n$.
	Suppose $G$ is an $n$-vertex graph in $\DegSeq_{t,\eps}$.
	Then $P=\PG{G}{\DegSeq_{t,\eps/2}}{s}$ satisfies $\delta_{2t}(P) \geq \left(1- e^{-\sqrt{s}} \right)  \tbinom{n-2t}{s-2t}$.
\end{lemma}

The following result was proved in our past work~\cite[Lemma 5.6]{LS23} (connectivity, aperiodicity) and by Treglown~\cite{Tre15} (space).

\begin{theorem}
	\label{thm:power-hamilt-connctedness-deg-seq}
	Given $t\geq 2$ and $\eps >0$, let $n$ be sufficiently large.
	Suppose $G$ is an $n$-vertex graph in $\DegSeq_{t,\eps}$.
	Then $K_t(G)$ is connected, has a perfect matching and contains a cycle of length $t+1$.
\end{theorem}

Since $K_t(G)$ consists of a single component, it is also easy to see that the consistency property of \cref{def:hamilton-framework} holds, which leads to the following.

\begin{corollary}
	Given $t\geq 2$ and $\eps >0$, let $n$ be sufficiently large.
	Then the family of $t$-graphs $K=K_t(G)$ on $n \geq n_0$ vertices with $G \in \DegSeq_{t,\eps}$ admits a Hamilton framework $F$ with $F(K) = K$.
\end{corollary}

We may thus apply \cref{thm:framework-bandwidth,thm:framework-connectedness-robust} to obtain new bandwidth and connectedness results.
The same arguments apply to many other host graph properties considered in the graph setting~\cite{LS23}.

Beyond degree conditions, inheritance is also known to hold for quasi-random and uniformly dense hypergraph families \cite[Lemma 9.1]{Lan23}.
So our work also applies to these situations.

\subsubsection*{Stability}

Our work can be combined with results from property testing to obtain stability results~\cite[Corollary 2.7]{Lan23}.
As an example in this direction, we highlight the recent work by Letzter and Ranganathan~\cite{LR25}.
They used the machinery of `blow-up covers', in combination with a stability analysis, to find exact `supported codegree' conditions ensuring tight Hamiltonicity in uniform hypergraphs.

\subsubsection*{Algorithmic considerations}

The decision problem for (tight) Hamiltonicity is computationally intractable.
So whenever we consider a sufficient condition for Hamilton cycles, it is natural to ask whether it can be verified efficiently.
Given a graph property with an intersecting Hamilton framework, it can be checked in $O(n^s)$ steps whether $s$-robustness holds.
In this sense, the conditions of \cref{thm:framework-bandwidth-directed} are easily checkable.

Beyond the decision problem, we believe that the proof also offers a scheme to efficiently find blow-ups of Hamilton cycles.
However, this is not trivial and requires additional details at multiple stages of the argument.
{Recently, Espuña~\cite{Espuna2026} proved an algorithmic version of Erd\H{o}s' theorem (\Cref{thm:erd64}), which ensures the existence of large blow-ups in dense hypergraphs.
This could be used in a first step towards a constructive version of our main result.}

\subsubsection*{Partite frameworks}

Our proof techniques can easily be adapted to obtain partite versions of \cref{pro:blow-up-cover}.
This could be particularly interesting in the setting of `transversal' substructures, where a similar approach has been used to find perfect tilings~\cite{Lan23}.

\section*{Acknowledgements}

We thank Amedeo Sgueglia for helpful comments on an earlier version of the draft.

Richard Lang was supported by the Marie Skłodowska-Curie Actions (101018431), the Ramón y Cajal programme (RYC2022-038372-I) and by grant PID2023-147202NB-I00 funded
by MICIU/AEI/10.13039/501100011033.
Nicolás Sanhueza-Matamala was supported by ANID-FONDECYT Iniciación Nº11220269 grant and by ANID-FONDECYT Regular Nº1251121 grant.

\bibliographystyle{amsplain}
\bibliography{../bibliography.bib}

\appendix

\section{Inheritance}\label{sec:inheritance}

For a $k$-graph $G$ on $n$ vertices, we denote the \emph{relative degree} of a $d$-set $D \subset V(G)$ by $\rdeg_G(S) = \deg_G(S) / \binom{n-d}{k-d}$.
The next lemma implies \cref{lem:inheritance-minimum-degree} directly, and a straightforward modification of the proof (which we omit) also yields \cref{lem:inheritance-minimum-degree2}.

\begin{lemma}[Degree inheritance] \label{lem:inheritance-degree}
	Let $\eps,1/r \gg 1/s \gg 1/n$.
	Suppose $G$ is an $n$-vertex graph.
	Let $P$ be the $s$-graph on $V(G)$ with an edge $S$ whenever every $d$-set $D \subset S$ satisfies $\rdeg_{G[S]}(D) \geq \rdeg_G(D) - \eps$.
	Then $\delta_{r}(P) \geq \left(1- \exp({-\Omega(s)}) \right)  \tbinom{n-r}{s-r}$.
\end{lemma}

We show \cref{lem:inheritance-degree} using the following corollary of the Azuma--Hoeffding inequality (see Frieze and Pittel~\cite[Appendix B]{FP04}).
\COMMENT{Our proof follows the exposition of Frieze and Pittel~\cite[Appendix B]{FP04}.}

\begin{lemma}\label{lem:concentration}
	Let $V$ be an $n$-set with a function $h$ from the $s$-sets of $V$ to $\REALS$.
	Suppose that there exists $K \geq 0$ such that $|{h(S)-h(S')}| \le K$
	for any $s$-sets $S, S' \subset V$ with $|S \cap S'| = s-1$.
	Let $S \subset V$ be an $s$-set chosen uniformly at random.
	Then, for any $\ell >0$,
	\begin{equation*}
		{\rm Pr}( |h(S) - \Exp [h(S)]| \geq \ell) \leq 2 \exp\left( -\frac{ \ell^2}{2\min\{s,n-s\} K^2}\right).
	\end{equation*}
\end{lemma}

\COMMENT{
	\begin{proof}
		Consider a random permutation $\omega$ of $V$.
		Let $S \subset V$ consist of the first $s$ elements of $\omega$, and set $X(\omega) = h(S)$.
		We define the corresponding Doob martingale as follows.
		For a fixed permutation $(x_1,\dots,x_n)$ of $V$ and $0 \leq i \leq n$, let
		$$X_i(x_1,\dots,x_i) = \Exp(X \mid \omega_j = x_j,~ 1\leq j \leq i).$$
		We claim that
		\begin{equation}\label{itm:AH}
			|X_i(x_1,\dots,x_i) - X_i(x_1,\dots,x_{i-1},x_i')| \leq K
		\end{equation}
		for all $i$-tuples $(x_1,\dots,x_i)$ and $(x_1,\dots,x_{i-1},x_i')$ with (internally) distinct entries.
		Indeed, in this case we can conclude by the Azuma--Hoeffding inequality~\cite[Theorem 2.25]{JLR01}.
		Otherwise we repeat the above argument using the bijection between sets and their complements.
		It remains to show the above claim.
		Consider
		\begin{equation*}
			\Omega_1 = \{\omega \in \Omega \colon \omega_j = x_j,~1\leq j \leq i\}
		\end{equation*}
		and
		\begin{equation*}
			\Omega_1' = \{\omega \in \Omega \colon \omega_j = x_j,~1\leq j \leq i-1,~\omega_i = x_i'\}.
		\end{equation*}
		Define a map $f\colon \Omega_1 \to \Omega_1'$ as follows.
		For $\omega = x_1x_2\ldots x_{i-1}x_iy_{i+1}\ldots y_n$ and $y_j = x'_i$, set
		\[
			f(\omega) = x_1x_2\ldots x_{i-1}x'_iy_{i+1}\ldots y_{j-1}x_iy_{j+1}\ldots y_n.
		\]
		Since $f$ is a bijection, we have
		\[
			|Z_i(x_1, x_2, \ldots, x_i) - Z_i(x_1, x_2, \ldots, x'_i)| = \left|\frac{\sum_{y_{i+1},\ldots,y_n} (Z(\omega) - Z(f(\omega)))}{(N-i)!}\right| \leq K,
		\]
		as desired.
	\end{proof}
}

\begin{proof}[Proof of \cref{lem:inheritance-degree}]
	Set $V = V(G)$.
	Let $R \subset V$ be an $r$-set.
	Let $S \subset V$ be drawn uniformly among all $s$-sets that contain $R$.
	We say that a $d$-set {$D\subset V$ is \emph{bad for $S$}} if $\rdeg_{G[S \cup D]} (D) < \rdeg_G(D) - \eps$.
	Moreover, $S$ is \emph{good} if it contains no $d$-set which is bad for $S$.
	So $P$ is the $s$-graph of good edges.
	It suffices to show that, conditioning on $R \subset S$, a $d$-set $D$ is bad for $S$ with probability $ \exp(-\Omega(s))$.
	Indeed, in this case,
	\begin{align*}
		\Pr(\text{$S$ is good} \mid R \subset S) & \leq \sum_{D \in \binom{V}{d}} \Pr(\text{$D$ is bad for $S$} \mid D \cup R \subset S)              \\
		                                         & \leq  \sum_{D \in \binom{V}{d}} \Pr(\text{$D$ is bad for $S$} \mid R \subset S) \Pr( D \subset  S) \\
		                                         & = \Pr(\text{$D$ is bad for $S$} \mid R \subset S ) \binom{s}{d} = \exp({-\Omega(s)}).
	\end{align*}

	To obtain the above bound, consider a set $D \subset V(G)$ with $\delta = \rdeg_G(D) - \eps$.
	We can assume that $\delta \geq 0$ as there is nothing to show otherwise.
	Let $h(S) = \deg_{G[S \cup D]}(D)$.
	We first compute $\Exp [h(S)]$.
	Denote by $(s)_r = s(s-1)\dots (s-r+1)$ the \emph{falling factorial}.
	For a fixed edge $e \in G$ with $e \cap (R \cup D) = D$, the probability that $e \sm D \subset S$ is
	\begin{equation*}
		\quad\binom{n-|R \cup (e \sm D)|}{s-|R \cup  (e \sm D)|} \binom{n-|R|}{s-|R|}^{-1}
		=     \frac{(s)_{r}}{(n)_{r}}   \binom{n-(k-d)}{s-(k-d)} \left(\frac{(s)_{r}}{(n)_{r}}\binom{n}{s} \right)^{-1}
		= \binom{s}{k-d} \binom{n}{k-d}^{-1}.
	\end{equation*}
	There are at least $(\delta + \eps) \binom{n-d}{k-d} - \binom{r-d}{k-d} \geq (\delta+(7/8)\eps) \binom{n}{k-d}$ edges $e \in G$ with $e \cap (R \cup D) = D$.
	It follows that $\Exp [h(S)] \geq  (\delta + (3/4) \eps)  \binom{s-d}{k-d}$.
	\COMMENT{(Here we used the fact that $\binom{a}{b} \binom{a-b}{c} = \binom{a}{b+c} \binom{b+c}{b}$.)}

	Moreover, for any two $s$-sets $Q,Q' \subset V(G)$  with $|R\cap R'| =s-1$ that both contain $R \cup D$, we have $|h(Q) - h(Q')| \leq \binom{s-|R \cup D|-1}{k-|R \cup D|-1} \leq \frac{k}{s} \binom{s-d}{k-d} K$.
	So by \cref{lem:concentration} applied with $\ell = (\eps/4) \binom{s-d}{k-d}$ and $K = \frac{k}{s} \binom{s-d}{k-d}$, we have
	\begin{align*}
		{\rm Pr}\left (   h(S) <   \delta \binom{s-d}{k-d} \right) & \leq
		{\rm Pr}\left (   \Exp [h(S)] - h(S) \geq  (\eps/4) \binom{s-d}{k-d} \right)                                                            \\
		                                                           & \leq 2 \exp\left( -\frac{(\eps/4)^2}{s-|R \cup D|} 2\frac{s^2}{k^2}\right)
		=  \exp(-\Omega(s)) .\qedhere
	\end{align*}
\end{proof}

\begin{proof}[Proof of \cref{lem:inheritance-degree-sequence}]
	Let $D \subset V(G)$ be a set of $2t$ vertices.
	Let $S \subset V(G)$ be selected uniformly among all $s$-sets that contain $D$.
	By \cref{lem:inheritance-degree}, every vertex $v \in S$ satisfies $\rdeg_{G[S]}(v) \geq \rdeg_G(v) - \eps/8$ with probability at least $1- \exp({-\Omega(s)})$.

	Next, introduce $\eta$ with $\eps \gg \eta \gg 1/s$.
	We partition $V(G)$ into `intervals' $W_1,\dots,W_\ell$ of size $(1 \pm \eta) (\eps/8) n$ such that $\deg(u) \leq \deg(v)$ whenever $u \in W_i$ and $v \in W_j$ with $1 \leq i < j \leq t$.
	It follows by the Chernoff bound that $|S \cap W_i| =  (1 \pm 2\eta) (\eps/8)s$ for every $1 \leq i \leq \ell$ with probability at least $1- \exp(-\Omega(s))$.

	To finish, we show that $G[S] \in \DegSeq_{t,\eps/2}$ if $S$ satisfies the above two properties.
	Indeed, consider a vertex $v \in S \cap W_j$ for some $1 \leq j \leq \ell$.
	So the relative position of $v$ in the degree sequence of $G$ is at least $(1-\eta)(j-1)/\ell$.
	By assumption, we have $$\deg_G(v) > (t-2)n/t + ((1 - \eta) (j-1)/\ell)n + \eps n.$$
	Similarly, the relative position of $v$ in the degree sequence of $G[S]$ is at most
	$(\ell - (1-2\eta)(\ell-j/\ell))/\ell \leq j/\ell+2\eta.$
	We conclude that
	\begin{align*}
		\rdeg_{G[S]}(v) & \geq \rdeg_{G}(v)(n-1) - (7/8) \eps                       \\
		                & \geq (t-2)/t + ((1 - \eta) (j-1)/\ell) + (3/4)\eps        \\
		                & \geq (t-2)s/t + ((1 - \eta) (j/\ell+2\eta)s + (\eps/2) s.
	\end{align*}
	It follows that $G[S] \in  \DegSeq_{t,\eps/2}$.
\end{proof}

\section{Proof of the link graph lemma} \label{sec:cooley-mycroft-lemma}

In this section, we give a proof of Lemma~\ref{lem:cooley-mycroft} following the exposition of Cooley and Mycroft~\cite{CM17}.
We use the following theorem of Erdős and Gallai~\cite{EG}, which gives a sharp bound on the smallest possible size of a maximum matching in a graph of given order and size.

\begin{theorem}[Erdős and Gallai~\cite{EG}]\label{erdosgallai}
	Let $n,s \in \NATS$ with $s \leq n/2$.
	Suppose $G$ is a graph on $n$ vertices and
	\[e(G) > \max \left\{\binom{2s-1}{2}, \binom{n}{2} - \binom{n-s+1}{2} \right\}.\]
	Then $G$ admits a matching with at least $s$ edges.
\end{theorem}

\begin{proof}[Proof of Lemma~\ref{lem:cooley-mycroft}]
	During the proof, we make repeated use of the fact that for $0 < x < 1$ we have $\binom{xt}{2} < x^2\binom{t}{2}$.
	Choose $\gamma$ such that $\varepsilon \gg \gamma \gg 1/n$.
	We will show the following series of claims about $C$ and $C'$, which in particular easily imply all of the results stated in the lemma:
	\begin{enumerate}[(a)]
		\item \label{item:cooleymycroft-a} $C$ spans at least $(2/3 + \sqrt{\varepsilon/2})t$ vertices,
		\item \label{item:cooleymycroft-b} $C$ has at least $(4/9 + \varepsilon) \binom{t}{2}$ edges,
		\item \label{item:cooleymycroft-c} $C$ has a matching of size at least $(2/3 + 2\gamma)t$,
		\item \label{item:cooleymycroft-d} $C$ and $C'$ have an edge in common and
		\item \label{item:cooleymycroft-e} $C$ has a triangle.
	\end{enumerate}

	Let us prove part~\ref{item:cooleymycroft-a}.
	Suppose for a contradiction that every connected component in $L$
	has at most $(2/3+\sqrt{\varepsilon/2})t$ vertices.
	Then we may form disjoint sets $A$ and $B$ such
	that $V(L) = A \cup B$, such that $A$ and~$B$ can each be written as a union of
	connected components of $L$, and such that $|A|, |B| \leq (2/3+\sqrt{\varepsilon/2})t$. We then have
	\[ e(L) \leq \binom{|A|}{2} + \binom{|B|}{2} \leq \binom{\left(\frac{2}{3} + \sqrt{\frac{\varepsilon}{2}} \right)t}{2} + \binom{\left(\frac{1}{3} - \sqrt{\frac{\varepsilon}{2}} \right)t}{2} < \left(\frac{5}{9} + \varepsilon \right) \binom{t}{2},\]
	giving a contradiction.

	Now we prove part~\ref{item:cooleymycroft-b}.
	Indeed, part~\ref{item:cooleymycroft-a} implies that the number of edges of $L$ which are not in $C$ is at most
	\[\binom{t-v(C)}{2} < \binom{\left(\frac{1}{3}-\sqrt{\frac{\varepsilon}{2}}\right)t}{2} < \left(\frac{1}{9}+{\frac{\varepsilon}{2}}\right)\binom{t}{2}.\]

	Now we check that part~\ref{item:cooleymycroft-c} holds.
	Let $x$ be such that $v(C) = (1-x)t$, so $x$ is the proportion of vertices of~$V$
	which are not in $C$. In particular $0 \leq x < 1/3-\sqrt{\varepsilon/2}$ by part~\ref{item:cooleymycroft-a}.
	Observe that at most $\binom{xt}{2} \leq x^2\binom{t}{2}$ edges of $L$ are not in $C$, so $C$ has more than $\left(5/9 - x^2 \right) \binom{t}{2}$ edges.
	It is easily checked that the inequality
	\[\left(\frac{5}{9} - x^2 \right) \binom{t}{2} > \binom{2s -1}{2} = \max \left\{\binom{2 s-1}{2},  \binom{t'}{2} - \binom{t'-s+1}{2} \right\}\]
	holds for $t' =\lceil(1-x)t\rceil$, $s = \lceil(1/3+\gamma)t\rceil$ and any $0 \leq x < 1/3-\sqrt{\varepsilon/2}$, as $1/t \ll \gamma \ll \varepsilon$.
	So by Theorem~\ref{erdosgallai} the component $C$ admits a matching $M$ of size $(2/3+2\gamma)t$ as claimed.

	To see part~\ref{item:cooleymycroft-d}, note that parts~\ref{item:cooleymycroft-a}--\ref{item:cooleymycroft-c} apply also to $C_2$.
	Let us write $C_1 = C$ and $C_2 = C'$.
	Fix $\alpha$ and $\beta$ so that $|V(C_1)| = (1-\alpha)t$ and
	$|V(C_2)| = (1 - \beta) t$. By part~\ref{item:cooleymycroft-a}, $0 \leq \alpha, \beta < 1/3$ and
	$|V(C_1) \cap V(C_2)| \geq (1 - \alpha - \beta) t$. Similarly to before, at most
	$\binom{\alpha t}{2} \leq \alpha^2\binom{t}{2}$ edges of $L_1$ are not in $C_1$, and at most
	$\binom{\beta t}{2} \leq \beta^2\binom{t}{2}$ edges of $L_2$ are not in $C_2$. Now suppose for a
	contradiction that $C_1$ and $C_2$ have no edges in common. Then we have
	\begin{align*}
		\left(\frac{5}{9}-\alpha^2\right)\binom{t}{2} & + \left(\frac{5}{9} - \beta^2\right) \binom{t}{2}
		< e(C_1 \cup C_2)                                                                                                                                        \\
		                                              & \leq \binom{v(C_1)}{2} + \binom{v(C_2)}{2} - \binom{|V(C_1) \cap V(C_2)|}{2}                             \\
		                                              & \leq \binom{(1-\alpha) t}{2} + \binom{(1-\beta) t}{2} - \binom{(1- \alpha - \beta)t}{2}                  \\
		                                              & = \left((1-\alpha)^2 + (1-\beta)^2 - (1-\alpha-\beta)^2\right)\binom{t}{2}                               \\
		                                              & \hspace{1cm} - \tfrac{t}{2}\left(\alpha(1-\alpha)+\beta(1-\beta)-(\alpha+\beta)(1-\alpha - \beta)\right) \\
		                                              & \leq (1-2\alpha\beta) \binom{t}{2}.
	\end{align*}
	So $1/9 < \alpha^2 + \beta^2 - 2 \alpha \beta = (\alpha - \beta)^2$, which implies that $|\alpha - \beta| > 1/3$, contradicting the fact that $0 \leq \alpha, \beta < 1/3$. We deduce that $C_1$ and $C_2$ must have an edge in common, as required.

	We conclude the proof by showing part~\ref{item:cooleymycroft-e}.
	If $e(C) \geq 1/2 \binom{v(C)}{2}$, we are done by Mantel's theorem.
	So assume otherwise.
	But then $$\left(\frac{5}{9} +\eps \right) \binom{n}{2} \leq e(G) \leq e(C) + e(G - V(C)) \leq \frac{1}{2} \binom{v(C)}{2} + \binom{n-v(C)}{2}.$$
	Writing $v(C) = \alpha n$, this results in $5/9 \leq \alpha^2/2 + (1-\alpha)^2$.
	Solving for $\alpha$ gives $\alpha = 2/9 (3 - \sqrt{3}) \approx 0.28$, which contradicts part \ref{item:cooleymycroft-a}.
\end{proof}

\end{document}